\pdfoutput=1
\documentclass[final]{siamart171218}

\newcommand{\benum}{\begin{enumerate}}
\newcommand{\eenum}{\end{enumerate}}

\newcommand{\cdb}[1]{{\color{blue}{#1}}}
\newsavebox\CBox
\newcommand\hcancel[2][0.5pt]{%
  \ifmmode\sbox\CBox{$#2$}\else\sbox\CBox{#2}\fi%
  \makebox[0pt][l]{\usebox\CBox}%
  \rule[0.5\ht\CBox-#1/2]{\wd\CBox}{#1}}
\renewcommand{\hcancel}[1]{\unskip}
\newcommand{\cbc}[1]{#1}
\newcommand{\crso}[1]{{\color{red}{\hcancel{#1}}}}

\usepackage{tikz,amsmath,amssymb, longtable,enumerate, graphicx, todonotes, array,cite}
\usepackage{booktabs}
\usepackage{algorithm,algcompatible}
\renewcommand{\COMMENT}[2][.4\linewidth]{%
  \leavevmode\hfill\makebox[#1][l]{//~#2}}
\newcolumntype{?}[1]{!{\vrule width #1}}
\usepackage{lineno}
\numberwithin{equation}{subsection}
\usepackage{verbatim}
\usepackage[space]{grffile}
\usepackage{diagbox}
\usepackage[title]{appendix}
\graphicspath{{./}}
\usepackage{color}
\usepackage[font=scriptsize]{subcaption}
\usepackage{defs}
\usepackage{fonts}
\usepackage{stmaryrd}
\usepackage{bbm}

\title{Hybrid solver for the radiative transport equation using finite volume and discontinuous Galerkin
\thanks{This manuscript has been supported in part by UT-Battelle, LLC, under Contract No. DE-AC0500OR22725 with the U.S. Department of Energy and by Los Alamos National Laboratory operated by Los Alamos National Security, LLC, for the National Nuclear Security Administration of U.S. Department of Energy  under Contract No. DEAC52-06NA25396. The United States Government retains and the publisher, by accepting the article for publication, acknowledges that the United States Government retains a non-exclusive, paid-up, irrevocable, world-wide license to publish or reproduce the published form of this manuscript, or allow others to do so, for the United States Government purposes. The Department of Energy will provide public access to these results of federally sponsored research in accordance with the DOE Public Access Plan (\texttt{http://energy.gov/downloads/doe-public-access-plan}). 
}}%

\author{Vincent Heningburg\thanks{Mathematics Department, University of Tennessee, Knoxville, TN 37996, USA. (\email{vheningb@vols.utk.edu}).}
\and Cory D. Hauck\thanks{Computational Mathematics Group, Computer Science and Mathematics Division, Oak Ridge National Laboratory, Oak Ridge, TN 37831, USA. (\email{hauckc@ornl.gov}).}}

\begin{document}
\allowdisplaybreaks
\maketitle
\begin{abstract}
We propose a hybrid spatial discretization for the radiative transport equation that combines  a second-order discontinuous Galerkin (DG) method and a second-order finite volume (FV) method.  The strategy relies on a simple operator splitting that has been used previously to combine different angular discretizations.  Unlike standard FV methods with upwind fluxes, the hybrid approach is able to accurately simulate problems in scattering dominated regimes.   However, it requires less memory and yields a faster computational time than a uniform DG discretization.   In addition, the underlying splitting allows naturally for hybridization in both space and angle.  Numerical results are given to demonstrate the efficiency of the hybrid approach in the context of discrete ordinate angular discretizations and Cartesian spatial grids.

\end{abstract}
\tableofcontents

\section{Introduction}\label{sec:Intro}
Radiative transport equations (RTEs) are used to describe the movement of particles --- including neutrons \cite{Lewis,Case-Zweifel-1967}, photons \cite{Mihalis-Mihalis-1999,Pomraning-1973}, neutrinos \cite{Buras, Mezzacappa} and charged particles \cite{ZHENGMING1993673}---through a surrounding material medium.  As they pass through the medium, these particles interact with the material via scattering and emission/absorption processes.  The RTE describes the state of the particle system via a scalar density function defined over a position-momentum phase space.

Discontinuous Galerkin (DG) methods  are commonly used in the spatial discretization of the RTE.  In fact, the original method was formulated in \cite{ReedHill} specifically for this purpose.  In the context of radiation transport, one of the main benefits of DG is its ability to accurately approximate the RTE in scattering dominated regimes, while still using upwind fluxes \cite{Guermond, MorelLarsen, adams}.  However, to do so, the method requires additional unknowns per cell when compared to finite volume (FV) and finite difference approaches that rely on the surrounding stencil to obtain high-order.  The resulting increase in memory is usually deemed worth the additional cost since the method is able to give accurate answers without having to resolve mean-free-path lengths.

In this paper, we propose a hybrid discretization strategy that reduces the memory requirement of the usual DG approach, while still accurately simulating the RTE in scattering dominated regimes (with under-resolved meshes).  The strategy, which has been used in  \cite{Crockatt, Coryhybrid} to combine different angular discretizations, is based on a splitting of the RTE  into two components: one with a scattering source (the collided equation) and one without (the uncollided equation) \cite{Alcouffe}.  This splitting carries with it the flexibility to discretize the phase space in each component equation separately. We choose to discretize the collided equation in space with DG, since it handles scattering well,  and the uncollided equation with FV, since it uses fewer degrees of freedom.  The  result is a hybrid discretization that uses less memory, but still behaves well in scattering dominated regimes. 
For purposes of illustration, we implement the spatial hybrid scheme in combination with a discrete ordinate angular discretization.  In this context we show that, under reasonable assumptions, the spatial hybrid captures numerically the diffusion limit, which is realized asymptotically in the  limit of infinite scattering.  We show how the spatial hybrid can be used in conjunction with existing angular hybrid methods, by combining it with the discrete ordinate hybrid from \cite{Crockatt}.   Finally, we demonstrate that the spatial hybrid is more efficient, in terms of memory usage and time to solution, than the standard DG discretization of the transport equation.

The remainder of the paper is organized as follows. In Section \ref{sec:Background}, we briefly recall the  linear RTE, the standard discrete ordinate method, and the diffusion limit. In Section \ref{sec:Hybrid}, we introduce the hybrid method and discuss its implementation. We also give a formal proof (in the steady-state setting) that shows the method captures the diffusion limit.  In Section \ref{sec:Results}, we investigate the diffusion limit numerically and then demonstrate the efficiency of the hybrid approach on two benchmark problems.   A summary and conclusions are given in Section \ref{sec:Conclude}. In an appendix, we provide details on operation counts and memory usage, as well as some specifics of the finite volume reconstruction.

\section{Background}\label{sec:Background}
\subsection{Linear transport equation}\label{sec:RTE}
Let $X\subset \bbR^3$ be an open domain with Lipschitz continuous boundary $\partial X$; let $\bbS^2$ be the unit sphere in $\bbR^3$; and let $\Gamma:=X\times\sph$. Denote points in $X$ by $r=(x,y,z)$ and points in $\sph$ by $\Omega=(\Omega_x,\Omega_y,\Omega_z)$, and define the following sets:
$$\partial \Gamma=\partial X\times\bbS^2,\quad\partial \Gamma^{\pm}=\{(r,\Omega)\in\partial \Gamma\colon \pm\Omega\cdot n(r)>0\},$$
where $n(r)$ is the unit outward normal vector at points $r\in \partial X$ where the boundary is sufficiently smooth. We consider a scaled, linear radiative transport equation (RTE) \cite{Lewis, Chandra} that models mono-energetic particles moving with unit speed and interacting with a material medium through isotropic scattering and emission/absorption processes.  The complete initial-boundary value problem is
\begin{subequations}
\label{eq:RTE}
\begin{alignat}{3}\label{eq:RTEa}
(\epsilon \p_t + \cL_{\epsilon })\psi(t,r,\Omega)&=(\cS_\epsilon \psi)(t,r)+\epsilon q(t,r,\Omega),\quad&&(t,r,\Omega)\in(0,T)\times\Gamma,\\
\psi(t,r,\Omega)&=\psi_0(r,\Omega),\quad&&(t,r,\Omega)\in \{0\}\times\Gamma,\\
\psi(t,r,\Omega)&=\psi_{\rm{b}}(t,r,\Omega),\quad&&(t,r,\Omega)\in (0,T)\times\partial\Gamma^-,
\end{alignat}
\end{subequations}
where $0<\epsilon < 1$ is a dimensionless parameter,
\begin{align}
&(\cL_{\epsilon }\psi)(t,r,\Omega) 
	= \Omega\cdot\nabla_r\psi(t,r,\Omega)+\frac{\sigma_{\rm{t}}(r)}{\epsilon }\psi(t,r,\Omega)
	 \quad \text{and}	\quad \\
  &(\cS_\epsilon)(r) = \left( \frac{\sigma_{\rm{t}}(r)}{\epsilon }-\epsilon \sigma_{\rm{a}}(r) \right) \cA, \quad \text{with}
  \ \cA v=\frac{1}{4\pi}\int_{\sph}v(\Omega)\ d\Omega,\ \forall\ v\in L^1(\sph).
\end{align}
 Here $\psi(t,r,\Omega)$ is the density of particles at time $t$ with respect to the measure $d\Omega dr$.  The initial condition $\psi_0$, boundary data $\psi_{\rm{b}}$, and source $q$ are known functions.  The functions $\sigma_{\rm{t}}$ and $\sigma_{\rm{a}}$, which are independent of $\epsilon$, are the (non-dimensional) total and absorption cross-sections, respectively.    With the above scaling, $\sigs=\sigma_{\rm{t}} -\epsilon^2 \sigma_{\rm{a}}\ge0$ is the (non-dimensional) scattering cross-section.   In particular, \eqref{eq:RTEa} is  scattering dominated when $\epsilon \ll 1$.
 A discussion of well-posedness for \eqref{eq:RTE} can be found in \cite{Lions} and a discussion of the scaling can be found in \cite{MorelLarsen2}.
\subsection{Discrete ordinates}
Due to its popularity and easy implementation, we use discrete ordinates for angular discretization in the paper.  The discretization is based on a quadrature rule of order $N$ that is defined by a set of $N^*$ discrete angles $\{\dAng{i}{N}\}_{i=1}^{N^*}\subset\bbS^2$ and weights $\{\dWgt{i}{N}\}_{i=1}^{N^*}$. Here $N^* = N^*(N)$ is a positive integer that is  monotonically increasing as a function of $N$; the dependence of $N^*$ on $N$ depends on the specific  rule.  Usually $N^*$ = $O(N^2)$; this will be the case for the quadrature rules used in this paper.  

The discrete ordinate approximation of \eqref{eq:RTEa} is a vector-valued function $\bspsi^N = \bspsi^N(t,r)$,  whose components $\psi_i^N$, $1\le i\le N^*$, satisfy  
\begin{subequations}\label{eq:RTESN}
\begin{alignat}{3}
(\epsilon \p_t + L_{\epsilon,i}^N)\psi_i^N(t,r)&=S^N_{\epsilon} \bspsi^N(t,r)+\epsilon q_i^N(t,r),\quad&&(t,r)\in(0,T)\times X,\\
\psi_i^N(t,r)&=\psi_0(r,\dAng{i}{N}),\quad&&(t,r)\in \{0\}\times X,\\
\psi_i^N(t,r)&=\psi_{\rm{b}}(t,r,\dAng{i}{N}),\quad&&(t,r)\in (0,T)\times\partial X_i^{-},
\end{alignat}
\end{subequations}
where 
\begin{align}
\label{eq:SN_op_defs}
&L_{\epsilon,i}^N \psi_i^N(t,r)
	=\dAng{i}{N} \cdot \nabla_r\psi_i^N(t,r)+\frac{\sigt(r)}{\epsilon}\psi_i^N(t,r), \quad\text{and}\quad \\
	\label{eq:SN_op_defs2}
&S^N_{\epsilon} (r)
	=\left( \frac{\sigma_{\rm{t}}(r)}{\epsilon }-\epsilon \sigma_{\rm{a}}(r) \right)A^N,\quad\text{with}\quad A^N\bv=\frac{1}{4\pi}\sum_{i=1}^{N^*}\dWgt{i}{N}v_i,\quad\forall~ \bv\in \bbR^{N^*}.
\end{align}
Here the components of $\bspsi^N$ approximate $\psi$ at each quadrature point. Additionally,
\begin{equation} \label{eq:dissource}
q_i^N(t,r)=q(t,r,\dAng{i}{N})
\quad\text{and}\quad
\partial X_i^{\pm} = \{r\in\partial X\colon\pm\dAng{i}{N}\cdot n(r)>0\}
\end{equation}
are, respectively, the discretized source and the inflow/outflow boundaries with respect to $\dAng{i}{N}$.
 
\subsection{Diffusion limit}
In scattering dominated regions, the solution of \eqref{eq:RTE} can be approximated by the solution of a diffusion equation \cite{Habetler, LarsenKeller}; that is, if  
\begin{equation}
\inf_{r\in X} \sigma_{\rm{t}}(r) > 0 \quad\text{and} \quad \inf_{r\in X} \sigma_{\rm{a}}(r) > 0,
\end{equation} 
then for any compactly embedded subset $X_0 \Subset X$ and all $\epsilon$ sufficiently small,\\ $\psi(t, r,\Omega) = \phi(t,r) + O(\epsilon )$ for all $r \in X_0$,
 where $\phi$ satisfies the following diffusion equation
 \begin{equation}\label{eq:difflimeq}
     \partial_t\phi(t,r)-\nabla_r \cdot \left(\frac{1}{3\sigma_{\rm{t}}(r)}\nabla_r\phi(t,r)\right) +\sigma_{\rm{a}}(r)\phi(t,r)=(\cA q)(t,r).
 \end{equation}
Here $X_0$ must be bounded away from $\partial X$ due to the boundary layers of width $O(\epsilon )$ which can appear in the solution of \eqref{eq:RTE} but not the solution of \eqref{eq:difflimeq} \cite{MorelLarsen}.  Discretizations of  \eqref{eq:RTEa} that in the limit $\epsilon \to 0$, become  stable and consistent discretizations of \eqref{eq:difflimeq} are said to capture the (interior) diffusion limit. 


In this paper, we are interested in numerical methods that can capture both the (interior) diffusion limit and the steady-state limit of \eqref{eq:RTE}.  For time integration, the steady-state limit suggests a fully implicit approach.   For spatial discretization, there are a variety of approaches, including edge-based finite difference methods like diamond differencing \cite{Lewis}; DG methods with sufficiently rich trial spaces \cite{MorelLarsen, adams, Guermond}; finite volume methods with modified fluxes \cite{Gosse2004,jin1996,hauck2009temporal}; and methods based on even-odd parity \cite{adams1991even, jin2000uniformly, egger2012mixed} and second-order forms of the RTE \cite{lewis2010second}.   We focus here on DG methods with upwind fluxes, which are particularly attractive because they capture the diffusion limit, are fairly robust to general boundary conditions, and also enable the use of sweeping techniques \cite{Larsen,adams2002fast} that are commonly used for solving the quasi steady-state problems generated by a fully implicit time integrator.

Several angular discretizations have been shown to work well with upwind DG. The initial analysis of the diffusion limit can be found in \cite{MorelLarsen}.  There the discrete ordinates method is employed in a one-dimensional slab geometry.
This analysis was later extended to the multi-dimensional setting in \cite{adams} for a variety of geometries. In \cite{Guermond}, a finite element discretization is used for the angular variables, and the authors re-establish the results from \cite{adams} using functional analytic tools.  Specifically, it is shown that the upwind DG approximation can capture the (interior) diffusion limit as long as the trial space supports global linear functions. In \cite{mcclarren2008}, a spherical harmonic ($P_N$) angular discretization is combined with an upwind DG spatial discretization and a semi-implicit time integration scheme in order to achieve the diffusion limit. 


\section{Spatial hybrid}\label{sec:Hybrid}
In this section, we introduce and analyze a hybrid spatial discretization for the RTE.  The underlying formulation is based on the idea of first collision source \cite{Alcouffe} and has been used in  \cite{Coryhybrid} as a way to combine different angular discretizations in a fully implicit time integration scheme.  Here we combine DG discretization, which performs well in the diffusion limit, with FV discretization, which uses less memory per computational cell, into a single discretization strategy.
\subsection{General hybrid formulation}\label{sec:SNhybrid}
The basic idea of first collision source is to separate $\psi$ into $\psi = \psi_{\rm{u}} +\psi_{\rm{c}}$,
where the uncollided flux $\psi_{\rm{u}}$ and the collided flux $\psi_{\rm{c}}$  satisfy 
\begin{subequations}\label{eq:hy}
\begin{align}
\label{eq:hyuncol}    (\epsilon \p_t + \cL_{\epsilon })\psi_{\rm{u}}&=\epsilon q,\\
\label{eq:hycol}    (\epsilon \p_t + \cL_{\epsilon })\psi_{\rm{c}}&=\cS_\epsilon\left(\psi_{\rm{u}} + \psi_{\rm{c}}\right).
\end{align}
\end{subequations}
The hybrid method evolves these equations for a time-step and then re-initalizes the values $\psi_{\rm{u}}$ and $\psi_{\rm{c}}$.  For illustration, with a backward Euler discretization these steps (in reverse order) give
\begin{equation}\label{eq:init}
\psi_{\rm{u}}^{n+1/2} = \psi_{\rm{u}}^{n} +\psi_{\rm{c}}^{n}   \quad \text{and} \quad \psic^{n+1/2} = 0
\end{equation}
and then
\begin{subequations}\label{eq:hytime}
\begin{align}
\label{eq:hyuncoltime}    \epsilon \frac{\psi_{\rm{u}}^{n+1}-\psi_{\rm{u}}^{n+1/2}}{\Delta t} + \cL_{\epsilon }\psi_{\rm{u}}^{n+1}&=\epsilon q^{n+1},\\
\label{eq:hycoltime}    \epsilon \frac{\psi_{\rm{c}}^{n+1}-\psi_{\rm{c}}^{n+1/2}}{\Delta t}+ \cL_{\epsilon }\psi_{\rm{c}}^{n+1}&=\cS_\epsilon \left(\psi_{\rm{u}}^{n+1} + \psi_{\rm{c}}^{n+1}\right),
\end{align}
\end{subequations}
where $q^{n+1}=q^{n+1}(r,\Omega) = q(t^{n+1},r,\Omega)$.   In practice, $\psiu$ is often approximated by a high-order angular discretization, while $\psic$ uses a low-order angular discretization. 
Since the components of $\psiu$ are uncoupled, a high degree of angular resolution is needed, but the unknowns can be solved in parallel with respect to the angles. For $\psic$, the components are strongly coupled through the operator  $\cS_\epsilon$, and therefore require a relatively small number of unknowns to resolve angular dependencies \cite{Coryhybrid}. \crso{Thus} \crso{t}\cbc{T}he point of the initialization step \eqref{eq:init} is to recover some of the benefits of the high-resolution angular discretization.

The reinitialization step \eqref{eq:init} can be substituted into \eqref{eq:hytime} to obtain a closed rule for updating $\psiu^{n}$ and $\psic^n$:
\begin{subequations}\label{eq:hytimemod}
\begin{align}
\label{eq:hyuncoltimemod}\cL_{\epsilon}^{\Delta t}{\psi}_{\rm{u}}^{n+1}&=\epsilon \left(\frac{\psiu^n + \psic^n}{\Delta t}+ q^{n+1}\right),\\
\label{eq:hycoltimemod}\cL_{\epsilon}^{\Delta t}{\psi}_{\rm{c}}^{n+1}&=\cS_\epsilon \left({\psi}_{\rm{u}}^{n+1} + {\psi}_{\rm{c}}^{n+1}\right),
\end{align}
\end{subequations}
where
\begin{equation}
  \cL_{\epsilon }^{\Delta t} = \Omega\cdot\nabla_r+\frac{\sigma_{\rm{t}}(r)}{\epsilon }+\frac{\epsilon }{\Delta t}.
\end{equation}
Adding \eqref{eq:hyuncoltimemod} and \eqref{eq:hycoltimemod} together yields a backward Euler discretization of \eqref{eq:RTEa}:
\begin{equation}
\label{eq:sumHybrid}
\cL_{\epsilon}^{\Delta t}{\psi}^{n+1}=\cS_\epsilon\psi^{n+1}+\epsilon \left(\frac{1}{\Delta t}\psi^{n}+ q^{n+1}\right),\\
\end{equation}
where $\psi^{n} = {\psi}^{n}_{\rm{u}} + {\psi}_{\rm{c}}^{n}$ and $\psi^{n+1} = {\psi}^{n+1}_{\rm{u}} + {\psi}_{\rm{c}}^{n+1}$. However, the hybrid strategy is to approximate \eqref{eq:hyuncoltimemod} and \eqref{eq:hycoltimemod} in angle and space using different discretizations, in which case the approximations for ${\psi}^{n+1}_{\rm{u}}$ and ${\psi}_{\rm{c}}^{n+1}$ cannot be added directly.

For the remainder of this section we fix $n$ and $\dt$ and focus on finding a hybrid approximation for $\psi^{n+1}$.  Since the discretizations in \eqref{eq:hytimemod} and \eqref{eq:sumHybrid} are implicit, we can reformulate the system as a steady-state problem. Let $V^{\rm u}$ and $V^{\rm c}$ be two finite-dimensional vector spaces and let $f^{\rm{u}}\in V^{\rm u}$ and $f^{\rm{c}} \in V^{\rm c}$ satisfy
 \begin{subequations}\label{eq:hynewapprox}
\begin{align}
\label{eq:hyuncolnew}L^{{\rm u}}_\epsilon f^{\rm{u}}&=\epsilon Q,\\
\label{eq:hycolnew}L^{{\rm c}}_\epsilon f^{\rm{c}}&= S^{\rm c,\rm u}_\epsilon f^{\rm{u}} 
	+ S^{\rm c,\rm c}_\epsilon f^{\rm{c}}.
\end{align}
\end{subequations}
Here $Q \in V^{\rm u}$ is an approximation of $\left(\frac{1}{\Delta t}\psi^{n}+ q^{n+1}\right)$; 
$S^{\rm c,\rm u}_\epsilon: V^{\rm u} \to V^{\rm c}$ and $S^{\rm c,\rm c}_\epsilon: V^{\rm c} \to V^{\rm c}$ are both approximations of $\cS_\epsilon$; and $L^{{\rm u}}_\epsilon:V^{\rm u}\to V^{\rm u}$ and $L^{{\rm c}}_\epsilon: V^{\rm c}\to V^{\rm c}$ are both approximations of $\cL_{\epsilon}^{\Delta t}$.

The next step is to compute an approximation $f \in V^{\rm u}$ of $\psi^{n+1}$ from $f^{\rm{u}}$ and $f^{\rm{c}}$. The strategy in \cite{Crockatt,Coryhybrid} is to let $f = f^{\rm{u}} + \cR f^{\rm{c}}$, where $\cR: V^{\rm c} \to V^{\rm u}$ is a ``relabeling operator".  Here, we instead follow \cite{CrockattDis} and solve the following approximation of \eqref{eq:sumHybrid}:
\begin{equation}\label{eq:hyclosed}
L^{{\rm u}}_\epsilon f= \left(S^{\rm u,\rm u}_\epsilon f^{\rm{u}}+ S^{\rm u,\rm c}_{\epsilon} f^{\rm{c}}\right) + \epsilon Q,
\end{equation}
where $S_\epsilon^{\rm u,\rm u}\colon V^{\rm u}\to V^{\rm u}$ and $S_\epsilon^{\rm u ,\rm c}\colon V^{\rm c}\to V^{\rm u}$ both approximate $\cS_\epsilon$.

While the formulation above uses backward Euler for the temporal discretization, we use a second-order diagonally implicit Runge-Kutta method \cite{DIRK} method for the time-dependent numerical experiments in Section \ref{sec:Results}. Like backward Euler, each stage of this method can be written into a steady-state form.

\subsection{Finite volume / discontinuous Galerkin hybrid}\label{sec:HyImp}
One of the important features of the hybrid method is that it is allows for different discretizations of each component of \eqref{eq:hynewapprox} as well as \eqref{eq:hyclosed}.   While the focus of this work is on the hybridization in space, we also allow for discrete ordinate angular discretizations of different orders.
Specifically, we use finite volume (FV) with (possibly) high-order discrete ordinates  to solve \eqref{eq:hyuncolnew} and \eqref{eq:hyclosed}, and we use discontinuous Galerkin (DG) with (possibly) low-order discrete ordinates to solve \eqref{eq:hycolnew}. The FV and DG discretizations will both be formally second-order and, to allow for sweeping, will use upwinding to define numerical traces at cell interfaces.

 Let the high-order discrete ordinate quadrature be order $N_{\rm u}$ with $N_{\rm u}^*$ points and weights.  Similarly, let the low-order discrete ordinate quadrature be of order $N_{\rm c}$ with $N_{\rm c}^*$ points and weights. For the remainder of the paper, we use the simplified notation
\begin{equation}
\{\oiu,\wku\}_{i=1}^{N_{\rm u}^*}:=\{\dAng{i}{N_{\rm u}},\dWgt{i}{N_{\rm u}}\}_{i=1}^{N_{\rm u}^*},\quad \{\oic,\wkc\}_{i=1}^{N_{\rm c}^*}:=\{\dAng{i}{N_{\rm c}},\dWgt{i}{N_{\rm c}}\}_{i=1}^{N_{\rm c}^*},
\end{equation}
and define the component boundaries
\begin{equation}
\partial X_i^{{\rm n},\pm} = \{r\in\partial X\colon\pm\dAng{i}{\rm n}\cdot n(r)>0\},\quad \text{for } {\rm n\in \{u, c\}}.
\end{equation}
We denote by $\nsol{S}_{N_{\rm{u}}}\nsol{S}_{N_{\rm{c}}}$ a hybrid angular discretization and by X-Y a spatial discretization that uses method X to discretize the uncollided component $\psi_{\rm{u}}$ and method Y to discretize the collided component $\psi_{\rm{c}}$.  In Table \ref{Table 0}, we summarize the leading order terms in the flop count and degrees of freedom, for both Cartesian and triangular meshes.  The values are given per iteration per spatial mesh cell; their derivation is explained in greater detail in Appendix \ref{App1}.

\begin{table}[tbhp]
\scriptsize
\centering
\begin{tabular}{|c|c|c|c|c|}
\cline{2-5}
			\multicolumn{1}{c|}{} &\multicolumn{2}{c|}{Cartesian}&\multicolumn{2}{c|}{Triangular}\\
\hline
{}& Flops & DOF & Flops & DOF\\
\hline
FV & $N^*$ & $N^*$ & $N^*$ & $N^*$\\
\hline
DG & $(2^d+2^{2d}) N^*$ & $2^d N^*$ & $(d+1)^2 N^*$ & $(d+1) N^*$\\
\hline
DG-DG & $(2^d+2^{2d})(N_{\rm u}^* + N_{\rm c}^*)$ & $2^d(N_{\rm u}^* + N_{\rm c}^*)$ & $(d+1)\left(d+2\right)(N_{\rm u}^* + N_{\rm c}^*)$ & $(d+1)(N_{\rm u}^* + N_{\rm c}^*)$\\
\hline
FV-DG & $N_{\rm u}^* + (2^d+2^{2d})N_{\rm c}^*$ & $N_{\rm u}^* + 2^dN_{\rm c}^*$ & $N_{\rm u}^* + (d+1)\left(d+2\right)N_{\rm c}^*$ & $N_{\rm u}^* +(d+1) N_{\rm c}^*$\\
\hline
\end{tabular}
\caption{Leading order operations (Flops) and degrees of freedom (DOF) per iteration per cell. Here  $d$ is the dimension of the spatial domain. DG methods on Cartesian grids use polynomials in $Q_1$ for each cell, and on triangular grids use polynomials in $P_1$. Further details can be found in Appendix \ref{App1}.}\label{Table 0} 
\end{table}

We formulate the spatial discretization using bilinear and linear operators common to DG discretization \cite{Han}.
Let $\{\cT_h\}_{h>0}$ be a regular family of partitions of $X$ into open elements $K$, with $h_K=\text{diam}(K)$ and $h=\max_{K\in \cT_h}h_K$. 
Let $Q_i(K)$ be the set of polynomials with support $K$ of maximum degree $i$ in each spatial dimension, and let $P_i(K)$ be the set of polynomial with support $K$ of total degree less than or equal to $i$.
Let $\cE_h^{\rm in}$ be the set of all interior edges of $\cT_h$; let $\cE_h^{\rm ex}$ be the set of exterior edges; and let $\cE_{h,i}^{{\rm n},\pm}= \cE_h^{\rm ex}\cap \partial X_{i}^{{\rm n},\pm}$ for $\rm n\in\{u, c\}$. 
For each edge $e$, let $n_e$ be a fixed normal vector with respect to $e$.  For interior edges, the direction of $n_e$ is chosen by convention.  
For exterior edges, we assume that $n_e$ points outward from the domain.  Given an edge $e$, let $v$ be any scalar-valued function that is continuous on the cells adjacent to $e$.  
Then the jump of $v$ at $r \in e$ is $\llbracket v\rrbracket (r)=v^+(r)-v^-(r)$, where
\begin{equation}
v^+(r) = \lim_{\varepsilon \to 0+} v(r + \varepsilon n_e)
\quad\text{and} \quad
v^-(r) = \lim_{\varepsilon \to 0+} v(r - \varepsilon n_e).
\end{equation}
For $k \in \{0,1 \}$, define

\begin{equation}
\cX_{h,k} = \{v\in L^2(X):\forall K\in \cT_h,\ v|_{K} \in Z_k(K)\},
\end{equation}
where $Z=Q$ for Cartesian grids and $Z=P$ for triangular grids. 
Let $\cW_{h,k}^{\rm{n}}=(\cX_{h,k})^{N_{\rm{n}}^*}$ for ${\rm n\in\{u,c\}}$ and $k=\{0,1\}$; for all $v^{\rm n}\in \cW_{h,k}^{\rm{n}}$ let
\begin{equation}
v^{\rm n}=[v_1^{\rm n},\ldots,v_{\nnumn}^{\rm n}]^T\cdb{,}
\quad {\rm n\in\{u,c\}},
\end{equation}
where $v_i^{\rm n}\in \cX_{h,k}$ for each $1\le i\le \nnumn$.

For finite volume methods, polynomial approximations are generated from cell averages in neighboring cells. In particular, approximations in $\cX_{h,1}$ for the uncollided equations are generated using elements of $\cX_{h,0}$. 

For the uncollided equation, the  bilinear form ${\cB}_{\epsilon}^{{\rm u}} \colon \cW_{h,0}^{\rm{u}} \times \cW_{h,0}^{\rm{u}} \to \bbR$, corresponding to the left-hand side of \eqref{eq:hyuncolnew}, is given by
\begin{equation}
	\label{eq:Bilin2}
	{\cB}_{\epsilon}^{{\rm u}}(g,\test{})  = \sum_{i=1}^{N_{\rm{u}}^*}\wku{\cB}_{\epsilon,i}^{{\rm u}}(g,\test{}),
\end{equation}
where $\cB_{\epsilon,i}^{{\rm{u}}}\colon \cW_{h,0}^{\rm{u}} \times \cW_{h,0}^{\rm{u}} \to \bbR$ 
is given by
\begin{align}\label{eq:FVbilin}
\cB_{\epsilon,i}^{{\rm{u}}}(g,\test{}) =& \sum_{K\in\cT_h}\int_{K} \frac{\sigma_{\rm{t}}}{\epsilon}\sol{i}\testk{}\ dr + \sum_{e\in\cE^{\rm in}_h}\int_e\oiu\cdot n_e (\SRfull\sol{})_i^{\uparrow}\llbracket \testk{}\rrbracket\ dr
 + \sum_{e\in \cE_{h,i}^{{\rm u},+}}\int_{e}\Omega_{i}^{\rm{u}}\cdot n_e (\SRfull\sol{})_i^{\uparrow}\testk{}\ dr,
\end{align}
with $\SRfull \colon \cW_{h,0}^{\rm{u}}\to\cW_{h,1}^{\rm{u}}$ a reconstruction operator, and for any $v \in \cW_{h,1}^{\rm n}$,
\begin{equation}
v_i^{\uparrow} = \begin{cases}
v_i^- , & \Omega_i^{\rm n} \cdot n_e > 0\cdb{,} \\ 
v_i^+ , & \Omega_i^{\rm n} \cdot n_e < 0, \\ 
\end{cases}\quad {\rm n\in\{u,c\}}.
\end{equation}  
Details of the operator $R$, for the case $Z_i(K)=Q_0(K)$ and $d=2$, are given in Appendix \ref{App2}.

The linear operator $\cF^{{\rm u}} \colon \cW_{h,0}^{\rm{u}} \to \bbR$, corresponding to the right-hand side of \eqref{eq:hyuncolnew}, is given by
\begin{align}
\cF^{\rm{u}}(\test{}) = \sum_{i=1}^{N_{\rm{u}}^*}\wku \cF_{i}^{{\rm u}}(\test{}),
\end{align}
where $\cF_{i}^{{\rm u}} \colon \cW_{h,0}^{\rm{u}} \to \bbR$ is given by 
	\begin{align}
	\cF_{i}^{{\rm u}}(\test{})= \sum_{K\in \cT_h}\int_{K} \epsilon q^{}_i\testk{}\ dr + \sum_{e\in \cE_{h,i}^{{\rm u}, -}}\int_e |\oiu\cdot n_e|\psi_{{\rm b},i}^{}\testk{}\ dr,
	\end{align}
with $q_i=q_i^{N_{\rm u}}(r)$ \cbc{(see \eqref{eq:dissource})} and $\psi_{{\rm b},i}=\psi_{\rm{b}}(r,\oiu)$. 

For the collided equations, the bilinear form ${\cB}_{\epsilon}^{{\rm c}}\colon \cW_{h,1}^{\rm{c}}\times\cW_{h,1}^{\rm{c}}\to\bbR$ corresponding to the left-hand side of \eqref{eq:hycolnew} is given by
\begin{equation}
{\cB}_{\epsilon}^{{\rm c}}(\sol{},\test{})  = \sum_{i=1}^{N_{\rm{c}}^*}\wkc{\cB}_{\epsilon,i}^{{\rm c}}(\sol{},\test{}),
\end{equation}
 where $\cB_{\epsilon,i}^{{\rm c}}\colon\cW_{h,1}^{\rm{c}}\times\cW_{h,1}^{\rm{c}}\to\bbR$ is given by
\begin{subequations}
\begin{align}
\cB_{\epsilon,i}^{{\rm c}}(\sol{},\test{}) =&\sum_{K\in\cT_h}\int_{K} -\sol{i}\Omega_{i}^{{\rm c}}\cdot\nabla_r \testk{} + \frac{\sigma_{\rm{t}}}{\epsilon } \sol{i}\testk{}\ dr\\
\nonumber   &+ \sum_{e\in\cE^{\rm in}_h}\int_e\Omega_{i}^{{\rm c}}\cdot n_e \sol{i}^{\uparrow}\llbracket \testk{}\rrbracket\ dr + \sum_{e\in \cE_{h,i}^{{\rm c}, +}}\int_{e}\Omega_{i}^{\rm c}\cdot n_e \sol{i}^{\uparrow}\testk{}\ dr.
\end{align}
\end{subequations}
Additionally, the bilinear forms $\cC_{\epsilon}^{\rm u,c}\colon\cW_{h,1}^{\rm{u}}\times\cW_{h,1}^{\rm{c}}\to\bbR$ and $\cC_{\epsilon}^{\rm c,c}\colon\cW_{h,1}^{\rm{c}}\times\cW_{h,1}^{\rm{c}}\to\bbR$ corresponding to the right-hand side of \eqref{eq:hycolnew} are given by
\begin{equation}
\cC_{\epsilon}^{\rm n,m}(\sol{},\test{}) = \sum_{i=1}^{N_{\rm n}^*}\wkl\cC_{\epsilon,i}^{\rm n,m},
\end{equation}
where $\cC_{\epsilon,i}^{\rm n,m}\colon\cW_{h,1}^{\rm{n}}\times\cW_{h,1}^{\rm{m}}\to\bbR$ is given by
\begin{equation}
\cC_{\epsilon,i}^{\rm n,m}(\sol{},\test{}) = \sum_{K\in\cT_h}\int_{K}\left(S_{\epsilon}^{\rm m,n}g\right)_iv_i\ dr,\quad {\rm m,n\in\{u,c\}},
\end{equation}
and $S_{\epsilon}^{\rm m,n}:\cW_{h,1}^{\rm{n}}\to\cW_{h,1}^{\rm{m}}$ is a linear operator given by
(cf. \eqref{eq:SN_op_defs2})
\begin{equation}
S_{\epsilon}^{\rm m,\rm n}v = \left(\frac{\sigt}{\epsilon}-\epsilon\siga\right) (A^{\rm n}v)\mathbbm{1}^{\rm m},\quad  \forall\ v\in \cW_{h,1}^{\rm n},\quad A^{\rm n}:=A^{N_{\rm n}},\ \text{for}\ {\rm m,n \in\{u,c\}}.
\end{equation}
Here $\mathbbm{1}^{\rm m}\in \bbR^{N_{\rm m}^\ast}$ is a vector whose components are all one.

Using these bilinear forms, our method is to find $f_{h}^{{\rm{u}}}\in\cW_{h,0}^{{\rm u}}$, $f_{h}^{{\rm{c}}}\in \cW_{h,1}^{\rm{c}}$, and $f_{h}\in\cW_{h,0}^{\rm{u}}$ such that the following holds:
\begin{subequations}\label{eq:hyFVDG}
\begin{alignat}{2}
\label{eq:hyFVDGa}\cB_{\epsilon}^{{\rm{u}}}(f_{h}^{\rm{u}},\test{}) &= \cF^{\rm{u}}(\test{}),\quad &\forall\ \test{}\in\cW_{h,0}^{\rm{u}},\\
\label{eq:hyFVDGb}\cB_{\epsilon}^{{\rm{c}}}(f_{h}^{\rm{c}}, \test{}) - \cC_{\epsilon}^{\rm c,c}(f_{h}^{\rm{c}}, \test{})&=  \cF^{\rm{c}}(\test{}),\quad &\forall\ \test{}\in\cW_{h,1}^{\rm{c}},\\
\label{eq:hyFVDGc}\cB_{\epsilon}^{{\rm{u}}}(f_{h},\test{}) &= \cF(\test{}),\quad &\forall\ \test{}\in\cW_{h,0}^{\rm{u}},
\end{alignat}
\end{subequations}
where the linear operators $\cF^{\rm{c}}\colon\cW_{h,1}^{\rm{c}}\to\bbR$ and $\cF\colon\cW_{h,0}^{\rm{u}}\to\bbR$ are given by
\begin{subequations}\label{eq:DGlinopsource}
\begin{align}
\cF^{\rm{c}}(\test{})&=\cC_{\epsilon}^{\rm u,c}(f_{h}^{\rm{u}},\test{}),\\
\cF(\test{})&= \cF^{\rm{u}}(\test{}) + \cC_{\epsilon}^{\rm u,u}(f_{h}^{\rm{u}},\test{}) + \cC_{\epsilon}^{\rm c,u}(f_{h}^{\rm{c}},\test{}).
\end{align}
\end{subequations}
In this formulation, $f_{h}$ is the approximation to $\psi$ in \eqref{eq:RTE}. 

To assemble the matrix components for the operator form of \eqref{eq:hyFVDG}, let $\mnumz=\dim(\cX_{h,0})$ and $\mnumo=\dim(\cX_{h,1})$. Then $\dim (\cW_{h,0}^{\rm u}) =\mnumz N_{\rm u}^*$  and $\dim (\cW_{h,0}^{\rm c})= \mnumo N_{\rm c}^*$.   Let 
\begin{equation}
\{\bbasu{(i,k)}: i=1, \dots N_{\rm u}^*, k = 1, \dots  \mnumz \}
\quad \text{and} \quad 
\{\bbasc{(i,k)}: i=1, \dots N_{\rm c}^*, k = 1, \dots  \mnumo \}
\end{equation}
be two sets of vector-valued basis functions for $\cW_{h,0}^{\rm u}$ and $\cW_{h,0}^{\rm c}$, respectively, and set
\begin{equation}
f_{h}^{\rm{u}} = \sum_{i=1}^{N_{\rm u}^*} \sum_{k=1}^{\mnumz}\cofuik{(i,k)}\bbasu{(i,k)},
\qquad
f_{h}^{\rm{c}} = \sum_{i=1}^{N_{\rm c}^*} \sum_{k=1}^{\mnumo}\cofcik{(i,k)}\bbasc{(i,k)}
\quad\text{and}\quad 
f_{h} = \sum_{i=1}^{N_{\rm u}^*} \sum_{k=1}^{\mnumz}\cofik{(i,k)}\bbasu{(i,k)}.
\end{equation}
Then matrix form of \eqref{eq:hyFVDG} is an equation for the coefficient vectors
\begin{align}
\cofu &= [\cofuik{(1,1)}, \cofuik{(1,2)},\ldots,\cofuik{(1,\mnumz)}, \cofuik{(2,1)},\ldots,\cofuik{(N_{\rm u}^*,\mnumz)}]^T \in \bbR^{N_{\rm u}^*\mnumz}, \\
\cofc &= [\cofcik{(1,1)}, \cofcik{(1,2)},\ldots,\cofcik{(1,\mnumo)}, \cofcik{(2,1)},\ldots,\cofcik{(N_{\rm c}^*,\mnumo)}]^T \in \bbR^{N_{\rm c}^*\mnumo} ,\text{and} \\
\cof &= [\cofik{(1,1)}, \cofik{(1,2)},\ldots,\cofik{(1,\mnumz)}, \cofik{(2,1)},\ldots,\cofik{(N_{\rm u}^*,\mnumz)}]^T \in \bbR^{N_{\rm u}^*\mnumz},
\end{align}
that takes the form
\begin{subequations}\label{eq:hyFVDGmatvec}
	\begin{alignat}{2}
	\label{eq:hyFVDGmatvecA}L_{\epsilon}^{\rm u} \cofu &= {\bf q}_{h}^{\rm u},\\
	\label{eq:hyFVDGmatvecB} L_{\epsilon}^{\rm c}\cofc - \CCo\cofc &= \CCt\cofu,\\
	\label{eq:hyFVDGmatvecC}L_{\epsilon}^{\rm u} \cof &= \CCth\cofc + \CCf\cofu + {\bf q}_{h}^{\rm u}
	\end{alignat}
\end{subequations}
where the components of 
$L_{\epsilon}^{\rm n}: \bbR^{(\nnumn\mnumn)}\to\bbR^{(\nnumn \mnumn)}$ for ${\rm n\in\{u,c\}}$ are 
\begin{equation}\label{eq:sweep}
\left(L_{\epsilon}^{\rm n}\right)_{(i,k),(i',k')} = \cB_{\epsilon}^{\rm n}(\bbasn{(i',k')},\bbasn{(i,k)}), \qquad  1\leq i,i'\le \nnumn ,\quad  1\leq k,k'\le \mnumn,\\
\end{equation}
the components of $\CCv:\bbR^{(\nnumm\mnumm)}\to\bbR^{(\nnumn\mnumn)}$, for ${\rm m,n\in\{u,c\}}$ are
\begin{align}
\left(\CCv\right)_{(i,k),(i',k')} = \cC_{\epsilon}^{\rm m,n}(\bbasm{(i',k')},\bbasn{(i,k)}), 
\end{align}
for $1\leq i'\le \nnumm,\ 1\leq i\le \nnumn ,\ 1\leq k'\le \mnumm,\ 1\leq k\le \mnumn$, and the $(i,k)$-th component $q_{(i,k)}^{\rm u}$ of the source vector $\bq_h^{\rm u}$ is
\begin{equation}\label{eq:setsource}
q_{(i,k)}^{\rm u} = \cF^{\rm u}(\bbasu{(i,k)}),\quad\forall\ i\le N_{\rm u}^*,\ k\le \mnumz.
\end{equation}
Because $L_{\epsilon}^{\rm u}$ has a sparse structure with triangular blocks, \eqref{eq:hyFVDGmatvecA} and \eqref{eq:hyFVDGmatvecC} can be easily inverted.
To solve \eqref{eq:hyFVDGmatvecB}, we reformulate it into a Krylov framework by inverting the streaming operator $L_{\epsilon,h}^{\rm c}$ and applying the discrete average operator $A^{\rm c}$ to both sides. 
Here we overload the operator $A^{\rm n}:(\cY)^{\nnumn}\to \cY$ for an arbitrary space $\cY$ and $\rm n\in\{u,c\}$, such that 
\begin{equation}\label{eq:integrate}
\left(A^{\rm n}\bv\right)_k = \frac{1}{4\pi}\sum_{i=1}^{\nnumn}\wkl v_{(i,k)},\ \forall\ 1\le k\le \mnumn,\ \forall\ \bv\in\bbR^{\mnumn\nnumn},\ {\rm n\in\{u,c\}}.
\end{equation}
This is consistent with the definition of $A^N$ in \eqref{eq:SN_op_defs2} when $\cY=\bbR$.
Additionally, let $\basu{k}$ and $\basc{k}$ be basis functions such that $\text{span}\{\basu{k}\}_{k=1}^{\mnumz}=\cX_{h,0}$ and $\text{span}\{\basc{k}\}_{k=1}^{\mnumo}=\cX_{h,1}$.  We assume that
\begin{equation}
\bbasn{(i,k)} = [\basni{1},\ldots,\basni{\nnumn}]^T,\quad {\rm n\in\{u,c\}},
\end{equation}
where $\basni{i'}=\basn{k}\delta_{i,i'}$
, in which case
\begin{align}
\left(\CCo\cofc\right)_{i,k} 
=\sum_{k'=1}^{\mnumo}\wkc\int_X\sigse\basc{k'}\basc{k}\left(A^{\rm c}\cofc\right)_{k'},\quad \cbc{\sigse=\left(\frac{\sigt}{\epsilon}-\epsilon\siga\right)}.
\end{align}
Let $\cofnphi=A^{\rm n}\cofn$, and let $\SSq: \bbR^{\mnumm}\to\bbR^{(\nnumn\mnumn)}$ for ${\rm m,n\in\{u,c\}}$, where
\begin{equation}\label{eq:copy}
\left(\SSq\right)_{(i,k),k'} = \wkl\int_X\sigse\basm{k'}\basn{k},\quad 1\le i\le\nnumn,\ 1\le k \le \mnumn,\ 1\le k'\le \mnumm.
\end{equation}  
Then 
\begin{equation}\label{eq:ScatOpConv}
\CCv\cofm = \SSq\cofmphi,\quad {\rm m,n\in\{u,c\}}.
\end{equation}
Using \eqref{eq:ScatOpConv}, we invert the transport operator $L_{\epsilon}^{\rm c}$ in \eqref{eq:hyFVDGmatvecB} and apply the discrete average operator $A^{\rm c}$ to both sides of the resultant equation. Then \eqref{eq:hyFVDGmatvecB} can be written in the following form:
\begin{equation}
\label{eq:hyFVDGmatvecBKrylov}\left(I^{\rm c} - A^{\rm c}(L_{\epsilon}^{\rm c})^{-1}\SSo\right)\cofcphi = A^{\rm c}(L_{\epsilon,h}^{\rm c})^{-1}\SSt\cofuphi,
\end{equation}
where $I^{\rm c}\in\bbR^{(\nnumc\mnumo)\times (\nnumc\mnumo)}$ is the identity matrix, and \eqref{eq:hyFVDGmatvecC} simplifies to 
\begin{equation}\label{eq:hyFVDGmatvecCphi}
L_{\epsilon}^{\rm u} \cof = \SSth\cofcphi + \SSf\cofuphi + {\bf q}_{h}^{\rm u}.
\end{equation}
In summary, \eqref{eq:hyFVDGmatvec} can be solved using \eqref{eq:hyFVDGmatvecBKrylov} and \eqref{eq:hyFVDGmatvecCphi} as outlined in Algorithm \ref{Alg1} below.  
\begin{algorithm}[H]
\caption{Steady State Spatial Hybrid Solution Algorithm. 
}\label{Alg1}
\footnotesize
\begin{algorithmic}
\STATE{\texttt{Data}: $\sigt\ge0$, $\siga\ge0$, $\epsilon > 0$, $\psi_{\rm b}$.}
\STATE{${ q}_{(i,k)}^{\rm u} \leftarrow \cF^{\rm u}(\bbasu{(i,k)})$,\ $\forall\ i\le \nnumu$,\ $\forall\ k\le \mnumz$, \COMMENT{Initialize Source, \eqref{eq:setsource}}\;}
\STATE{$\cofu\leftarrow \left(L_{\epsilon}^{\rm u}\right)^{-1}{\bf q}_{h}^{\rm u}$, 
\COMMENT{Solve \eqref{eq:hyFVDGmatvecA}}} 
\STATE{$\cofuphi \leftarrow A^{\rm u}\cofu,$
\COMMENT{Form $\cofuphi$}}
\STATE{$\cofcphi \leftarrow \left(I^{\rm c} + A^{\rm c}\left(L_{\epsilon}^{\rm c}\right)^{-1}\SSo\right)^{-1}A^{\rm c}\left(L_{\epsilon}^{\rm c}\right)^{-1}\SSt\cofuphi$, \COMMENT{Solve \eqref{eq:hyFVDGmatvecBKrylov} using GMRES}\;}
\STATE{$\cof\leftarrow \left(L_{\epsilon}^{\rm u}\right)^{-1}\left(\SSth\cofcphi + \SSf\cofuphi  + {\bf q}_{h}^{\rm u}\right)$. \COMMENT{Solve \eqref{eq:hyFVDGmatvecC}}\;}
 \end{algorithmic}
\end{algorithm}

\subsection{Diffusion limit, steady state}\label{sec:DiffLim}

In this section, we show that \eqref{eq:hyFVDG} converges to a consistent discretization of the steady-state form of the diffusion limit \eqref{eq:difflimeq}.  The analysis here closely follows \cite{Guermond}.
The scaling of  time in  \eqref{eq:RTEa} ensures that the steady-state analysis applies also to the time-dependent problem \eqref{eq:hytimemod}.

We expand the solutions $\solPsi{h}{\rm{u}},\solPsi{h}{\rm{c},}$, and $\solPsi{h}{}$ of \eqref{eq:hyFVDG} in formal Hilbert expansions:
\begin{subequations}
	\begin{align}\label{eq:HilEx}
	\solPsi{h}{{\rm{u}}}&=\solPsi{h}{\rm{u},(0)}+\epsilon\solPsi{h}{\rm{u},(1)}+\epsilon^2\solPsi{h}{\rm{u},(2)}+O(\epsilon^3),\\
	\solPsi{h}{{\rm{c}}}&=\solPsi{h}{\rm{c},(0)}+\epsilon\solPsi{h}{\rm{c},(1)}+\epsilon^2\solPsi{h}{\rm{c},(2)}+O(\epsilon^3),\\
	\solPsi{h}{}&=\solPsi{h}{(0)}+\epsilon\solPsi{h}{(1)}+\epsilon^2\solPsi{h}{(2)}+O(\epsilon^3),
	\end{align}
\end{subequations}
where $\solPsi{h}{{\rm u}, (j)},\solPsi{h}{(j)}\in \cX_{h,0}$ and $\solPsi{h}{{\rm \crso{u}}\cbc{c}, (j)}\in \cX_{h,1}$ for all $j\in\{0,1,2\}$.   The problem solved by the leading term $\solPsi{h}{(0)}$ is obtained by substituting the expansions \eqref{eq:HilEx} into \eqref{eq:hyFVDG} and matching powers of $\epsilon$. 
In order to perform the analysis, we assume that $\psi_b=0$, that $\sigt\in\cX_{h,0}$, and that all quadratures employed use positive weights and  are exact for polynomials up to degree two:
\begin{equation}
\sum_{i=1}^{N_{\rm n}^*}w_i^{\rm n}= 4 \pi,
\quad
\sum_{i=1}^{N_{\rm n}^*}w_i^{\rm n}\Omega_{i}^{\rm n} = 0,
\quad 
\sum_{i=1}^{N_{\rm n}^*}w_i^{\rm n}\left(\Omega_{i}^{\rm n}\otimes\Omega_{i}^{\rm n}\right) = \frac{4\pi}{3}I,\quad {\rm n}\in\{\rm u,c\}
\end{equation}
where $\otimes$ is the outer product and $I\in \bbR^{3\times 3}$ is the identity matrix.  To simplify the presentation, we set  
\begin{equation}
\solPhi{h}{(j)}{{\rm n}}=A^{{\rm n}}\solPsi{h}{{\rm n},(j)}
\end{equation} 
for $j\in\{0,1,2\}$ and ${\rm n}\in\{\rm{u},\rm{c}\}$, we denote test functions in  $\cW_{h,0}^{\rm u}$ and $\cW_{h,1}^{\rm c}$ by $\test{\rm u}$ and $\test{\rm c}$, respectively, and we assume that jumps across edges in $\cE_{h}^{\rm ex}$ are computed assuming a zero value on the exterior of $X$.  This last assumption allows us to combine terms over $\cE_h = \cE_{h}^{\rm in} \cup \cE_{h}^{\rm ex}$. 

After substituting the expansions \eqref{eq:HilEx} into \eqref{eq:hyFVDG}, the terms that balance at $\epsilon^{-1}$ are
\begin{subequations}\label{eq:APlim0}
\begin{align}
\label{eq:APlim0a}\sum_{i=1}^{N_{\rm u}^*}\wku\int_X\sigt\psolPsi{h}{{\rm{u}},(0)}{i}\testk{u}\ dr&=0,\\
\label{eq:APlim0b}\sum_{i=1}^{N_{\rm c}^*}\wkc\int_X\sigt\left(\psolPsi{h}{{\rm{c}},(0)}{i} - \solPhi{h}{(0)}{{\rm{c}}}\right)\testk{c}\ dr&=\sum_{i=1}^{N_{\rm u}^*}\wku\int_X\sigt \solPhi{h}{(0)}{{\rm{u}}}\testk{c}\ dr,\\
\label{eq:APlim0c}\sum_{i=1}^{N_{\rm u}^*}\wku\int_X\sigt\psolPsi{h}{(0)}{i}\testk{u}\ dr&=\sum_{i=1}^{N_{\rm u}^*}\wku\int_X\sigt \left(\solPhi{h}{(0)}{{\rm{u}}}+\solPhi{h}{(0)}{{\rm{c}}}\right)\testk{u}\ dr;
\end{align}
\end{subequations}
the terms that balance at $\epsilon^{0}$ are
\begin{subequations}\label{eq:APlim1}
\begin{align}
\label{eq:APlim1a}\sum_{i=1}^{N_{\rm u}^*}\wku&\int_{X}\sigt\psolPsi{h}{{\rm{u}},(1)}{i}\testk{u}\ dr + \sum_{e\in\cE_h}\int_e\Omega_{i}^{\rm{u}}\cdot n_e\left(\SRfull\solPsi{h}{{\rm{u}},(0)}\right)_i^{\uparrow}\bracket{\testk{u}}\ dS(r) = 0,\\
\label{eq:APlim1b}\sum_{i=1}^{N_{\rm c}^*}\wkc&\int_X-\psolPsi{h}{{\rm{c}},(0)}{i}\Omega_{i}^{\rm{c}}\cdot\nabla_r\testk{c}+\sigt\left(\psolPsi{h}{{\rm{c}},(1)}{i} - \solPhi{h}{(1)}{{\rm{c}}}\right)\testk{c}\ dr&\\
\nonumber&+\sum_{e\in\cE_h}\int_e\Omega_{i}^{\rm{c}}\cdot n_e \left(\psolPsi{h}{{\rm{c}},(0)}{i}\right)^{\uparrow}\llbracket \testk{c}\rrbracket\ dS(r) =\sum_{i=1}^{N_{\rm c}^*}\wkc\int_X\sigt \solPhi{h}{(1)}{{\rm{u}}}\testk{c}\ dr;
\end{align}
\end{subequations}
and the terms that balance at $\epsilon$ are
\begin{subequations}\label{eq:APlim2}
\begin{align}
\label{eq:APlim2a}\sum_{i=1}^{N_{\rm u}^*}\wku&\int_{X}\sigt\psolPsi{h}{{\rm{u}},(2)}{i}\testk{u} + \sum_{e\in\cE_h}\int_e\Omega_{i}^{\rm{u}}\cdot n_e\left(\SRfull\solPsi{h}{{\rm{u}},(1)}\right)_{i}^{\uparrow}\bracket{\testk{u}}\ dS(r) = \sum_{i=1}^{N_{\rm u}^*}\wku\int_X q_i\testk{u},\\
\label{eq:APlim2b}\sum_{i=1}^{N_{\rm c}^*}\wkc&\int_X-\psolPsi{h}{{\rm{c}},(1)}{i}\Omega_{i}^{\rm{c}}\cdot\nabla_r\testk{c}+\sigt\left(\psolPsi{h}{{\rm{c}},(2)}{i} - \solPhi{h}{(2)}{{\rm{c}}}\right)\testk{c} + \siga \solPhi{h}{(0)}{{\rm{c}}}\testk{c}\ dr&\\
\nonumber&+\sum_{e\in\cE_h}\int_e\Omega_{i}^{\rm{c}}\cdot n_e \left(\psolPsi{h}{{\rm{c}},(1)}{i}\right)^{\uparrow}\bracket{\testk{c}}\ dS(r) =\sum_{i=1}^{N_{\rm c}^*}\wkc\int_X\left(\sigt \solPhi{h}{(2)}{{\rm{u}}}-\siga \solPhi{h}{(0)}{{\rm{u}}}\right)\testk{c}\ dr.
\end{align}
\end{subequations}
The contributions of \eqref{eq:hyFVDGc} to the balance equations at order $\epsilon^0$ and $\epsilon$ are omitted in \eqref{eq:APlim1} and \eqref{eq:APlim2} as they will not be used in the analysis that follows.

Let $\cC_{h,1}$ be the subspace of $\cX_{h,1}$ where every element is continuous.  Let
\begin{equation}
J_{{\rm{c}},h}^{(0)}=\sum_{i=1}^{N_{\rm c}^*}\wkc \Omega_{i}^{\rm{c}}\psolPsi{h}{{\rm{c}},(1)}{i}.
\end{equation}
Let $\vecsrc=[q_1,q_2,\ldots,q_{N_{\rm u}^*}]^T$ and $\bar{\vecsrc} = A^{\rm{u}} \vecsrc$.  Let $P_0$ be the orthogonal projection from $L^2(X)$ onto $\cX_{h,0}$ with respect to the usual inner product.
Our main result is the following.
\begin{theorem}\label{thm:mainresult}
Let $\solPsi{h}{{\rm u}, (j)},\solPsi{h}{(0)}\in \cX_{h,0}$ and $\solPsi{h}{{\rm c}, (j)}\in \cX_{h,1}$ solve \eqref{eq:APlim0}, \eqref{eq:APlim1}, and \eqref{eq:APlim2} for $j\in\{0,1,2\}$. 
Additionally, let $\sigt\in\cX_{h,0}$ and $\sigt\ge\siga>0$. 
Then $\psolPsi{h}{(0)}{i}=P_0\solPhi{h}{(0)}{{\rm{c}}}$ for all $i\le N_{\rm u}^*$.
Moreover, for all $\testcon\in \cC_{h,1}$ and $\varphi\in (\cX_{h,1})^3$, $J_{{\rm{c}},h}^{(0)}$ and $\solPhi{h}{(0)}{{\rm{c}}}$ satisfy 
\begin{subequations}
\begin{align}
\label{eq:APlemord2-0a}\int_X - J_{{\rm{c}},h}^{(0)}\cdot\nabla_r\testcon  + 4\pi\siga \solPhi{h}{(0)}{{\rm{c}}}\testcon\ dr &= 4\pi\int_X \left(P_0\bar{\vecsrc}\right)\testcon\ dr,\\
\label{eq:APlemord2-0b}\int_X\left(\frac{4\pi}{3}\nabla_r \solPhi{h}{(0)}{{\rm{c}}}+\sigt J_{{\rm{c}},h}^{(0)}\right)\cdot \varphi \ dr&=0.
\end{align}
\end{subequations}
\end{theorem}
Theorem \ref{thm:mainresult} is a consistent discretization of the first-order, steady-state form of the diffusion limit \eqref{eq:difflimeq}. To prove it we first require some preliminary lemmas.

\begin{lemma}\label{lem:AP0}
Let $\solPsi{h}{{\rm{u}},(0)},\solPsi{h}{{\rm{c}},(0)},$ and $\solPsi{h}{(0)}$ solve \eqref{eq:APlim0}. Then $\psolPsi{h}{{\rm{u}},(0)}{i}=0$ for all $i\le N_{\rm u}^*$, $\psolPsi{h}{{\rm{c}},(0)}{i}=\solPhi{h}{(0)}{{\rm{c}}}$ for all $i\le N_{\rm c}^*$, and $\psolPsi{h}{(0)}{i}=P_0\solPhi{h}{(0)}{{\rm{c}}}$ for all $i\le N_{\rm u}^*$.
\end{lemma}
\begin{proof}
Let $\testk{u}= \psolPsi{h}{{\rm{u}},(0)}{i}$ for all $i\le N_{\rm u}^*$ in \eqref{eq:APlim0a}. Then
\begin{equation}\label{eq:APlem1eq1}
\sum_{i=1}^{N_{\rm u}^*}\wku\int_X\sigt\left(\psolPsi{h}{{\rm{u}},(0)}{i}\right)^2\ dr=0.
\end{equation}
Since $\sigt>0$ and $\wku>0$ for all $i\le N_{\rm u}^*$, it follows from \eqref{eq:APlem1eq1} that $\psolPsi{h}{{\rm{u}},(0)}{i}=0$ for all $i\le N_{\rm u}^*$.

We now show that $\psolPsi{h}{{\rm{c}},(0)}{i} = \solPhi{h}{(0)}{{\rm{c}}}$ for all $i\le N_{\rm c}^*$.
Let $\testk{c} = \psolPsi{h}{{\rm{c}},(0)}{i}-\solPhi{h}{(0)}{{\rm{c}}}$ for all $i\le N_{\rm c}^*$. Then \eqref{eq:APlim0b} becomes
\begin{equation}\label{eq:APlem1eq2}
\sum_{i=1}^{N_{\rm c}^*}\wkc\int_X\sigt\left(\psolPsi{h}{{\rm{c}},(0)}{i} - \solPhi{h}{(0)}{{\rm{c}}}\right)^2\ dr=0.
\end{equation}
Since $\sigt>0$ and $\wkc>0$ for all $i\le N_{\rm c}^*$, \eqref{eq:APlem1eq2} implies $\psolPsi{h}{{\rm{c}},(0)}{i} = \solPhi{h}{(0)}{{\rm{c}}}$ for all $i\le N_{\rm c}^*$.

We now show that $\psolPsi{h}{(0)}{i}=P_0\solPhi{h}{(0)}{{\rm{c}}}$ for all $i\le N_{\rm u}^*$. Let $\testk{u}=\psolPsi{h}{(0)}{i}-P_0\solPhi{h}{(0)}{{\rm{c}}}$ for all $i\le N_{\rm u}^*$ in \eqref{eq:APlim0c}. Since $\solPhi{h}{(0)}{{\rm u}}=0$, $\sigt\in \cX_{h,0}$, and $P_0$ is the orthogonal projection onto $\cX_{h,0}$, it follows then that
\begin{align}\label{eq:APlem1eq3}
\sum_{i=1}^{N_{\rm u}^*}\wku\int_{X}\sigt \left(\psolPsi{h}{(0)}{i}-P_0\solPhi{h}{(0)}{{\rm{c}}}\right)^2\ dr=0.
\end{align}
Since $\sigt>0$ and $\wku>0$ for all $i\le N_{\rm u}^*$, \eqref{eq:APlem1eq3} implies $\psolPsi{h}{(0)}{i}=P_0\solPhi{h}{(0)}{{\rm{c}}}$ for all $i\le N_{\rm u}^*$.
\end{proof}

\begin{lemma}\label{lem:AP1}
Let $\solPsi{h}{{\rm{u}},(0)}$, $\solPsi{h}{{\rm{c}},(0)}$, $\solPsi{h}{{\rm{u}},(1)}$, $\solPsi{h}{{\rm{c}},(1)}$ solve the system \eqref{eq:APlim0a}, \eqref{eq:APlim0b}, and  \eqref{eq:APlim1}.  Then $\psolPsi{h}{{\rm{u}},(1)}{i}=0$ for all $i\le N_{\rm u}^*$, and $\solPhi{h}{(0)}{{\rm{c}}}$ is continuous on $X$.
\end{lemma}
\begin{proof}
According to Lemma \ref{lem:AP0}, $\solPsi{h}{\rm{u},(0)}=0$.  Hence \eqref{eq:APlim1a} becomes
\begin{equation}
\label{eq:APlim0areduced}\sum_{i=1}^{N_{\rm u}^*}\wku\sum_{K\in\cT_h}\int_{K}\sigt\psolPsi{h}{{\rm{u}},(1)}{i}\testk{u} = 0.\\
\end{equation}
Similar to the proof of Lemma \ref{lem:AP0}, this implies $\psolPsi{h}{{\rm{u}},(1)}{i}=0$ for all $i\le N_{\rm u}^*$.

We next show that $\solPhi{h}{(0)}{{\rm{c}}}$ is a continuous function on $X$, in particular by showing that it is continuous across the cell edges.
Let $\testk{c}=\solPhi{h}{(0)}{{\rm c}}$ for all $1\le i\le N_{\rm c}^*$. From Lemma \ref{lem:AP0}, $\psolPsi{h}{{\rm{c}},(0)}{i}=\solPhi{h}{(0)}{{\rm{c}}}$ for all $i\le N_{\rm c}^*$.  Using the accuracy of the quadrature, which calculates the integral of degree one polynomials exactly, the first term in \eqref{eq:APlim1b} becomes 
\begin{equation}\label{eq:lem2eq1}
- \sum_{i=1}^{N_{\rm c}^*}\wkc \Omega_{i}^{\rm{c}}\cdot  \int_X\solPhi{h}{(0)}{{\rm{c}}}\nabla_r\solPhi{h}{(0)}{{\rm c}}\ dr=0.
\end{equation}
The second term in \eqref{eq:APlim1b} also vanishes because
\begin{equation}\label{eq:lem2eq2}
\sum_{i=1}^{\nnumc}\wkc\left(\psolPsi{h}{{\rm{c}},(1)}{i} - \solPhi{h}{(1)}{{\rm{c}}}\right) = 0.
\end{equation}
The third term in \eqref{eq:APlim1b} is
\begin{equation}\label{eq:lem2eq3}
\sum_{i=1}^{N_{\rm c}^*}\wkc\sum_{e\in\cE_h}\int_e\Omega_{i}^{\rm{c}}\cdot n_e \left(\solPhi{h}{(0)}{{\rm{c}}}\right)^{\uparrow}\bracket{\solPhi{h}{(0)}{{\rm{c}}}}\ dS(r) 
 = \sum_{e\in\cE_h}\int_e\left(\sum_{i=1}^{N_{\rm c}^*}\wkc|\Omega_{i}^{\rm{c}}\cdot n_e|\right)\bracket{\solPhi{h}{(0)}{{\rm{c}}}}^2 dS(r),
\end{equation}
while the right-hand side of \eqref{eq:APlim1b} vanishes.  Therefore, since $\wkc >0$, it follows from
\eqref{eq:lem2eq3} that $\bracket{\solPhi{h}{(0)}{{\rm{c}}}}=0$ for every edge in $\cE_h$.
\end{proof}

\begin{proof}[Proof of Theorem \ref{thm:mainresult}]
We first show that $\psolPsi{h}{{\rm{u}},(2)}{i}=\sigt^{-1}{P_0q_i}$ for all $i\le N_{\rm u}^*$. By Lemmas \ref{lem:AP0} and \ref{lem:AP1}, $\solPsi{h}{\rm{u},(0)}=0$ and $\solPsi{h}{\rm{u},(1)}=0$. This implies that \eqref{eq:APlim2a} becomes
\begin{equation}\label{eq:APlemord2-1}
\sum_{i=1}^{N_{\rm u}^*}\wku\int_{X}\left(\sigt\psolPsi{h}{{\rm{u}},(2)}{i}-q_i\right)\testk{u} = 0.
\end{equation} 
Let $\testk{u}=\sigt\psolPsi{h}{{\rm{u}},(2)}{i}-P_0q_i$ for all $i\le N^*_{\rm{u}}$. Since $P_0$ is the orthogonal projection onto $\cX_{h,0}$, \eqref{eq:APlemord2-1} implies
\begin{equation}\label{eq:APthm1eq2}
\sum_{i=1}^{N_{\rm u}^*}\wku\int_{X}\left(\sigt\psolPsi{h}{{\rm{u}},(2)}{i}-P_0q_i\right)^2= 0.
\end{equation}
Similar to the proof of Lemma \ref{lem:AP0}, \eqref{eq:APthm1eq2} implies $\psolPsi{h}{{\rm{u}},(2)}{i}=\sigt^{-1} {P_0q_i}$,
%
and, as a consequence,
\begin{equation}\label{eq:APlemord2-3}
\solPhi{h}{(2)}{{\rm{u}}}=\sigt^{-1} P_0\bar{\vecsrc}.
\end{equation}

We have shown that $\psolPsi{h}{(0)}{i}=P_0\solPhi{h}{(0)}{{\rm{c}}}$ in Lemma \ref{lem:AP0}. We now show that $\solPhi{h}{(0)}{{\rm{c}}}$ and $J_{{\rm{c}},h}^{(0)}$ satisfy \eqref{eq:APlemord2-0a}. 
Lemma \ref{lem:AP1} implies that $\solPhi{h}{(0)}{{\rm{c}}}\in \cC_{h,1}$.
Let $\testk{c}=\testcon\in \cX_{h,1}$ for all $1\le i\le\nnumc$. 
Then the first term of \eqref{eq:APlim2b} becomes
\begin{equation}\label{eq:Thmgrad}
\sum_{i=1}^{N_{\rm c}^*}\wkc\int_X-\psolPsi{h}{{\rm{c}},(1)}{i}\Omega_{i}^{\rm{c}}\cdot\nabla_r \testcon
= - \int_X J_{{\rm{c}},h}^{(0)}\cdot\nabla_r \testcon.
\end{equation}
The subsequent terms involving $\psolPsi{h}{{\rm{c}},(2)}{i}$ and $\solPhi{h}{(2)}{{\rm{c}}}$ in \eqref{eq:APlim2b} will cancel owing to the definition of $\solPhi{h}{(2)}{{\rm{c}}}$. The next term involving $\solPhi{h}{(0)}{{\rm{c}}}$ is simply
\begin{equation}\label{eq:Thmabs}
\sum_{i=1}^{N_{\rm c}^*}\wkc\int_X\siga \solPhi{h}{(0)}{{\rm{c}}}\testcon\ dr = \int_X 4\pi\siga \solPhi{h}{(0)}{{\rm{c}}}\testcon\ dr.
\end{equation}
We now restrict the test functions to be continuous on $X$. As a result the jump term in \eqref{eq:APlim2b} becomes 
 \begin{equation}
\sum_{e\in\cE_h}\int_e\Omega_{i}^{\rm{c}}\cdot n_e \left(\psolPsi{h}{{\rm{c}},(1)}{i}\right)^{\uparrow}\bracket{\testcon}\ dS(r) = 0. 
 \end{equation}
 Using the fact that $\psolPsi{h}{{\rm{u}},(0)}{i}=0$ (from Lemma \ref{lem:AP0}) and \eqref{eq:APlemord2-3}, the right side of \eqref{eq:APlim2b} becomes 
\begin{equation}\label{eq:Thmsrc}
\sum_{i=1}^{N_{\rm c}^*}\wkc\int_X\sigt \solPhi{h}{(2)}{{\rm{u}}}\testcon\ dr = \int_X 4\pi P_0\bar{\vecsrc}\testcon\ dr.
\end{equation}
Combining the results in \eqref{eq:Thmgrad}, \eqref{eq:Thmabs}, and \eqref{eq:Thmsrc}, \eqref{eq:APlim2b} becomes
\begin{equation}
\int_X - J_{{\rm{c}},h}^{(0)}\cdot\nabla_r \testcon+ 4\pi\siga \solPhi{h}{(0)}{{\rm{c}}}\testcon\ dr = \int_X 4\pi\sigt P_0\bar{\vecsrc}\testcon\ dr.
\end{equation} 

Now we show that $\solPhi{h}{(0)}{{\rm{c}}}$ and $J_{{\rm{c}},h}^{(0)}$ satisfy \eqref{eq:APlemord2-0b}. 
Let $\testk{\rm c}=\varphi\cdot\Omega_{i}^{\rm{c}}$ for all $1\le i\le \nnumc$, where $\varphi\in\left(\cX_{h,1}\right)^3$ is arbitrary.
Using integration by parts and recalling that $\psolPsi{h}{{\rm{c}},(0)}{i}=\solPhi{h}{(0)}{{\rm{c}}}$ for all $1\le i\le\nnumc$, (from Lemma \ref{lem:AP0}), and $\psolPsi{h}{{\rm{u}},(1)}{i}=0$ for all $i\le N_{\rm u}^*$, (from Lemma \ref{lem:AP1}), we can rewrite \eqref{eq:APlim1b} as the following:
\begin{align}\label{eq:APlem2-1}
\sum_{i=1}^{N_{\rm c}^*}\wkc&\int_X\left(\Omega_{i}^{\rm{c}}\cdot\nabla_r\solPhi{h}{(0)}{{\rm{c}}}+\sigt\left(\psolPsi{h}{{\rm{c}},(1)}{i} - \solPhi{h}{(1)}{{\rm{c}}}\right)\right)\left(\varphi\cdot\Omega_{i}^{\rm{c}}\right)\ dr&\\
\nonumber&+\sum_{e\in\cE_h}\int_F-\Omega_{i}^{\rm{c}}\cdot n_1 \bracket{\solPhi{h}{(0)}{{\rm{c}}}} \left(\varphi\cdot\Omega_{i}^{\rm{c}}\right)^{\downarrow}\ dS(r)  = 0.
\end{align}
The term involving $\nabla_r\solPhi{h}{(0)}{{\rm{c}}}$ in \eqref{eq:APlem2-1} is 
\begin{equation}\label{eq:Thmgrad2}
\sum_{i=1}^{N_{\rm c}^*}\wkc\int_X\nabla_r\solPhi{h}{(0)}{{\rm{c}}}\cdot\left(\Omega_{i}^{\rm{c}}\otimes\Omega_{i}^{\rm{c}}\right)\varphi\ dr = 
 \int_X\frac{4\pi}{3}\nabla_r\solPhi{h}{(0)}{{\rm{c}}}\cdot\varphi\ dr.
\end{equation}
Computing the contribution of $\solPsi{h}{\rm{c},(1)}$ in \eqref{eq:APlem2-1} we have,
\begin{equation}\label{eq:Thmtot}
\sum_{i=1}^{N_{\rm c}^*}\wkc\int_X\sigt\left(\psolPsi{h}{{\rm{c}},(1)}{i}-\solPhi{h}{(1)}{{\rm{c}}}\right)\left(\varphi\cdot\Omega_{i}^{\rm{c}}\right)\ dr
= \int_X\sigt J_{{\rm{c}},h}^{(0)}\cdot\varphi\ dr.
\end{equation}
All jump terms on the edges are zero since $\solPhi{h}{(0)}{{\rm{c}}}\in C_{h,1}$. Combining the results from \eqref{eq:Thmgrad2} and \eqref{eq:Thmtot} gives \eqref{eq:APlemord2-0b}.
\end{proof}

\section{Numerical Results}\label{sec:Results}
In this section, we compare the performance of the spatial hybrid to standard DG and FV approaches.  We also investigate the benefits of hybridization in both the angular and spatial variables.  In Section \ref{sec:NumDiffLim} we demonstrate the numerical properties of the hybrid in the diffusion limit. In the remaining subsections, we use benchmark problems to assess efficiency and accuracy.  

All numerical simulations are performed on a reduced spatial geometry that assumes no variations in the $z$ direction.   Discrete ordinates based on a tensor product quadrature \cite{Atk}
with $N^* = N^2$ are used to discretize in angle.  In all cases, the domain is Cartesian, the mesh is square, and the DG elements are $Q_1$.  The finite volume discretization uses a second-order reconstruction with slopes computed using only upwind information (see Appendix \ref{App2}).  For time-dependent problems, a second-order strongly $S$-stable DIRK scheme with $\alpha=1-1/\sqrt{2}$ \cite[Thm. 5]{DIRK} is used. 

All results in this section are expressed in terms of the particle concentration $\phi$, given by 
\begin{equation}
\phi(t,r) =\int_{\sph}\psi(t,r,\Omega)\ d\Omega.
\end{equation}
Given a reference solution $\Phi^{\rm{ref}}$ for $\phi$ and a numerical solution $\Phi$, errors are calculated in two relative norms:
\begin{equation}\label{eq:err}
E_2(t) = \frac{||\Phi^{\rm{ref}}- \Phi||_{L^2(X)}}{||\Phi^{\rm{ref}}||_{L^2(X)}},
\quad 
E_{\infty}(t) = \frac{||\Phi^{\rm{ref}}- \Phi||_{L^{\infty}(X)}}{||\Phi^{\rm{ref}}||_{L^{\infty}(X)}}.
 \end{equation}

All simulations were run on a machine with dual E5-2699 v4 CPUs, each with 22 physical cores (44 logical) running at 2.20 GHz. The machine has 512 GB of DDR4 memory running at 2133 MHz. The simulations in Section \ref{sec:NumDiffLim} were run with code implementing OpenMP parallelization with 16 threads. All remaining simulations were run in serial.

\subsection{Diffusion Limit Test}\label{sec:NumDiffLim}

We solve a steady-state version of \eqref{eq:RTE} in $x$-$y$ geometry with zero boundary condition on $\Gamma^{-}$, $\sigt=4.0$, $\siga=0.25$, and 
$q(x,y,\Omega) = 900x^2y^2(1-x)^2(1-y)^2m_{1,1}^2(\Omega)$, where
\begin{equation}
m_{1,1}(\Omega)=\sqrt{\frac{3}{4\pi}}\Omega_x.
\end{equation}

%
Using an $\nsol{S}_8$ angular discretization, we compare numerical results using standard DG, FV, and a DG-FV hybrid. We examine errors and order of convergence with respect to the spatial mesh $h$ as $\epsilon$ varies, using a DG spatial discretization with  $h = 1/256$ as a numerical reference.  Results are shown in Tables \ref{Table4-1}--\ref{Table4-3}. The DG-DG scheme maintains second-order convergence in $h$ for large and small $\epsilon$, although it loses order for intermediate values of $\epsilon$.  Reductions in order of this type are common in multiscale problems \cite{jin}.  The finite volume method performs well for large values of $\epsilon$, but the convergence is lost as $\epsilon$ gets smaller.  As expected, the new FV-DG hybrid performs similarly to the DG-DG method, with a similar drop in convergence order for intermiate values of $\epsilon$.  However, errors for the FV-DG hybrid are $2
$--$3$ times larger than the DG-DG scheme for larger $\epsilon$.

\begin{table}[tbhp]
	\centering
	\scriptsize
	{
		\begin{tabular}{|c|c|c|c|c|c|c|c|c|c|c|}
			\cline{2-11}
			\multicolumn{1}{c|}{} &\multicolumn{2}{c|}{$h=1/8$}&\multicolumn{2}{c|}{$h=1/16$}&\multicolumn{2}{c|}{$h=1/32$}&\multicolumn{2}{c|}{$h=1/64$}&\multicolumn{2}{c|}{$h=1/128$} \\
			\hline	
			$\epsilon$ &	$E_2$&	Ord.&	$E_2$&	Ord.&	$E_2$&	Ord.&	$E_2$&	Ord.&	$E_2$&	Ord.\\ \hline
			1 &	8.55E-3 &	- &	2.24E-3 &	1.93 &	5.90E-4 &	1.92 &	1.52E-4 &	1.95 &	3.71E-5 &	2.04 \\ \hline
			$2^{-1}$ &	8.05E-3 &	- &	2.18E-3 &	1.88 &	6.34E-4 &	1.78 &	1.78E-4 &	1.83 &	4.51E-5 &	1.98 \\ \hline
			$2^{-5}$ &	1.22E-2 &	- &	2.79E-3 &	2.12 &	6.63E-4 &	2.07 &	1.90E-4 &	1.80 &	6.57E-5 &	1.54 \\ \hline
			$2^{-9}$ & 1.40E-2 &	- &	3.45E-3 &	2.02 &	8.41E-4 &	2.04 &	1.99E-4 &	2.08 &	4.28E-5 &	2.22 \\ \hline
			$2^-{13}$ &	1.42E-2 &	- &	3.52E-3 &	2.01 &	8.72E-4 &	2.01 &	2.11E-4 &	2.05 &	4.62E-5 &	2.19 \\ \hline
	\end{tabular}}
	\caption{Errors $E_2$ (see \eqref{eq:err}) and convergence with respect to $h$ for DG $\nsol{S}_8$}\label{Table4-1}
\end{table}


\begin{table}[tbhp]
	\centering
	\scriptsize
	{
		\begin{tabular}{|c|c|c|c|c|c|c|c|c|c|c|}
			\cline{2-11}
			\multicolumn{1}{c|}{} &\multicolumn{2}{c|}{$h=1/8$}&\multicolumn{2}{c|}{$h=1/16$}&\multicolumn{2}{c|}{$h=1/32$}&\multicolumn{2}{c|}{$h=1/64$}&\multicolumn{2}{c|}{$h=1/128$} \\
			\hline	
			$\epsilon$ &	$E_2$&	Ord.&	$E_2$&	Ord.&	$E_2$&	Ord.&	$E_2$&	Ord.&	$E_2$&	Ord.\\ \hline
			1 &	6.84E-2 &	- &	1.69E-2 &	2.02 &	4.35E-3 &	1.96 &	1.10E-3 &	1.98 &	2.79E-4 &	1.98 \\ \hline
			$2^{-1}$ &	9.18E-2 &	- &	2.17E-2 &	2.08 &	5.29E-3 &	2.03 &	1.31E-3 &	2.01 &	3.34E-4 &	1.98 \\ \hline
			$2^{-5}$ &	3.98E-1 &	- &	9.53E-2 &	2.06 &	1.67E-2 &	2.51 &	2.79E-3 &	2.59 &	5.14E-4 &	2.44 \\ \hline
			$2^{-9}$ &	9.03E-1 &	- &	5.67E-1 &	0.67 &	1.57E-1 &	1.86 &	1.40E-1 &	0.16 &	1.93E-1 &	-0.46 \\ \hline
			$2^{-13}$ &	9.93E-1 &	- &	9.54E-1 &	0.06 &	8.39E-1 &	0.19 &	6.40E-1 &	0.39 &	8.84E-1 &	-0.47 \\ \hline
	\end{tabular}}
	\caption{Errors $E_2$ (see \eqref{eq:err}) and convergence with respect to $h$ for FV $\nsol{S}_8$}\label{Table3}
\end{table}

\begin{table}[tbhp]
\centering
\scriptsize
\begin{tabular}{|c|c|c|c|c|c|c|c|c|c|c|}
\cline{2-11}
\multicolumn{1}{c|}{} &\multicolumn{2}{c|}{$h=1/8$}&\multicolumn{2}{c|}{$h=1/16$}&\multicolumn{2}{c|}{$h=1/32$}&\multicolumn{2}{c|}{$h=1/64$}&\multicolumn{2}{c|}{$h=1/128$} \\
\hline	
$\epsilon$ &	$E_2$&	Ord.&	$E_2$&	Ord.&	$E_2$&	Ord.&	$E_2$&	Ord.&	$E_2$&	Ord.\\ \hline
1 &	3.65E-2 &	- &	1.01E-2 &	1.86 &	2.54E-3 &	1.98 &	6.35E-4 &	2.00 &	1.58E-4 &	2.00 \\ \hline
$2^{-1}$ &	2.65E-2 &	- &	6.53E-3 &	2.02 &	1.56E-3 &	2.07 &	3.81E-4 &	2.03 &	9.36E-5 &	2.02 \\ \hline
$2^{-5}$ &	1.59E-2 &	- &	3.36E-3 &	2.24 &	7.36E-4 &	2.19 &	1.97E-4 &	1.90 &	6.62E-5 &	1.58 \\ \hline
$2^{-9}$ &	1.71E-2 &	- &	3.96E-3 &	2.11 &	9.19E-4 &	2.11 &	2.09E-4 &	2.14 &	4.39E-5 &	2.25 \\ \hline
$2^{-13}$ &	1.72E-2 &	- &	4.03E-3 &	2.10 &	9.51E-4 &	2.08 &	2.22E-4 &	2.10 &	4.75E-5 &	2.22 \\ \hline
\end{tabular}
\caption{Errors $E_2$ (see \eqref{eq:err}) and convergence with respect to $h$ for FV-DG $\nsol{S}_8$ $\nsol{S}_8$}\label{Table4-3}
\end{table}

\subsection{Line source benchmark}
The line source is a benchmark problem that was first formulated in \cite{Ganapol} as a means to verify numerical methods and assess any strengths or weaknesses. The problem describes an initial pulse of particles distributed isotropically along an infinite line in space moving through a purely scattering material medium as time evolves. In the reduced two-dimensional geometry, the initial pulse is expressed as a delta function at the origin of the two-dimensional domain.

\subsubsection{Example 1}\label{sec:Eff}

In this example, we demonstrate how hybridization can be used to compute qualitatively similar solution with less computational resources.  We simulate \eqref{eq:RTE} with $\epsilon=1$ and approximate the initial condition using a Gaussian distribution with small standard deviation $\beta$:
\begin{equation}\label{eq:LSinicond}
\psi_0(r,\Omega)=\frac{1}{8\beta^2\pi^2}e^{\frac{-|r|^2}{2\beta^2}},\quad \beta = 0.09 .
\end{equation}
We consider problem with an absorption cross-section $\sigma_{\rm{a}}=0$, total cross-section $\sigma_{\rm{t}}=1$, source $q=0$, and boundary condition $\psi_{\rm{b}}=0$. For reference, a semi-analytic solution is computed using the algorithm described in \cite{Garrett};  see Figure \ref{isofigure0}.   We simulate the problem using a $301\times301$ grid on domain $[-1.5,1.5]\times[-1.5,1.5]$. The time step is $\Delta t=5\Delta x$ and the final time is $t=1$.   Several different orders of angular discretization are considered.

 The results in Figure \ref{isofigure1} show that the numerical solution changes dramatically based on the number  discrete ordinates used.  However, the choice of spatial discretization makes little difference in the qualitative solution.  What does vary is the computational time and memory usage.  These quantities, both real and predicted, are depicted in Figure \ref{Bar2}.  A detailed description of how the predicted values were computed is given in Appendix \ref{App1}. 

%
%
%
\begin{figure}[tbhp]
\centering
\makebox[\linewidth][c]{%
\captionsetup[subfigure]{justification=centering}
\centering
\begin{subfigure}[b]{0.4\textwidth}
\centering
\includegraphics[width=.85\textwidth]{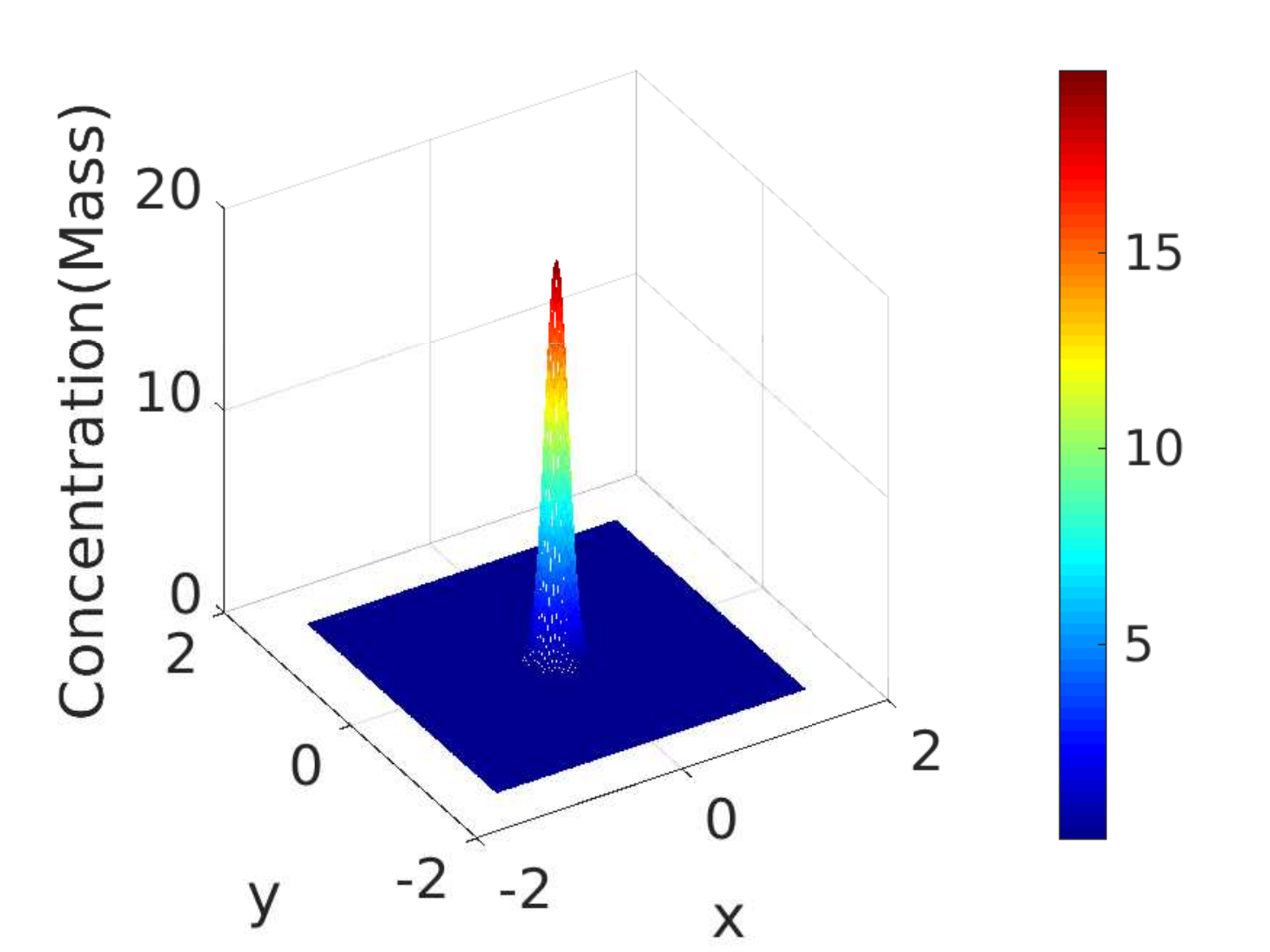}
\caption{Initial Condition}
\end{subfigure}%
\begin{subfigure}[b]{0.4\textwidth}
\centering
\includegraphics[width=.85\textwidth]{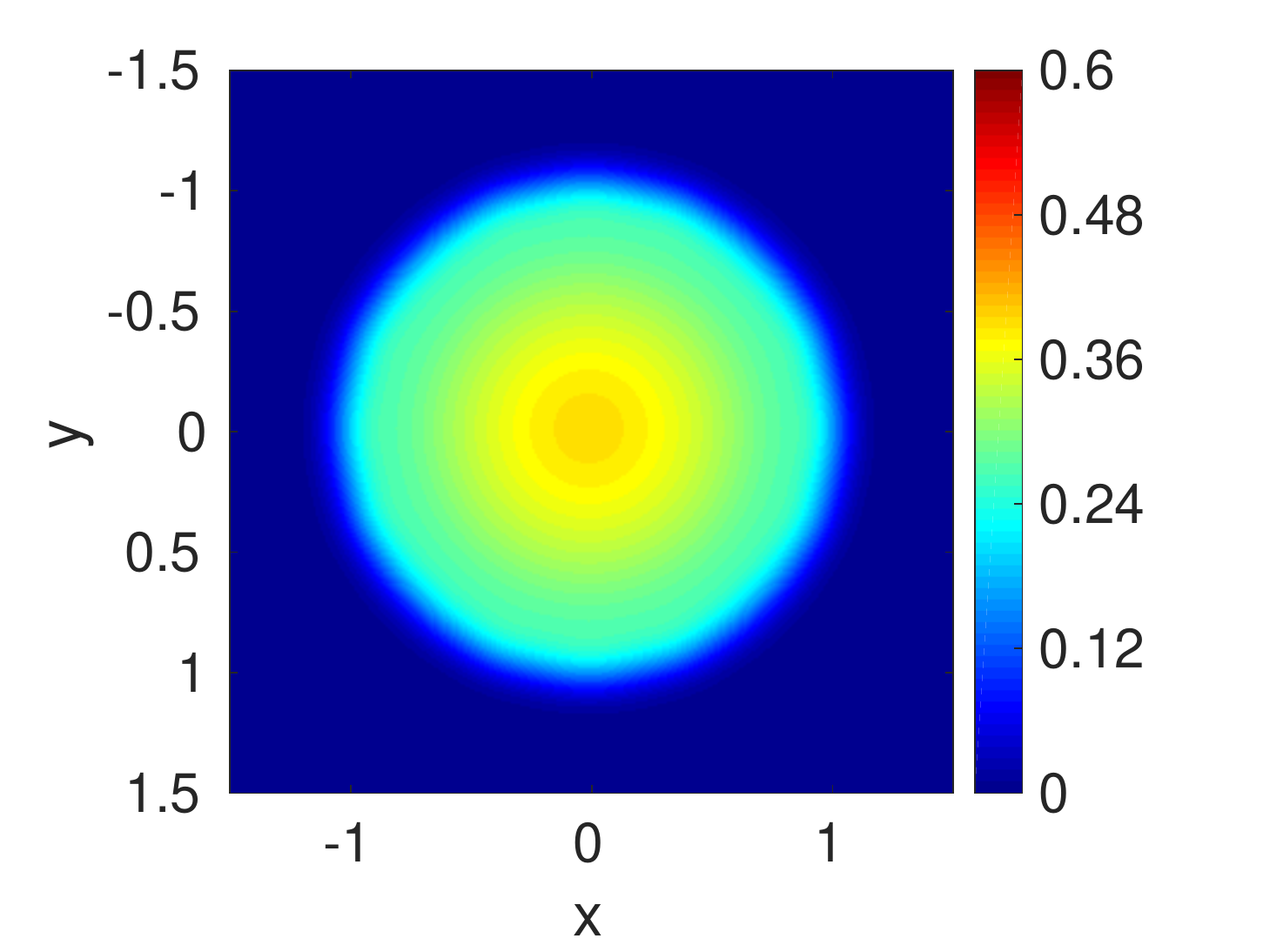}
\caption{Semi-Analytic, $t=1$}
\end{subfigure}%
}\\
\caption{Semi analytic solution for the line source benchmark using initial condition \eqref{eq:LSinicond} with $\beta=0.09$ at time $t=1$.}
\label{isofigure0}
\end{figure}

\begin{figure}[tbhp]
\centering
\makebox[\linewidth][c]{%
\captionsetup[subfigure]{justification=centering}
\centering
\begin{subfigure}[b]{0.25\textwidth}
\centering
\includegraphics[width=1\textwidth]{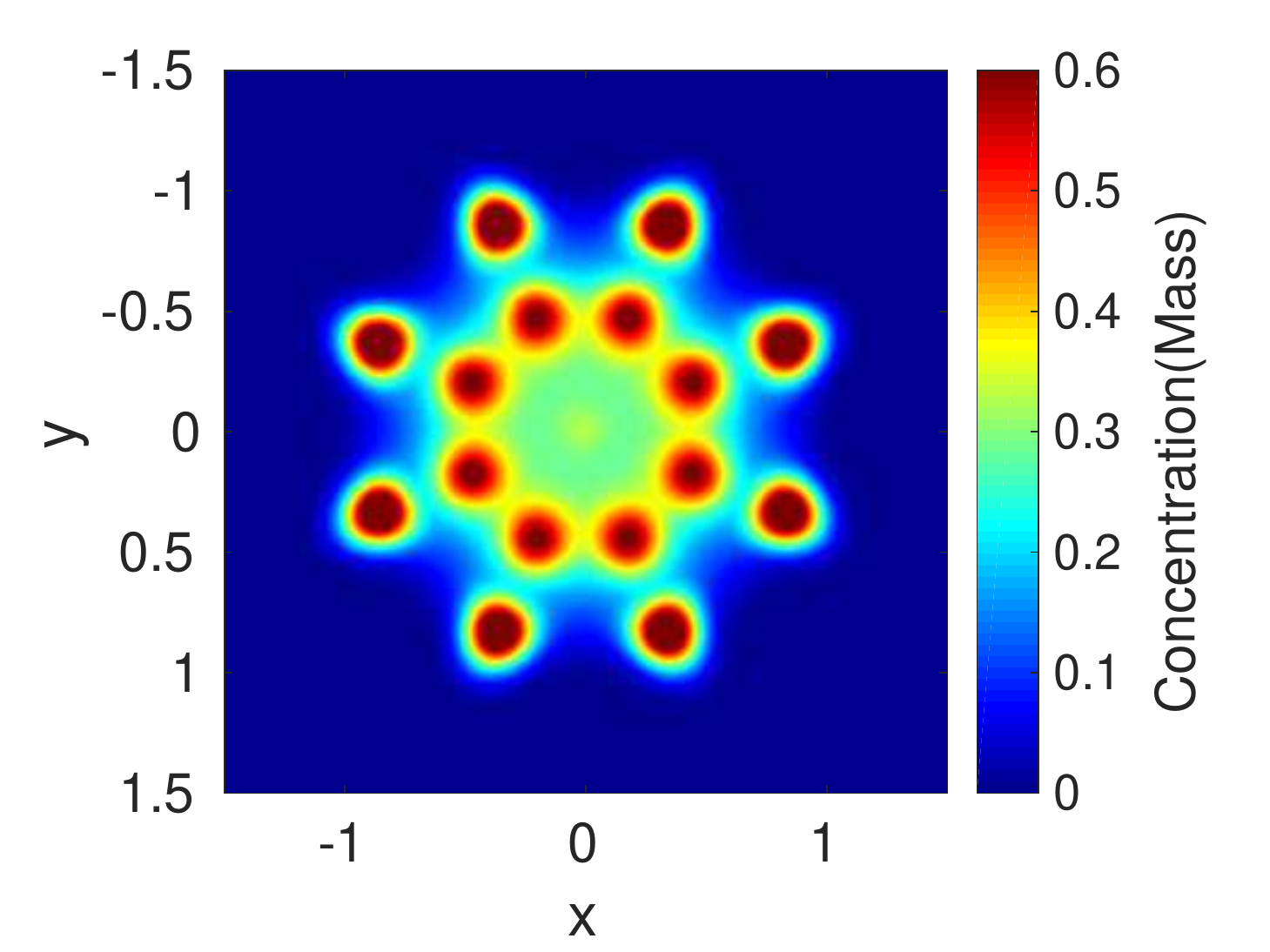}
\caption{FV $\nsol{S}_{4}$,}
\end{subfigure}%
\begin{subfigure}[b]{0.25\textwidth}
\centering
\includegraphics[width=1\textwidth]{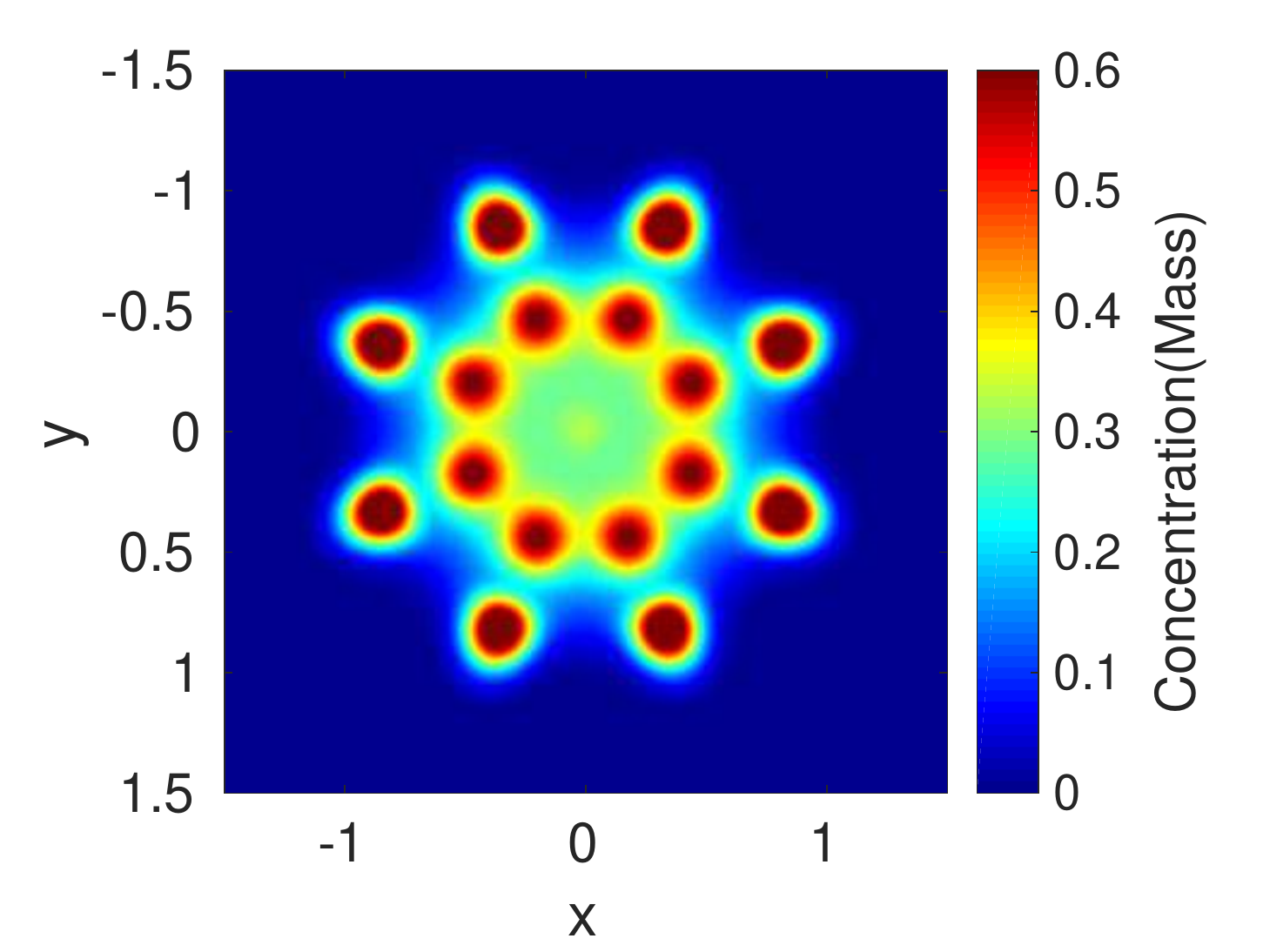}
\caption{DG $\nsol{S}_{4}$.}
\end{subfigure}%
\begin{subfigure}[b]{0.25\textwidth}
\centering
\includegraphics[width=1\textwidth]{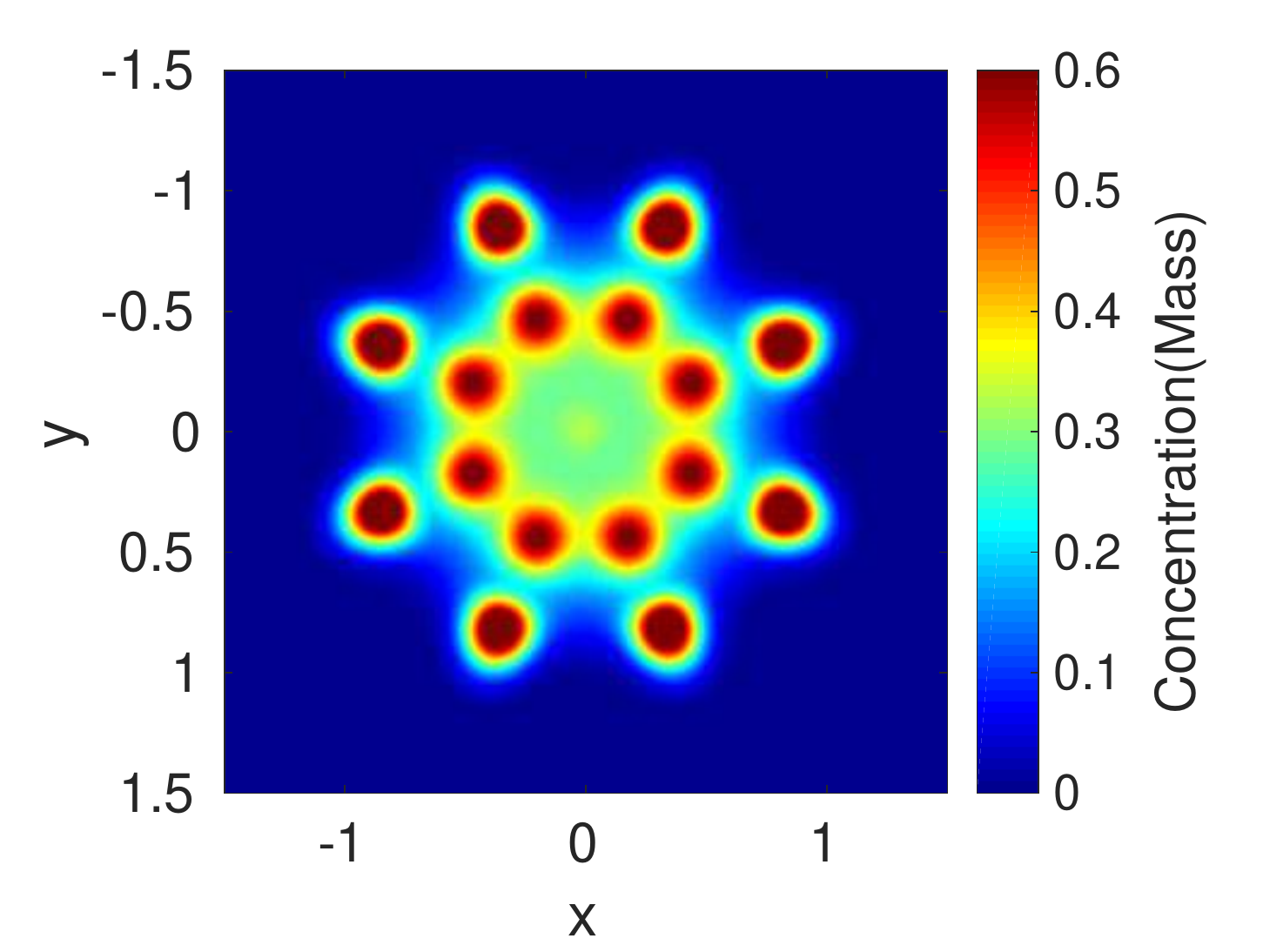}
\caption{DG-DG $\nsol{S}_{4}\nsol{S}_4$}
\end{subfigure}%
\begin{subfigure}[b]{0.25\textwidth}
\centering
\includegraphics[width=1\textwidth]{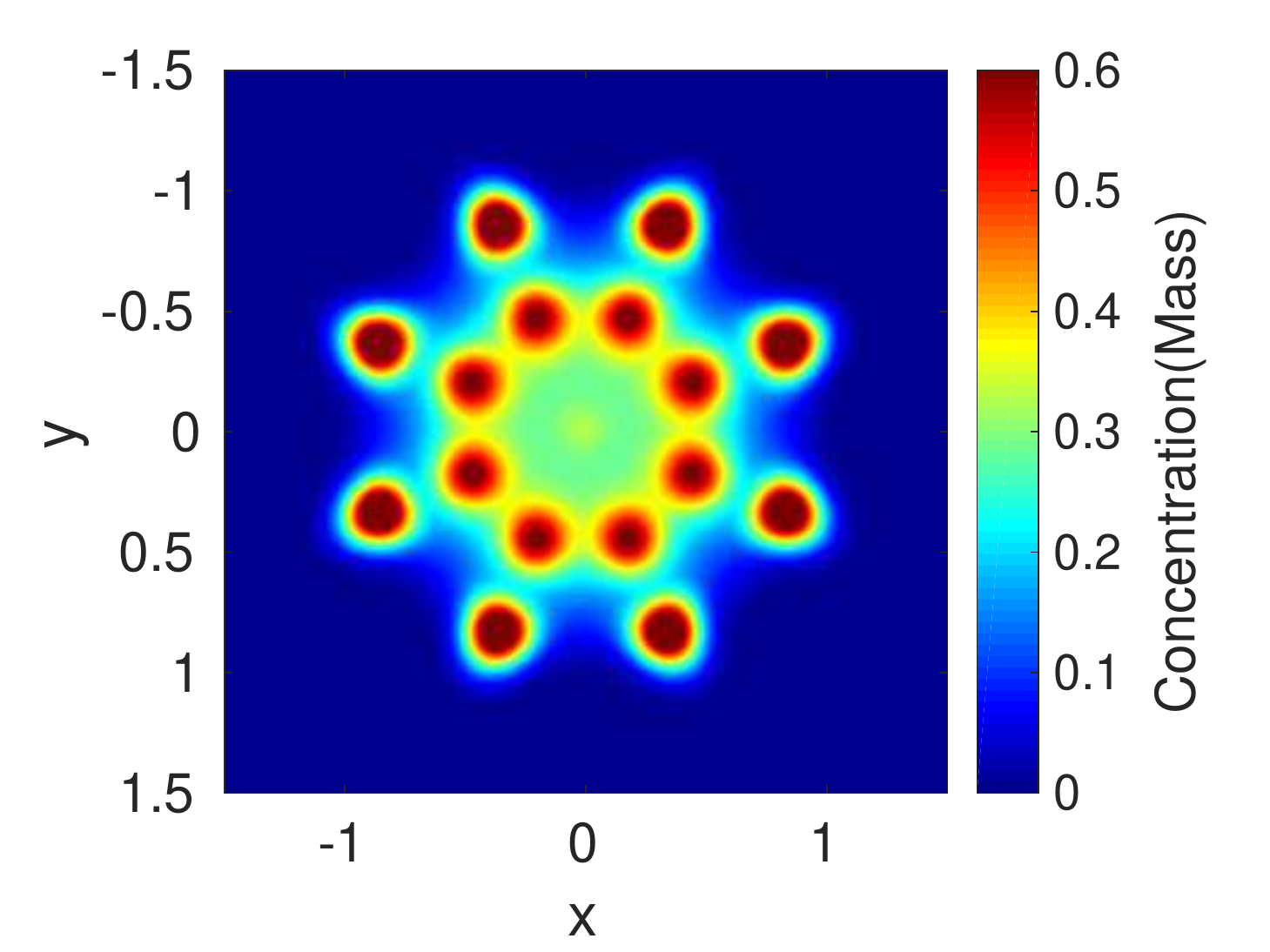}
\caption{FV-DG $\nsol{S}_{4}\nsol{S}_4$}
\end{subfigure}%
}\\
\makebox[\linewidth][c]{%
\captionsetup[subfigure]{justification=centering}
\centering
\begin{subfigure}[b]{0.25\textwidth}
\centering
\includegraphics[width=1\textwidth]{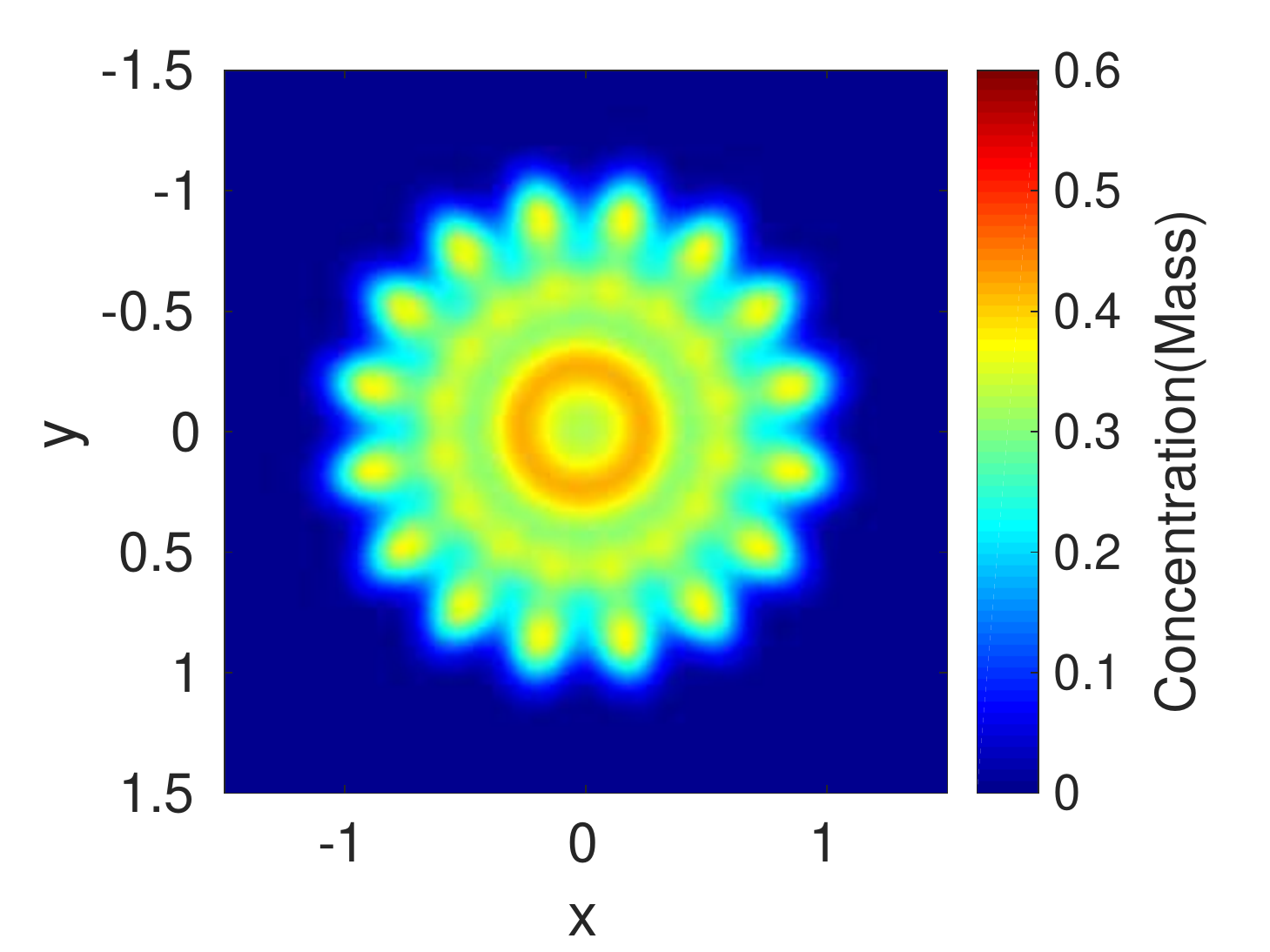}
\caption{FV $\nsol{S}_{8}$}
\end{subfigure}%
\begin{subfigure}[b]{0.25\textwidth}
\centering
\includegraphics[width=1\textwidth]{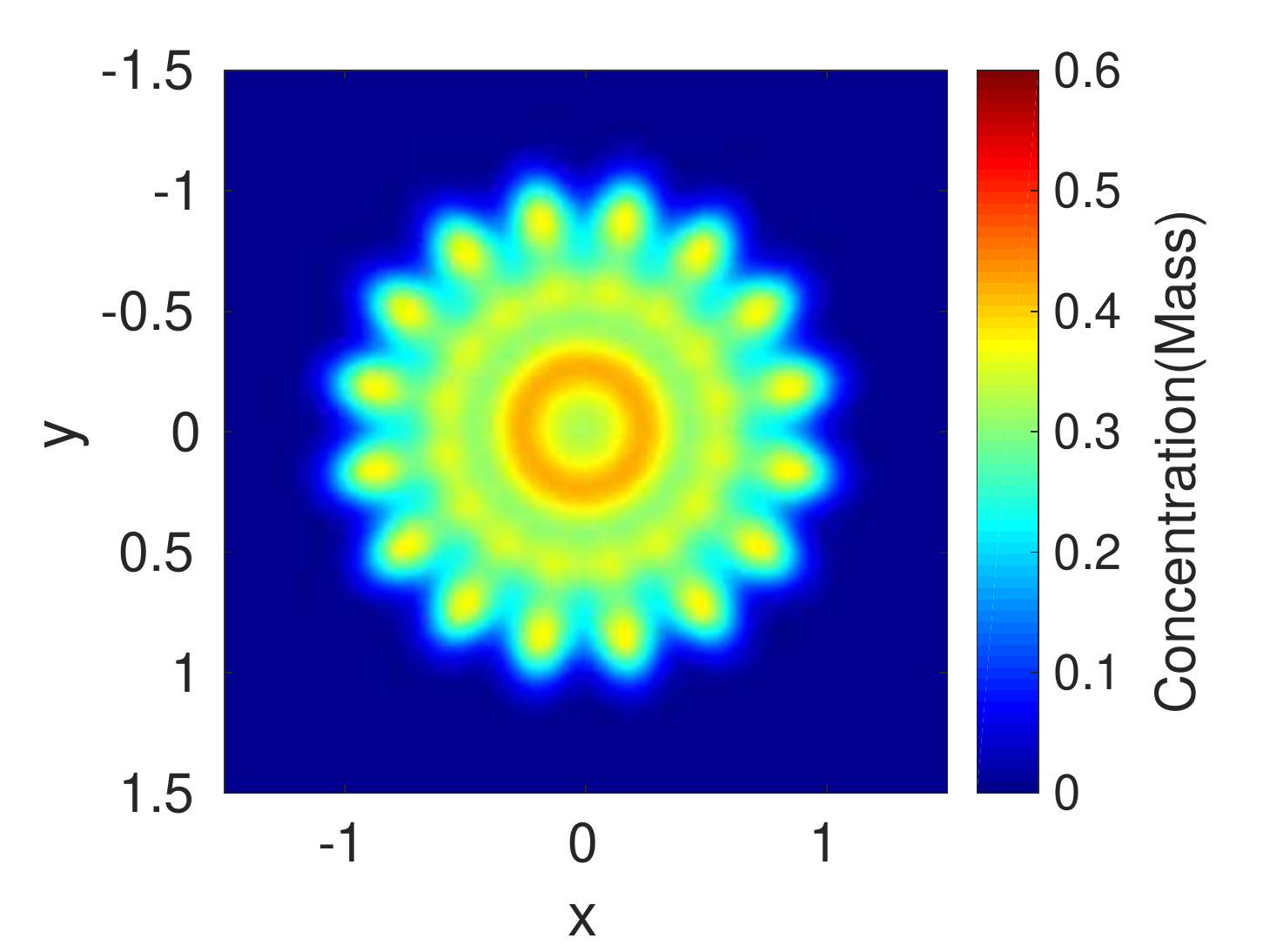}
\caption{DG $\nsol{S}_{8}$}
\end{subfigure}%
\begin{subfigure}[b]{0.25\textwidth}
\centering
\includegraphics[width=1\textwidth]{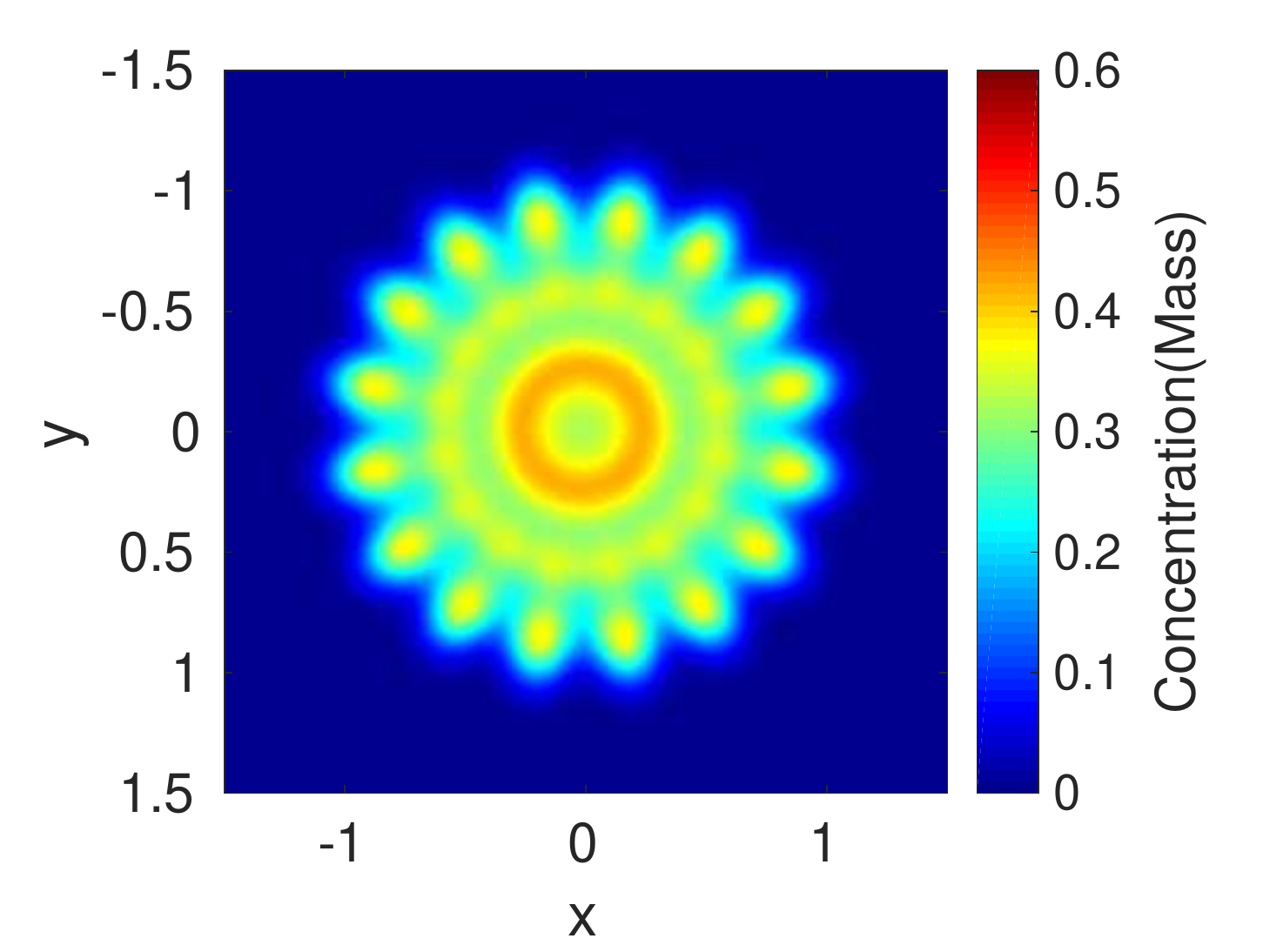}
\caption{DG-DG $\nsol{S}_{8}\nsol{S}_4$}
\end{subfigure}%
\begin{subfigure}[b]{0.25\textwidth}
\centering
\includegraphics[width=1\textwidth]{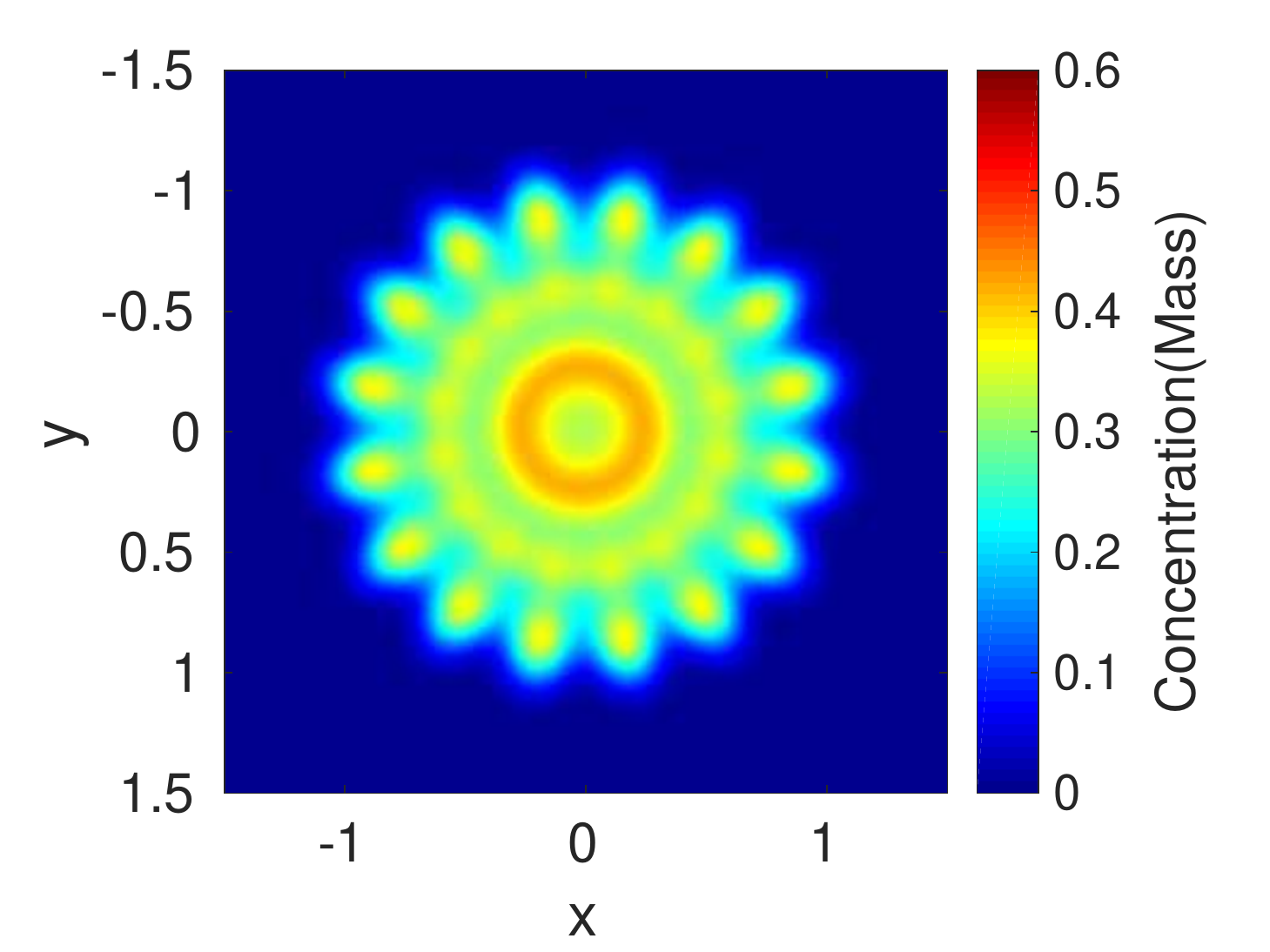}
\caption{FV-DG $\nsol{S}_{8}\nsol{S}_4$, }
\end{subfigure}%
}\\
\makebox[\linewidth][c]{%
\captionsetup[subfigure]{justification=centering}
\centering
\begin{subfigure}[b]{0.25\textwidth}
\centering
\includegraphics[width=1\textwidth]{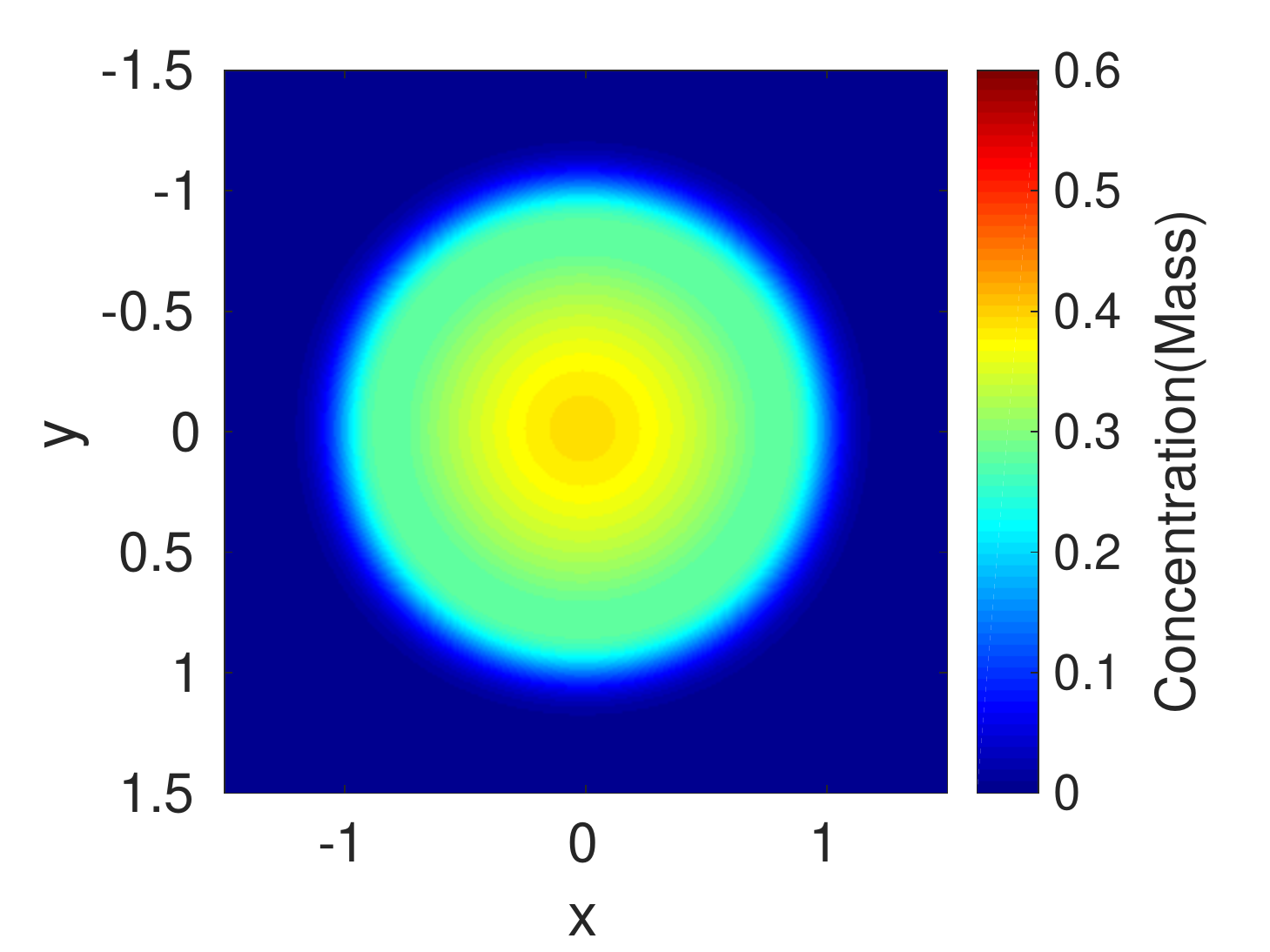}
\caption{FV $\nsol{S}_{24}$}
\end{subfigure}%
\begin{subfigure}[b]{0.25\textwidth}
\centering
\includegraphics[width=1\textwidth]{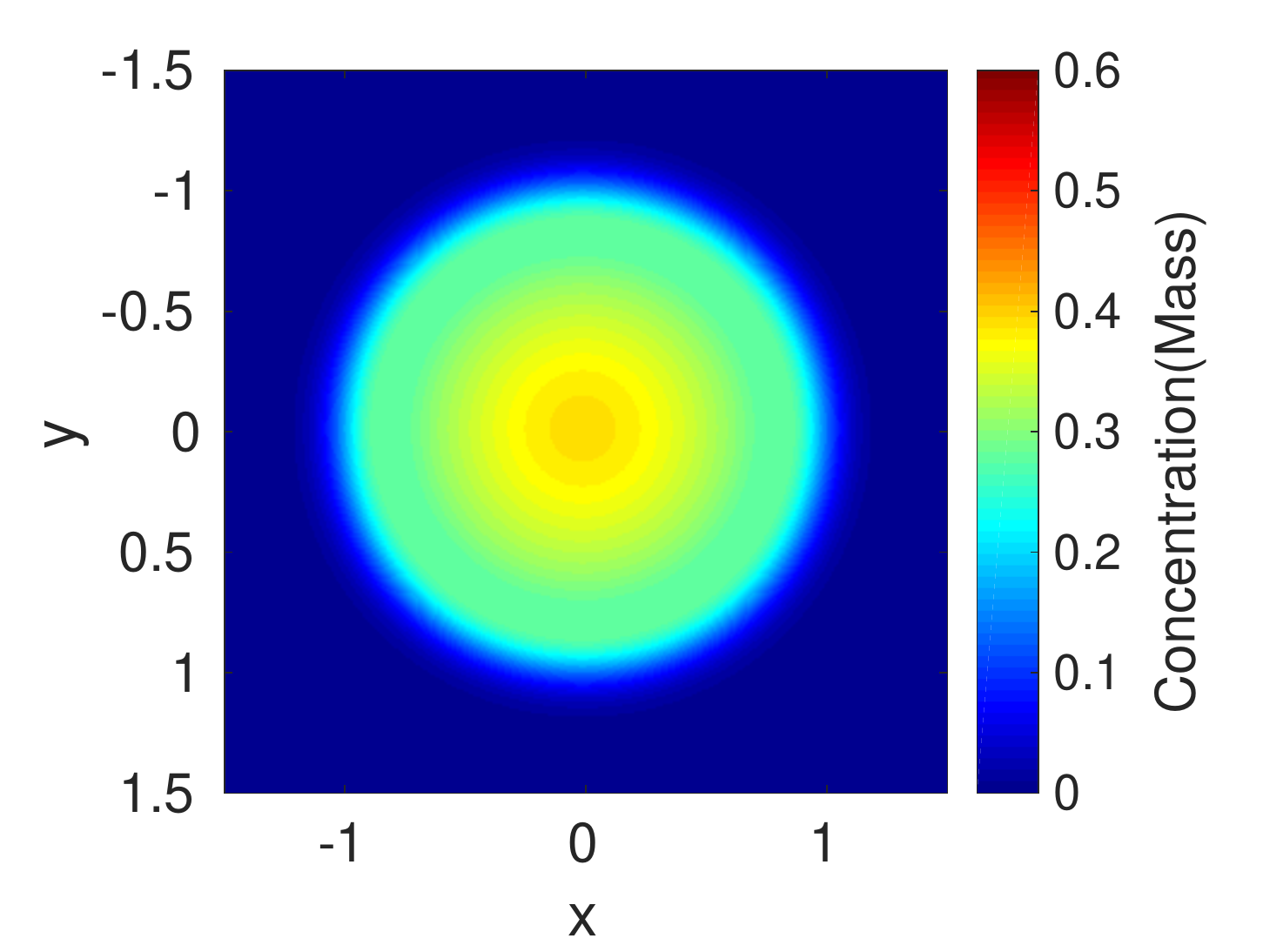}
\caption{DG $\nsol{S}_{24}$, }
\end{subfigure}%
\begin{subfigure}[b]{0.25\textwidth}
\centering
\includegraphics[width=1\textwidth]{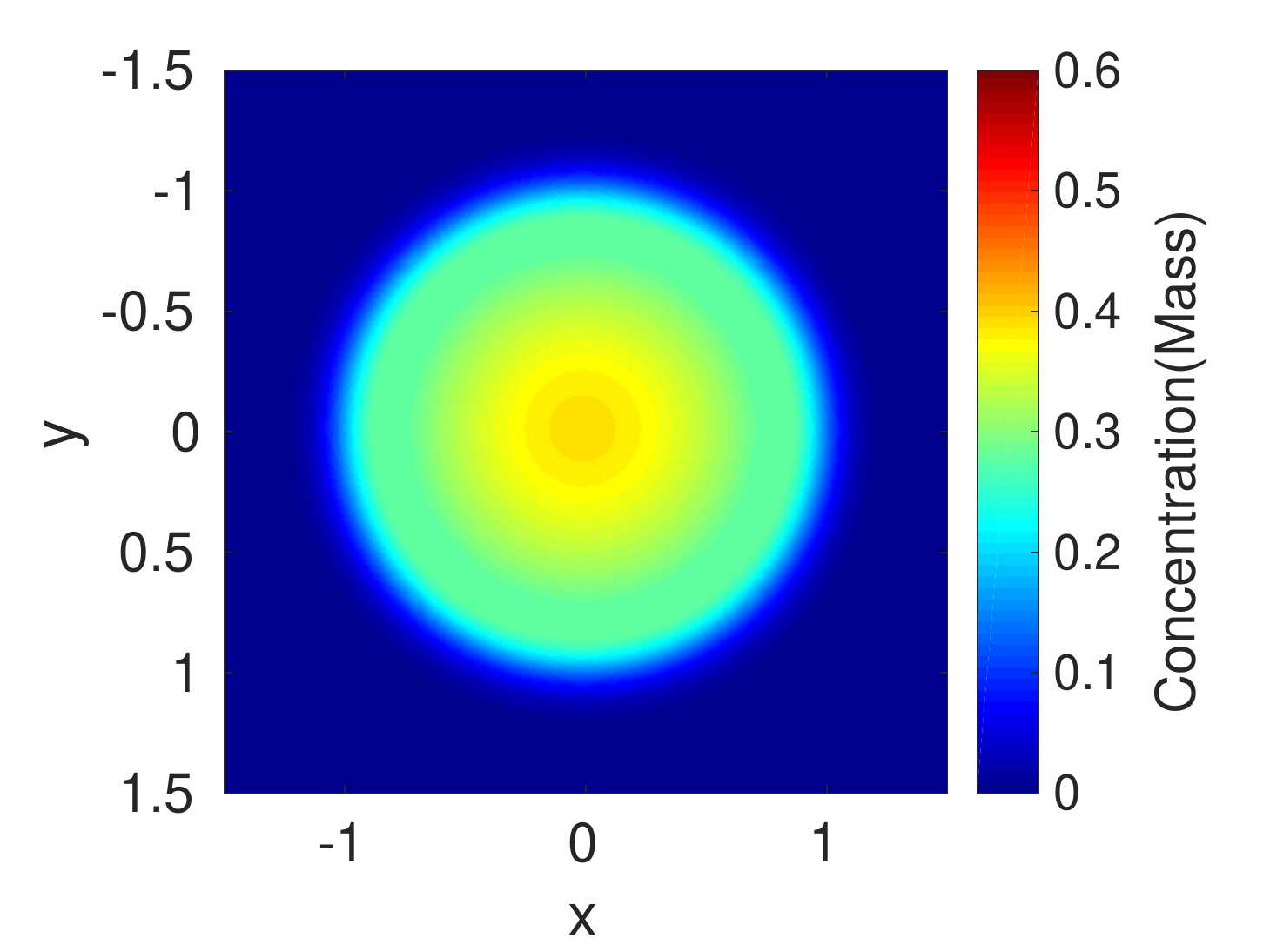}
\caption{DG-DG $\nsol{S}_{24}\nsol{S}_4$,}
\end{subfigure}%
\begin{subfigure}[b]{0.25\textwidth}
\centering
\includegraphics[width=1\textwidth]{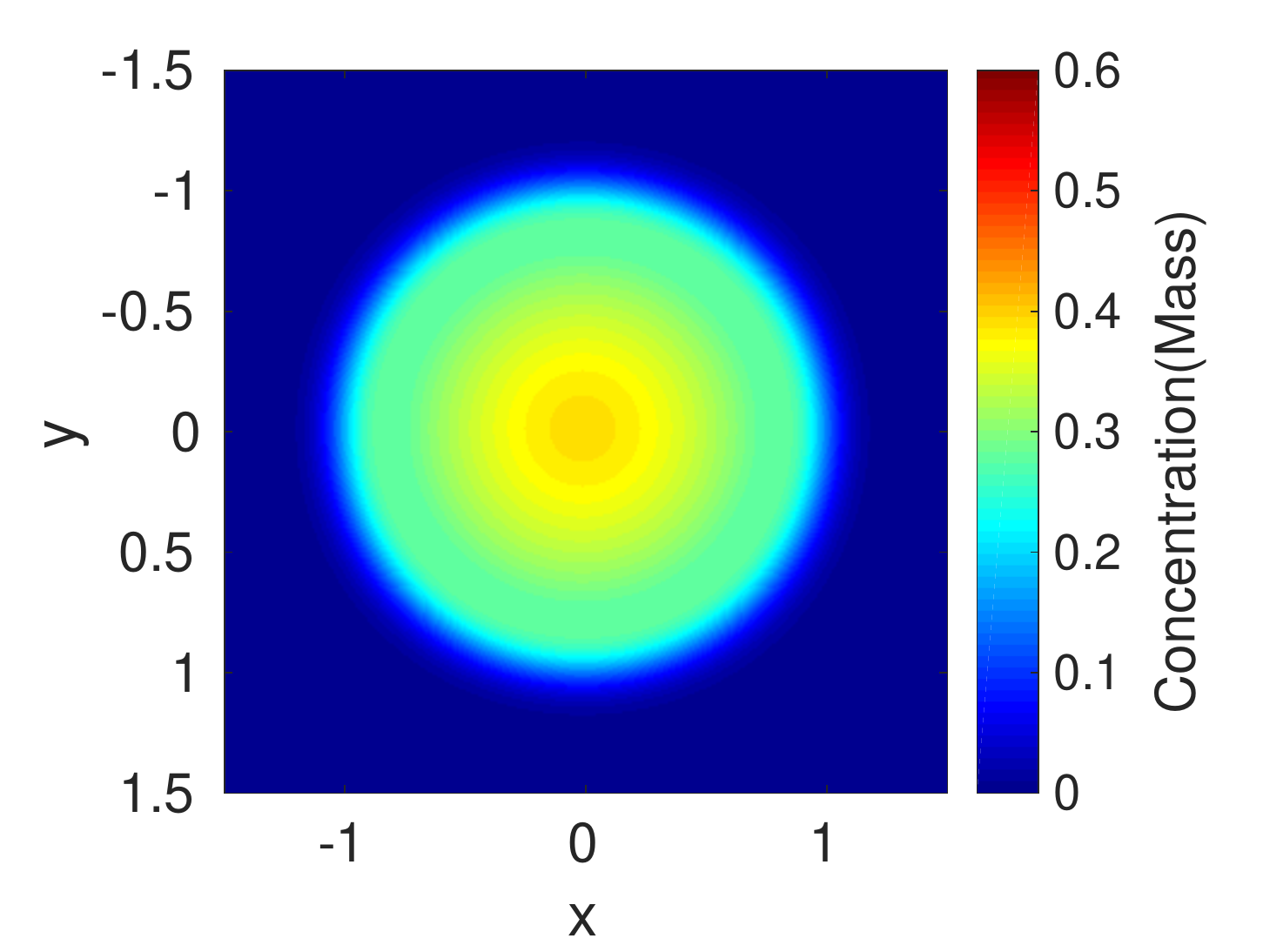}
\caption{FV-DG $\nsol{S}_{24}\nsol{S}_4$}
\end{subfigure}%
}\\
\caption{Solutions of particle concentration to the line source problem with initial condition \eqref{eq:LSinicond}. Simulations are run on a $301\times301$ grid to $t=1$ with time step $\Delta t = 5\Delta x$.  Occurrence of ray-effects is strongly related to the angular resolution of the uncollided equations. Qualitative results of each method are similar, regardless of the angular resolution of the collided equations in the hybrid methods. }
\label{isofigure1}
\end{figure}

\begin{figure}[tbhp]
\centering
\makebox[\linewidth][c]{%
\captionsetup[subfigure]{justification=centering}
\centering
\begin{subfigure}[b]{0.33\textwidth}
\centering
\includegraphics[width=1\textwidth]{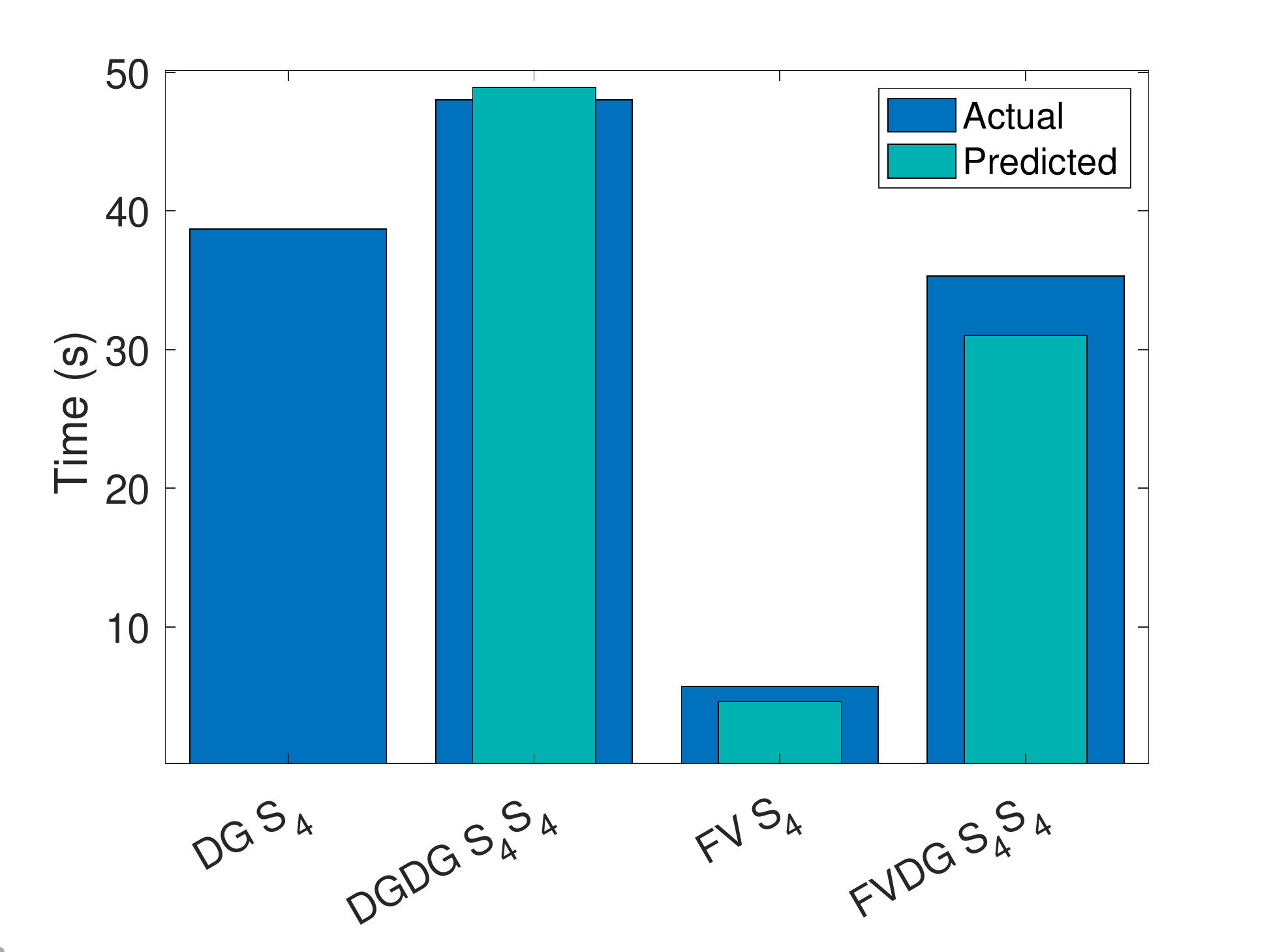}
\end{subfigure}%
\begin{subfigure}[b]{0.33\textwidth}
\centering
\includegraphics[width=1\textwidth]{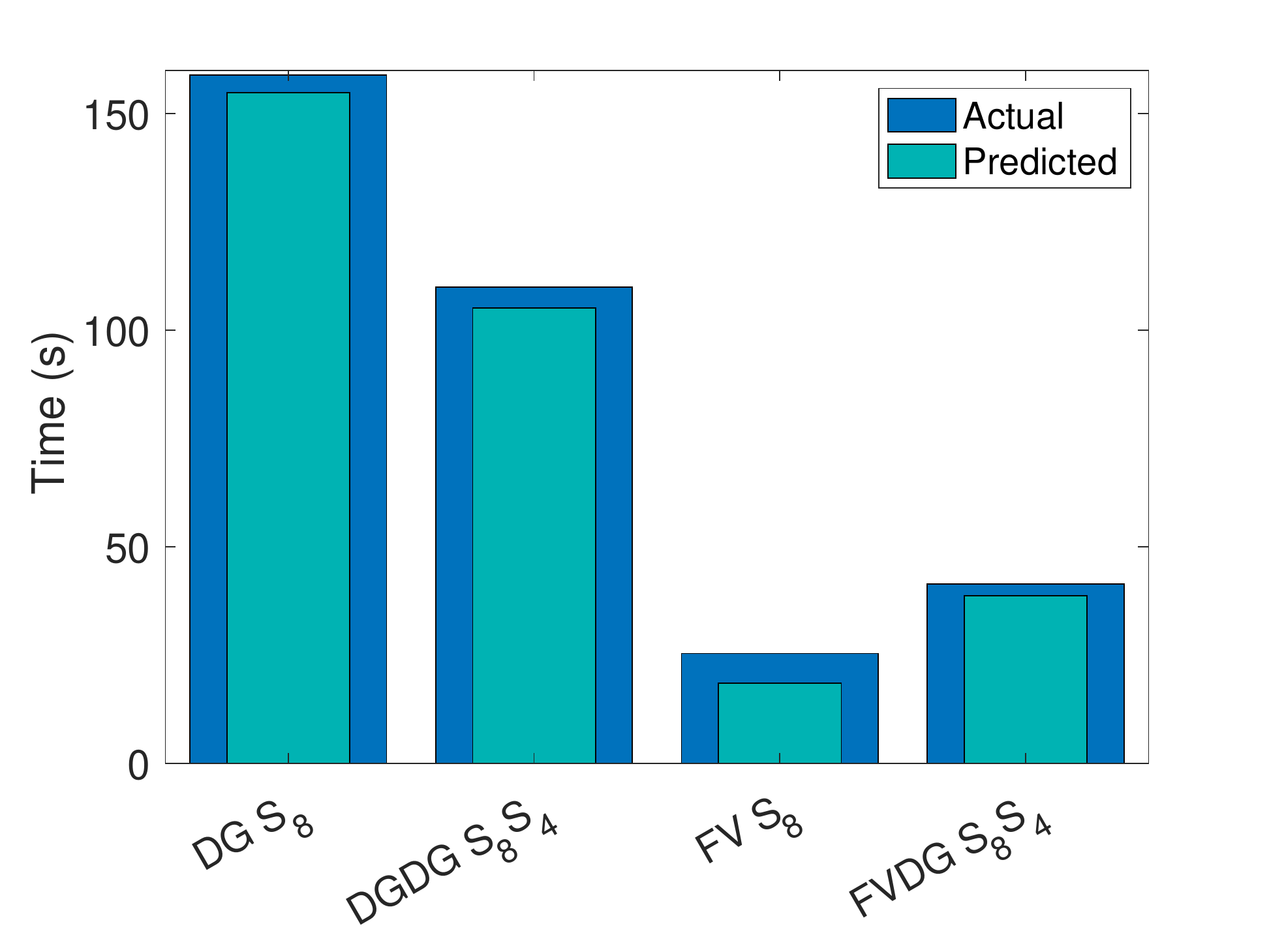}
\end{subfigure}%
\begin{subfigure}[b]{0.33\textwidth}
\centering
\includegraphics[width=1\textwidth]{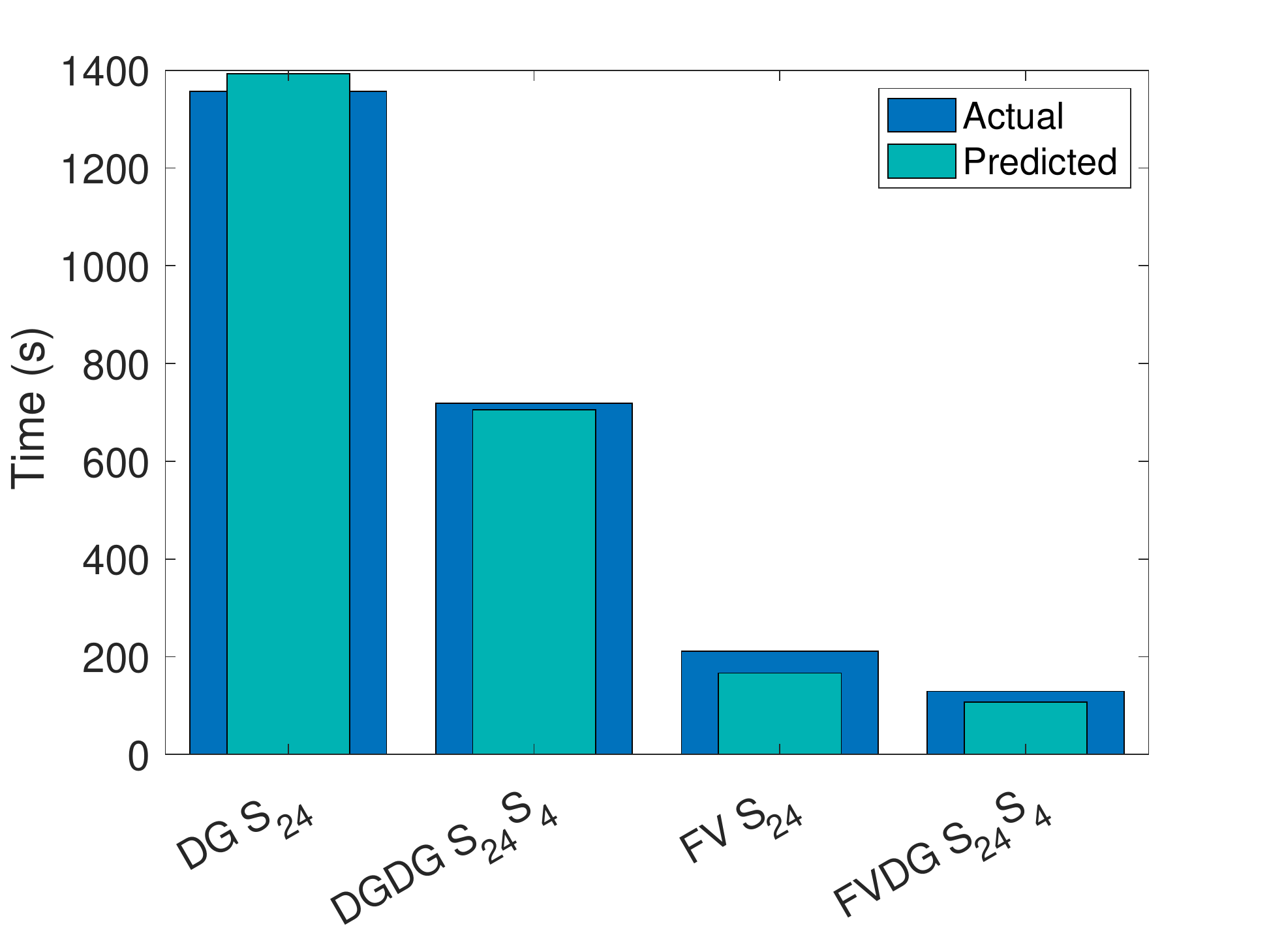}
\end{subfigure}%
}\\
\makebox[\linewidth][c]{%
\captionsetup[subfigure]{justification=centering}
\centering
\begin{subfigure}[b]{0.33\textwidth}
\centering
\includegraphics[width=1\textwidth]{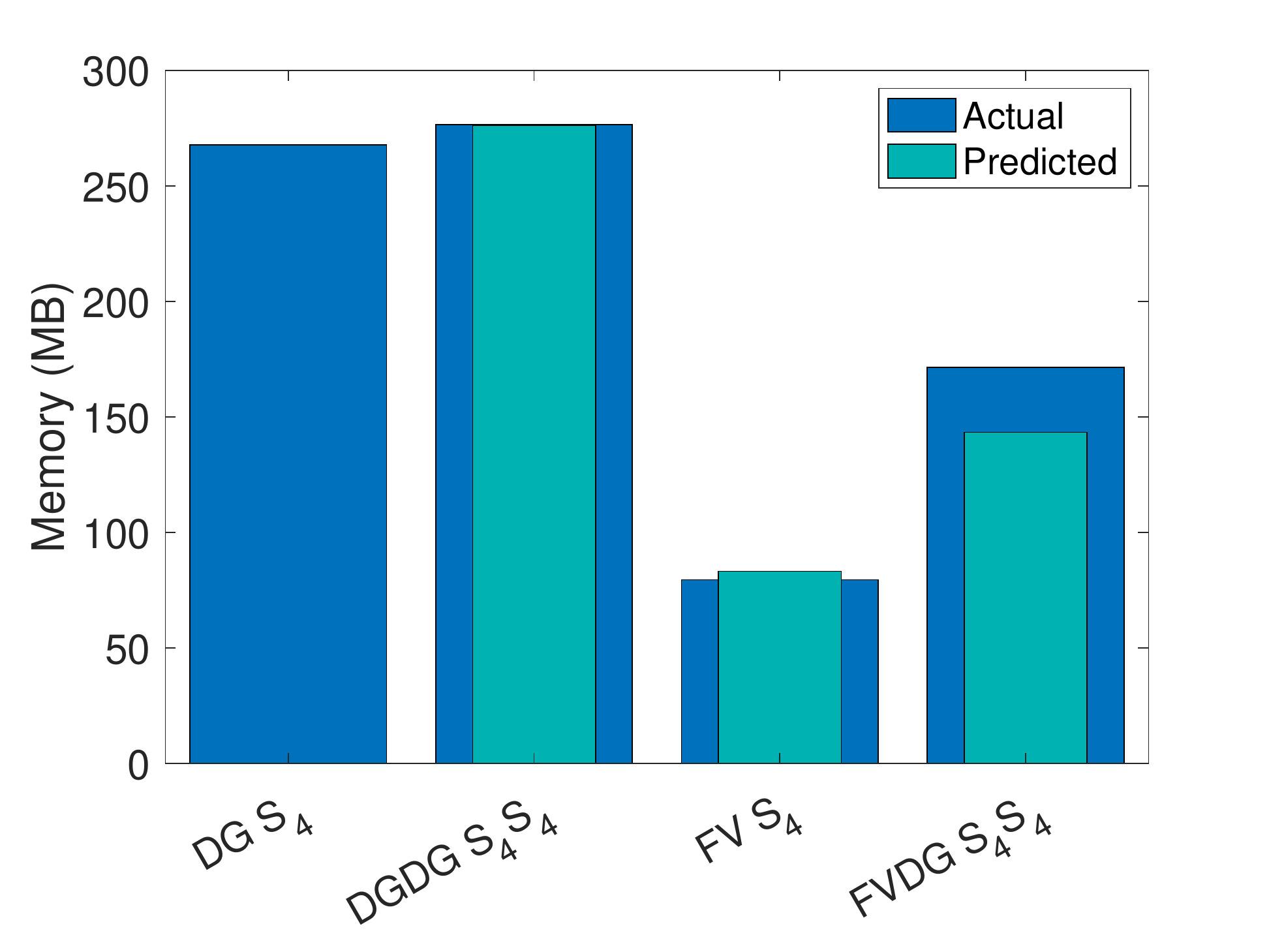}
\end{subfigure}%
\begin{subfigure}[b]{0.33\textwidth}
\centering
\includegraphics[width=1\textwidth]{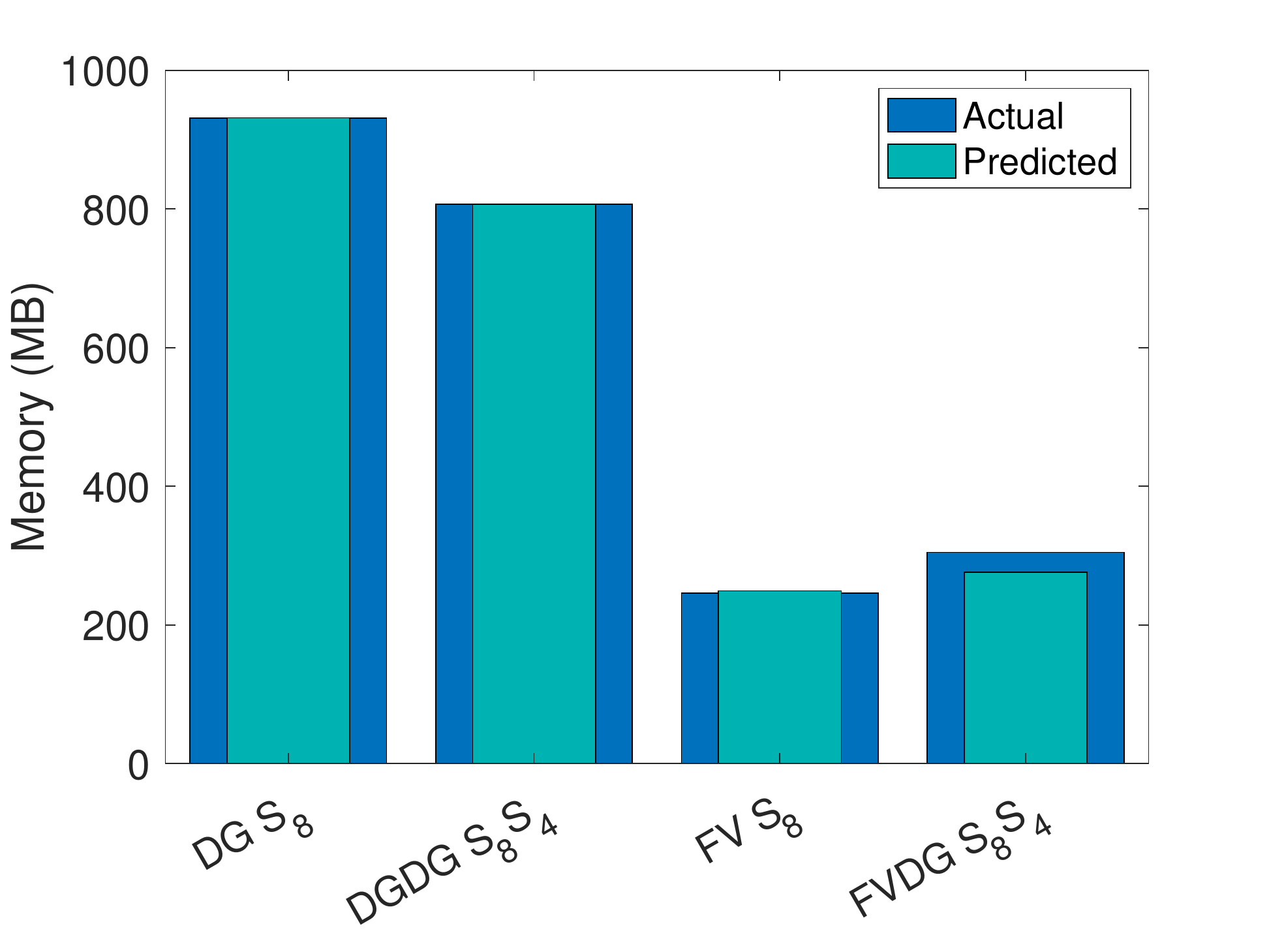}
\end{subfigure}%
\begin{subfigure}[b]{0.33\textwidth}
\centering
\includegraphics[width=1\textwidth]{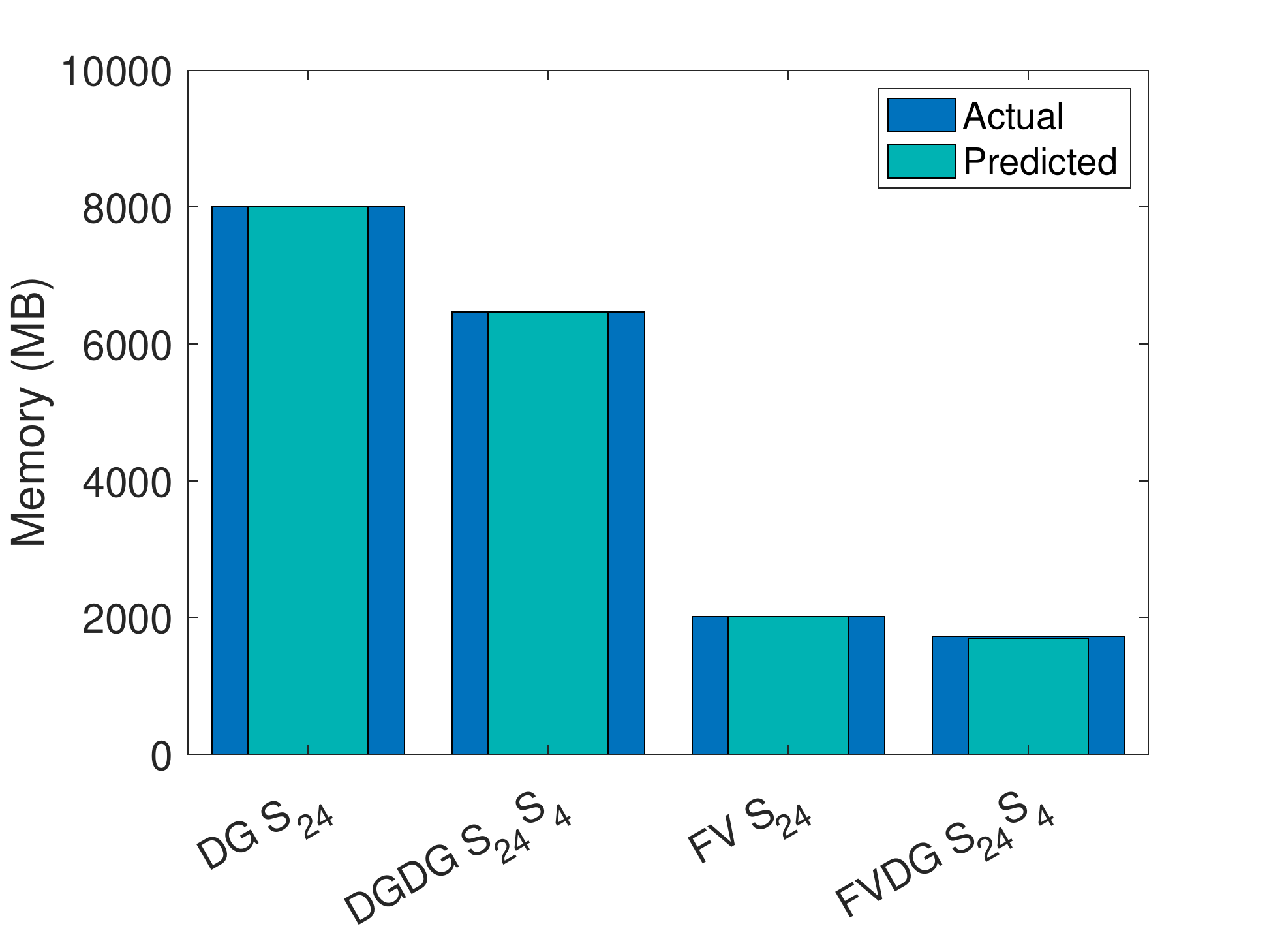}
\end{subfigure}%
}\\
\caption{CPU timings in seconds and maximum memory usage in megabytes for Figure \ref{isofigure1}.  
Wide blue bars represent the actual (measured) quantities.  An external script \cite{memusg} was used to query the computers memory usage. Thin teal bars were generated based on the scaling arguments, code inspection, and calibration using the  DG $\nsol{S}_4$ simulation.  
See Appendix \ref{App1} for details.
}
\label{Bar2}
\end{figure}

\subsubsection{Example 2}\label{sec:AccThin}
In this example, we demonstrate how the computational advantages of the spatial hybrid shown in Figure \ref{Bar2} can be leveraged to produce more accurate solutions in less run time.  The numerical parameters used are the same as in Section \ref{sec:Eff}, except that $\beta=0.045$ in \eqref{eq:LSinicond} and $\Delta t=3\Delta x$. 

It was observed in \cite{Crockatt,CrockattDis} that hybridization in angle can achieve more accurate numerical solutions than a standard approach by simply increasing the angular resolution in the uncollided equation while reducing the resolution in the collided equation.   The results in Figure \ref{isofigure1} are consistent with this observation, and for a standard DG discretization, the results of such a strategy are shown in the middle two columns of Table \ref{table5}.  While the angular hybrid reduces error by roughly a factor of three, it also (in this case) increases run time by a factor of two.   However,  the run time (and the memory footprint) can be reduced by introducing hybrization in space.  Indeed, the results in the far right column of Table \ref{table5} show that additional hybridization in space significantly decreases the run time of the angular hybrid, thereby producing a more accurate answer in less time when compared to the base method.

\begin{figure}[tbhp]
\centering
\makebox[\linewidth][c]{%
\captionsetup[subfigure]{justification=centering}
\centering
\begin{subfigure}[b]{0.25\textwidth}
\centering
\includegraphics[width=1\textwidth]{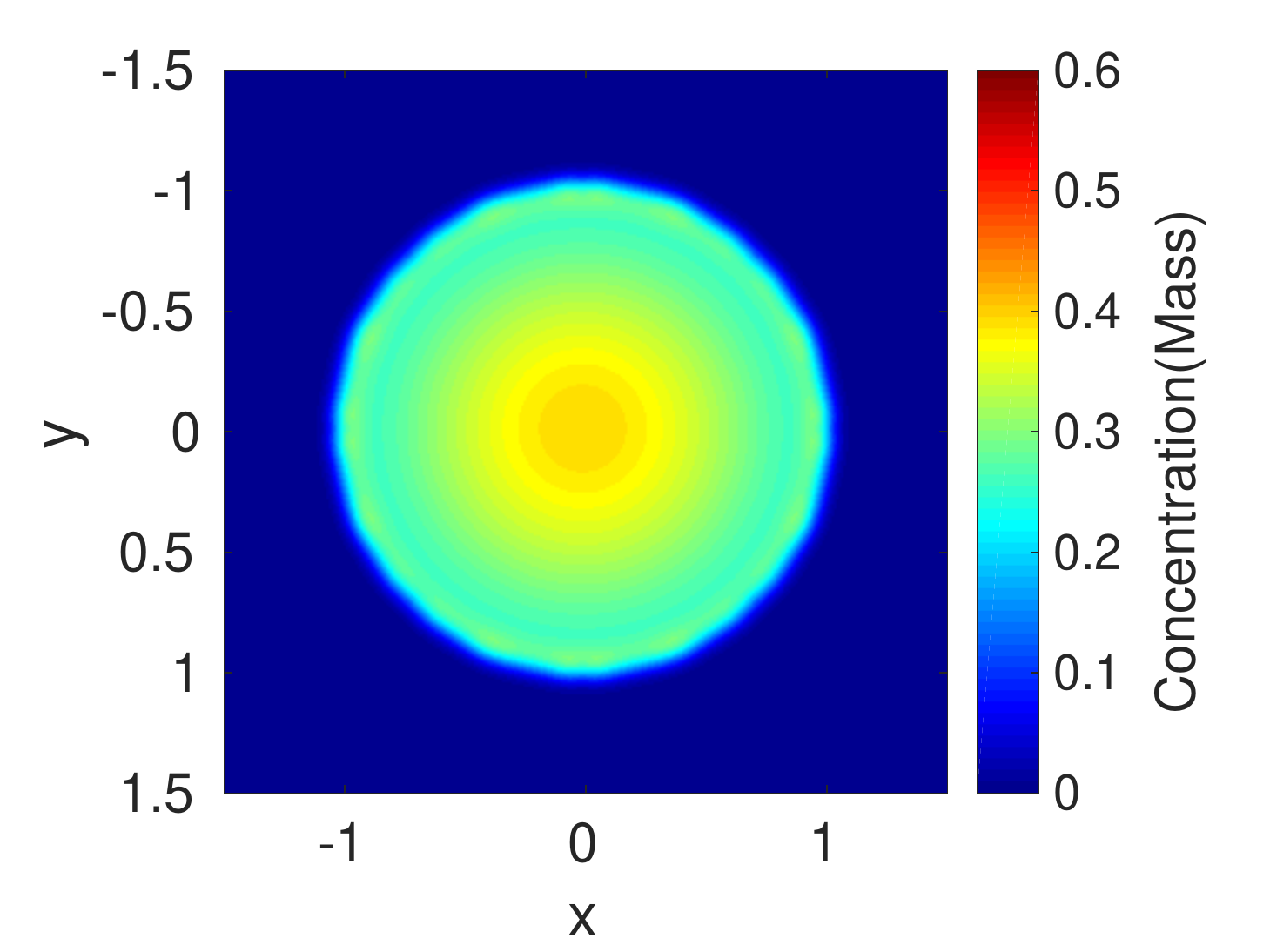}
\caption{Semi-Analytic}
\end{subfigure}%
\begin{subfigure}[b]{0.25\textwidth}
\centering
\includegraphics[width=1\textwidth]{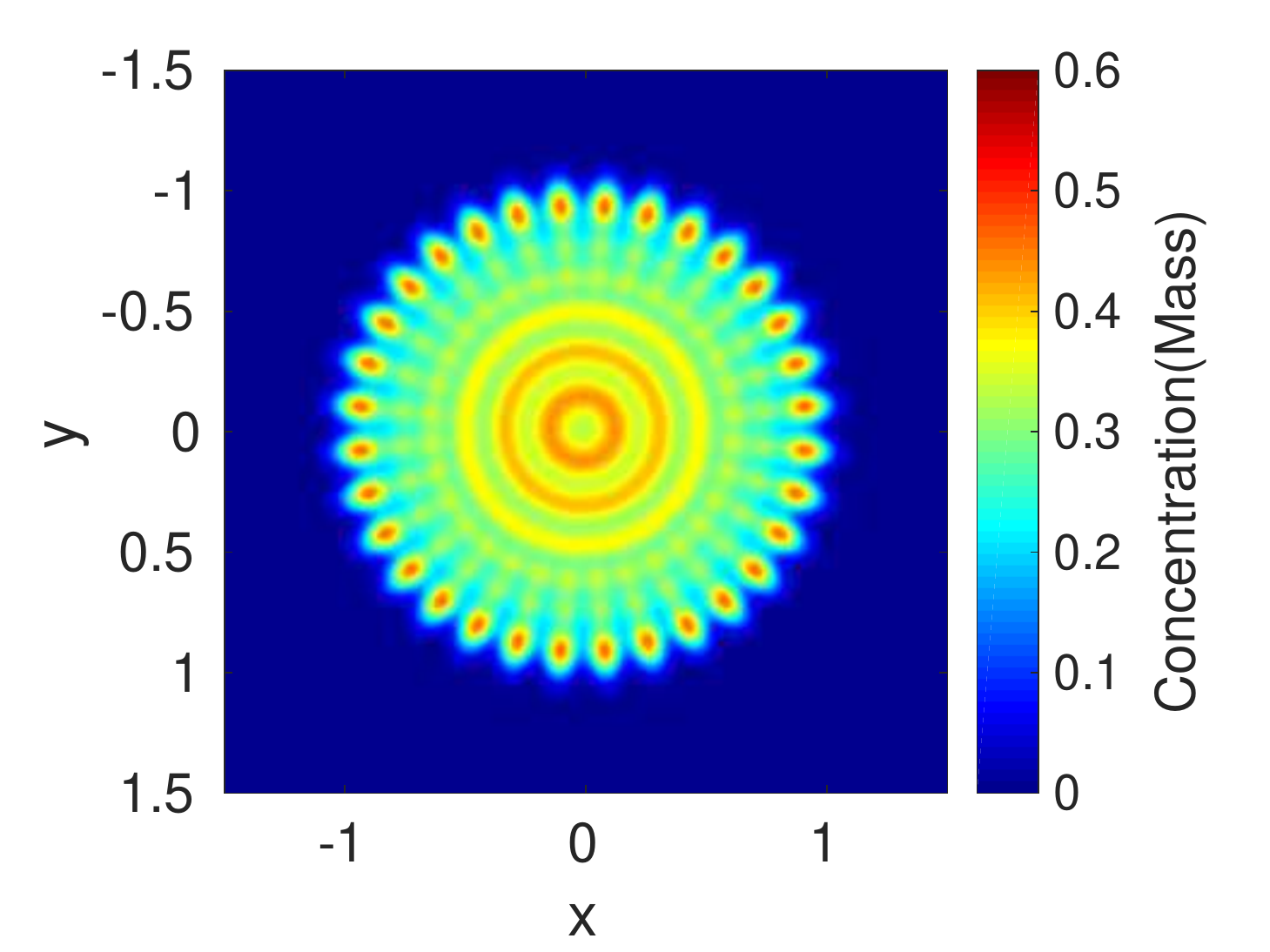}
\caption{DG $\nsol{S}_{16}$}
\end{subfigure}%
\begin{subfigure}[b]{0.25\textwidth}
\centering
\includegraphics[width=1\textwidth]{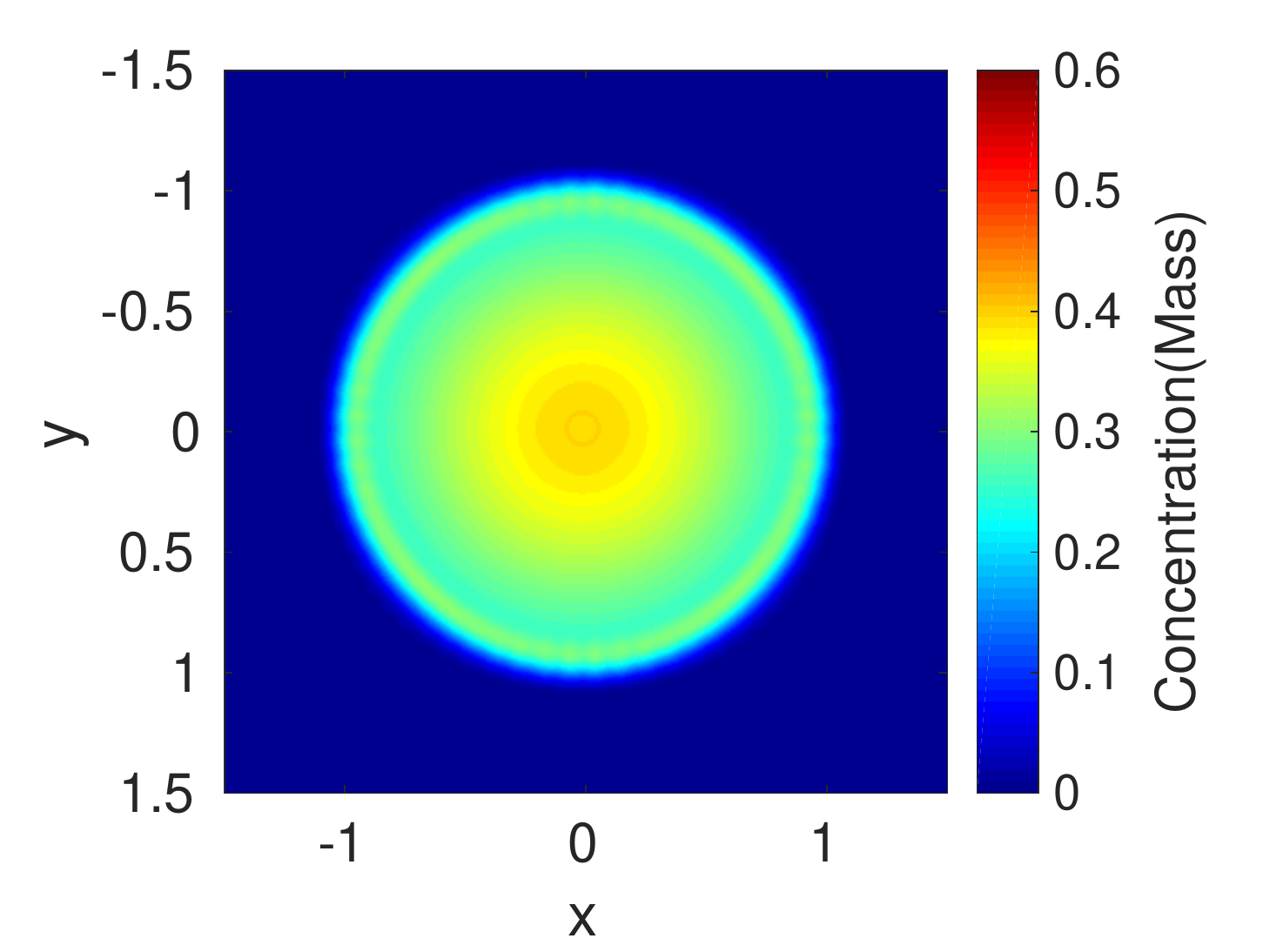}
\caption{FV-DG $\nsol{S}_{32}\nsol{S}_4$}
\end{subfigure}%
\begin{subfigure}[b]{0.25\textwidth}
\centering
\includegraphics[width=1\textwidth]{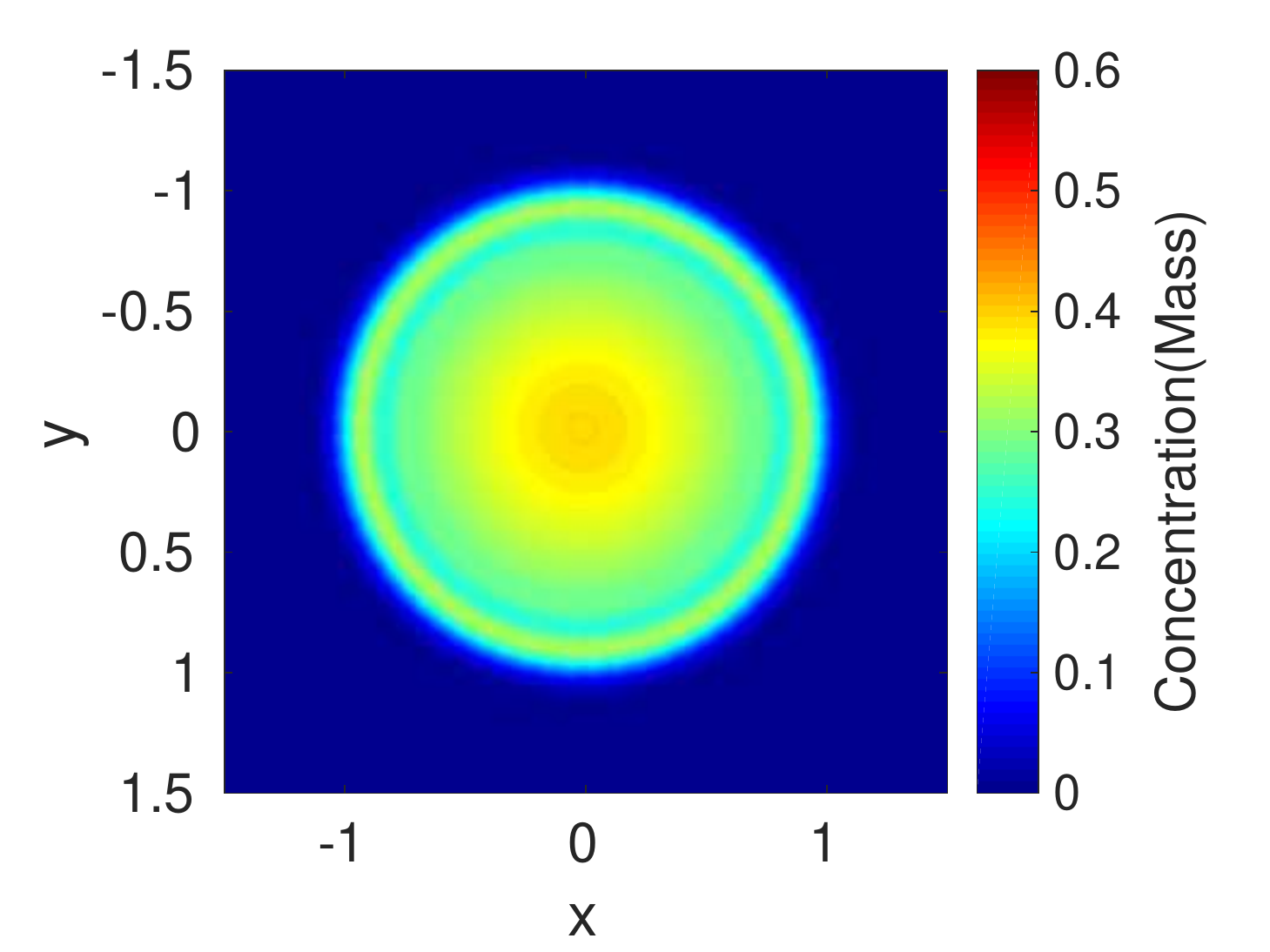}
\caption{DG-DG $\nsol{S}_{32}\nsol{S}_4$}
\end{subfigure}%
}\\
\makebox[\linewidth][c]{%
\captionsetup[subfigure]{justification=centering}
\centering
\begin{subfigure}[b]{0.33\textwidth}
\centering
\includegraphics[width=1\textwidth]{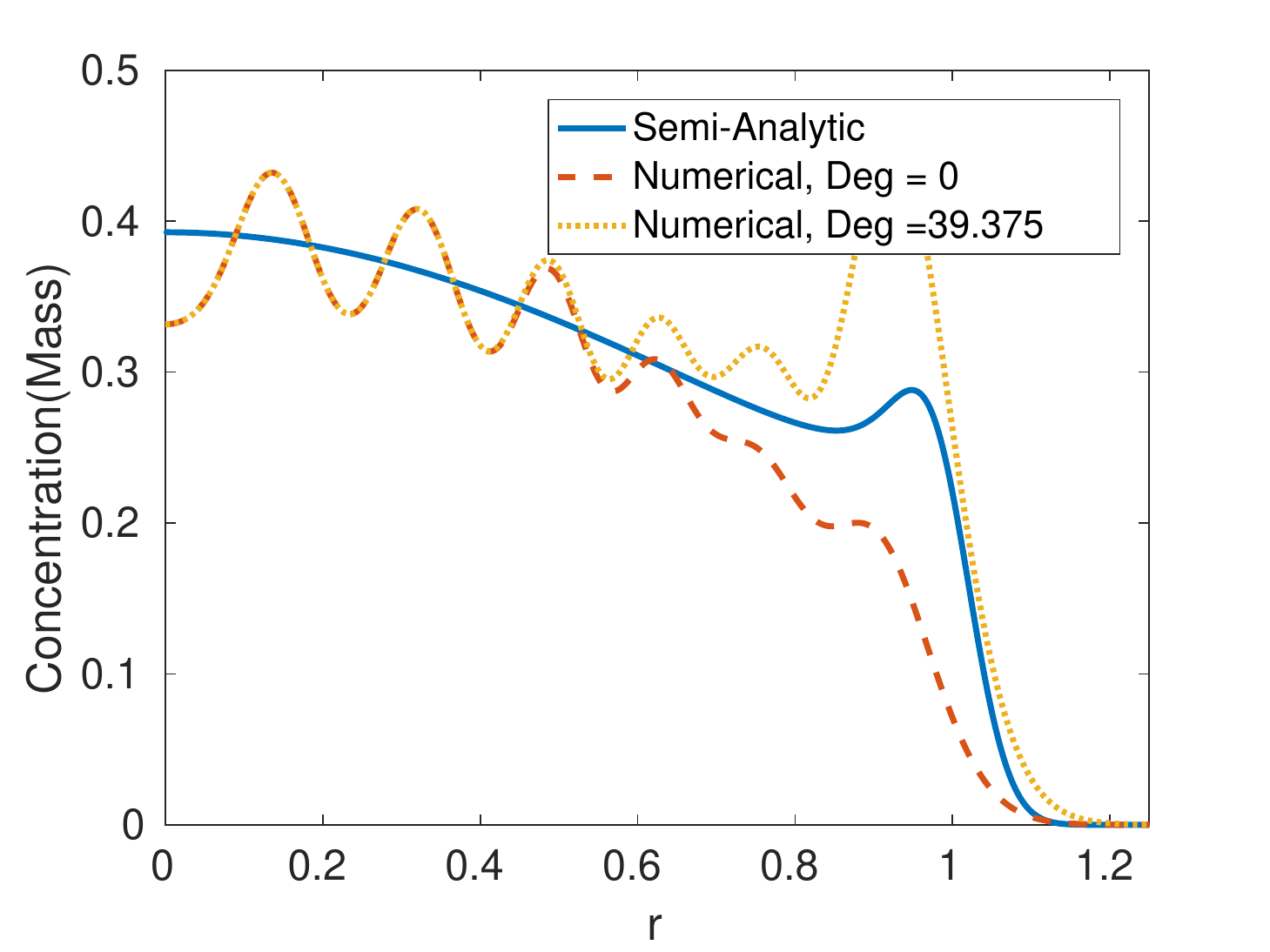}
\caption{DG $\nsol{S}_{16}$}
\end{subfigure}%
\begin{subfigure}[b]{0.33\textwidth}
\centering
\includegraphics[width=1\textwidth]{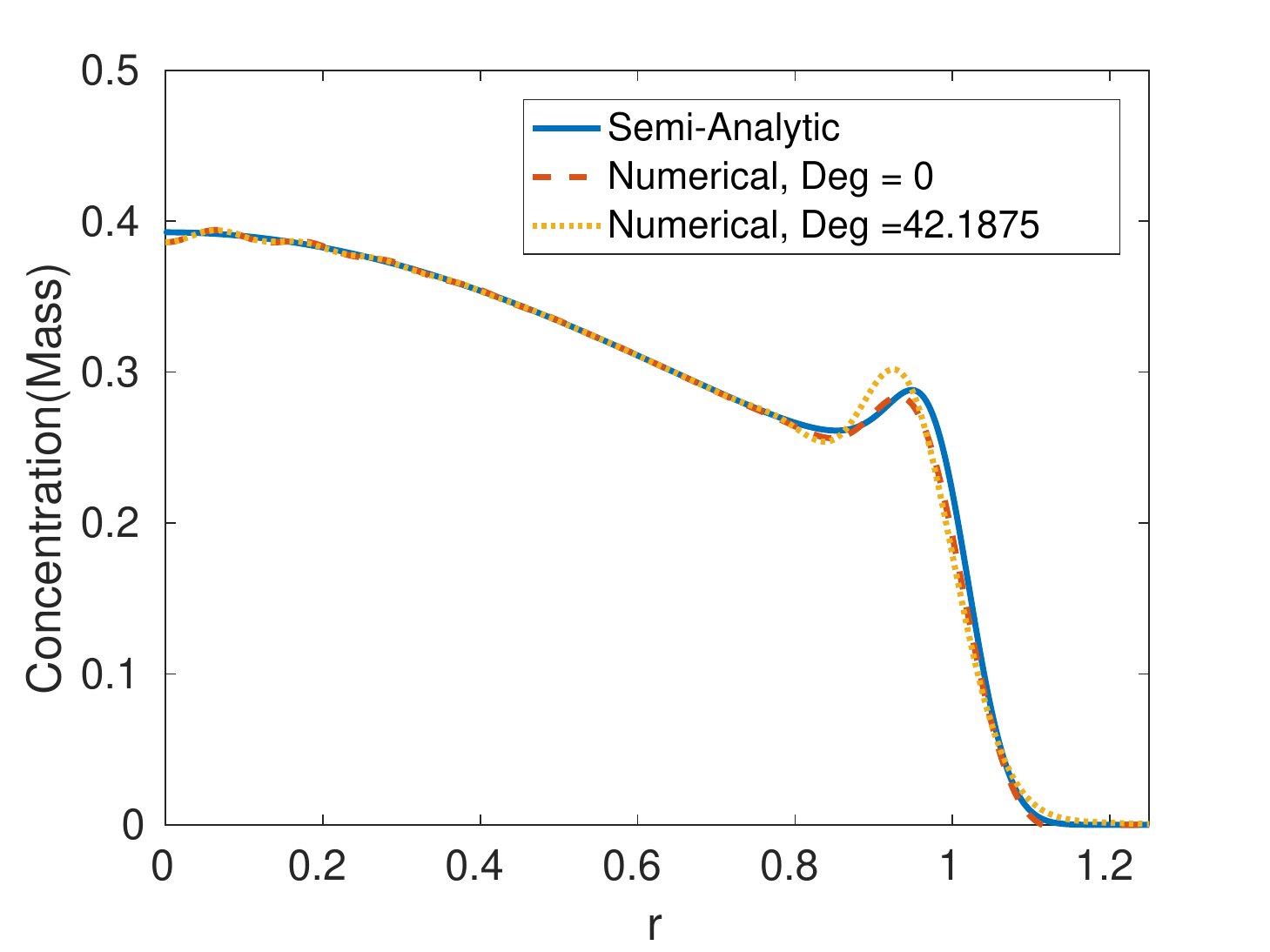}
\caption{FV-DG $\nsol{S}_{32}\nsol{S}_4$}
\end{subfigure}%
\begin{subfigure}[b]{0.33\textwidth}
\centering
\includegraphics[width=1\textwidth]{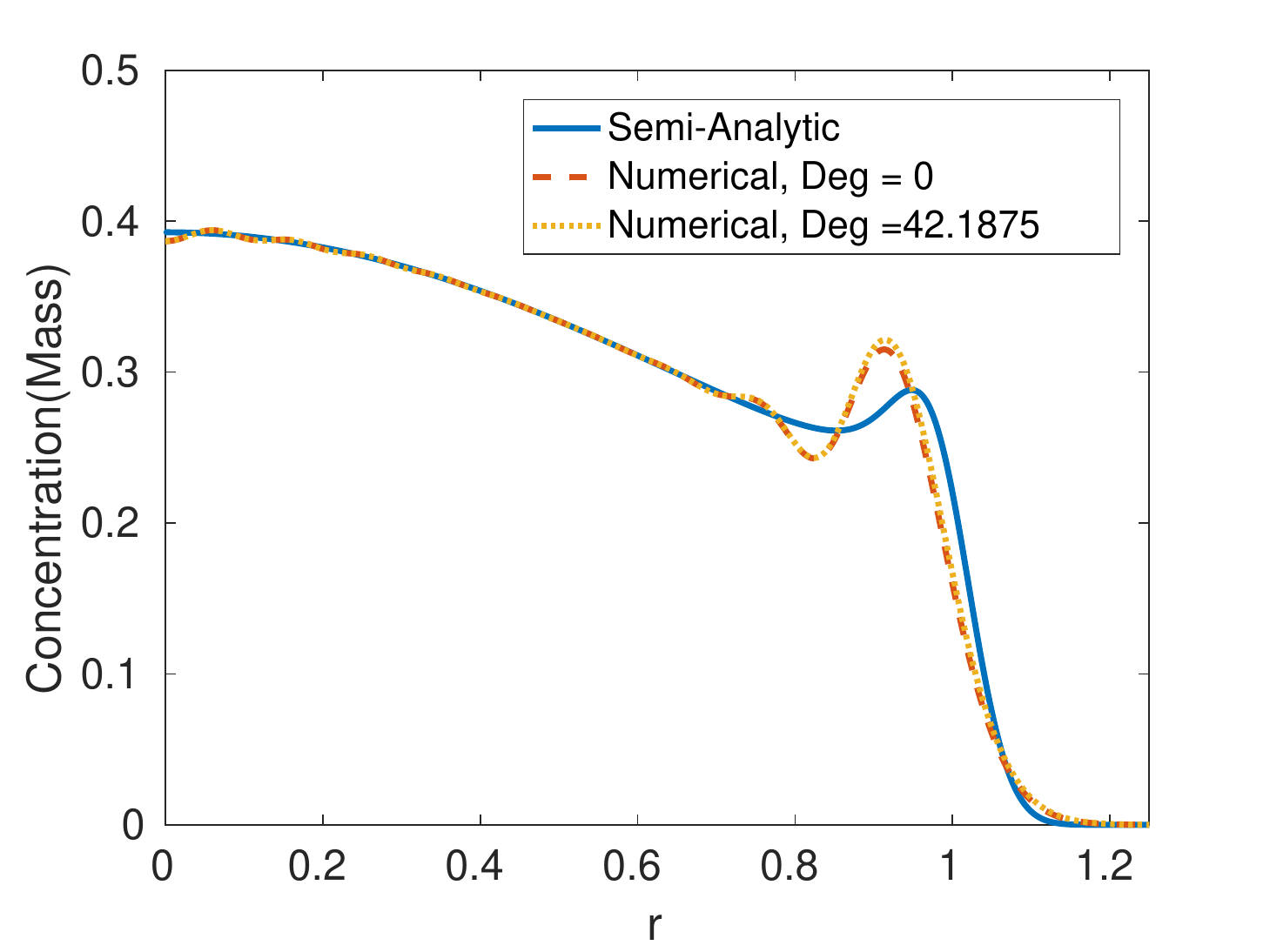}
\caption{DG-DG $\nsol{S}_{32}\nsol{S}_4$}
\end{subfigure}%
}
\caption{Line source plots of particle concentration. Simulations are run on a $301\times301$ grid to $t=1$ with $\Delta t = 3\Delta x$. Top row: two-dimensional heat maps.  Bottom row:  line-outs along the x-axis and along an ordinate direction. Solutions use initial condition \eqref{eq:LSinicond} with $\beta = 0.045$.}
\label{isofigure5}
\end{figure}

\begin{table}[tbhp]
\centering
\footnotesize
\begin{tabular}{|c|c|c|c|}
\hline
Method & DG $\nsol{S}_{16}$ & DG-DG $\nsol{S}_{32}\nsol{S}_4$ & FV-DG $\nsol{S}_{32}\nsol{S}_4$\\ \hline
Run time (mins)& 14.8 & 30.6 & 5.0 \\ \hline
$E_2$ & 0.18 & 0.067 & 0.031\\ \hline
$E_{\infty}$ & 0.46 & 0.18 & 0.11\\ \hline
\end{tabular}
\caption{Table of run times and errors (see \eqref{eq:err}) for numerical solution in Figure \ref{isofigure5}.}\label{table5}
\end{table}
\subsection{Lattice problem}
In this section, we use a more realistic example to demonstrate that the efficiency afforded by the spatial hybrid.
The lattice test was first proposed in \cite{Brunner} as a cartoon loosely based on a nuclear reactor core assembly.  The problem is a checkerboard of highly scattering and highly absorbing regions with vacuum boundaries as shown in Figure \ref{Latfigure1}.  The computational domain is a $7\times 7$ square divided into smaller squares with side length one. The middle square is an isotropic source, surrounded by a checkerboard of purely scattering and purely absorbing squares as shown in Figure \ref{Latfigure1}(a) with material parameters in Figure \ref{Latfigure1}(b).  The initial data is void and the boundary conditions are absorbing, i.e., $\psi_0 =0$ and $\psi_{\rm{b}}=0$.  
We simulate the problem with $504\times 504$ spatial grid on the  domain $[-3.5,3.5]\times[-3.5,3.5]$, set $\Delta t =  10\Delta x$, and run to a final time $t=2.8$.

The results in Table \ref{table6} show again that the angular hybrid is capable of producing a better solution by increasing the angular resolution in the uncollided equation and decreasing the angular  resolution in the collided equation. However, as in Section \ref{sec:AccThin}, the resulting improvement in accuracy comes at the cost of increased run time. The spatial hybrid is able to reduce the run time significantly while producing a solution with comparable errors.

\begin{figure}[bp!]
\centering
\makebox[\linewidth][c]{%
\captionsetup[subfigure]{justification=centering}
\centering
\begin{subfigure}[b]{0.33\textwidth}
\centering
\includegraphics[width=.6\textwidth]{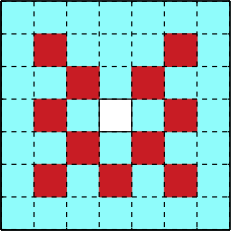}
\vspace{.22in}
\caption{Material Coefficients Location}
\end{subfigure}%

\begin{subfigure}[b]{0.33\textwidth}
\centering
\scalebox{.95}{
\begin{tabular}{l|r|r|r|}
 & $\sigma_{\rm{a}}$ & $\sigma_{\rm{t}}$ & $q$\\
 \hline
 red squares & 10 & 10 & 0\\
 \hline
 blue squares & 0 & 1 & 0\\
 \hline
 white square & 0 & 1 & 1\\
\hline
\end{tabular}}
\vspace{.3in}
\caption{Material properties}
\end{subfigure}

\begin{subfigure}[b]{0.33\textwidth}
\centering
\includegraphics[width=1\textwidth]{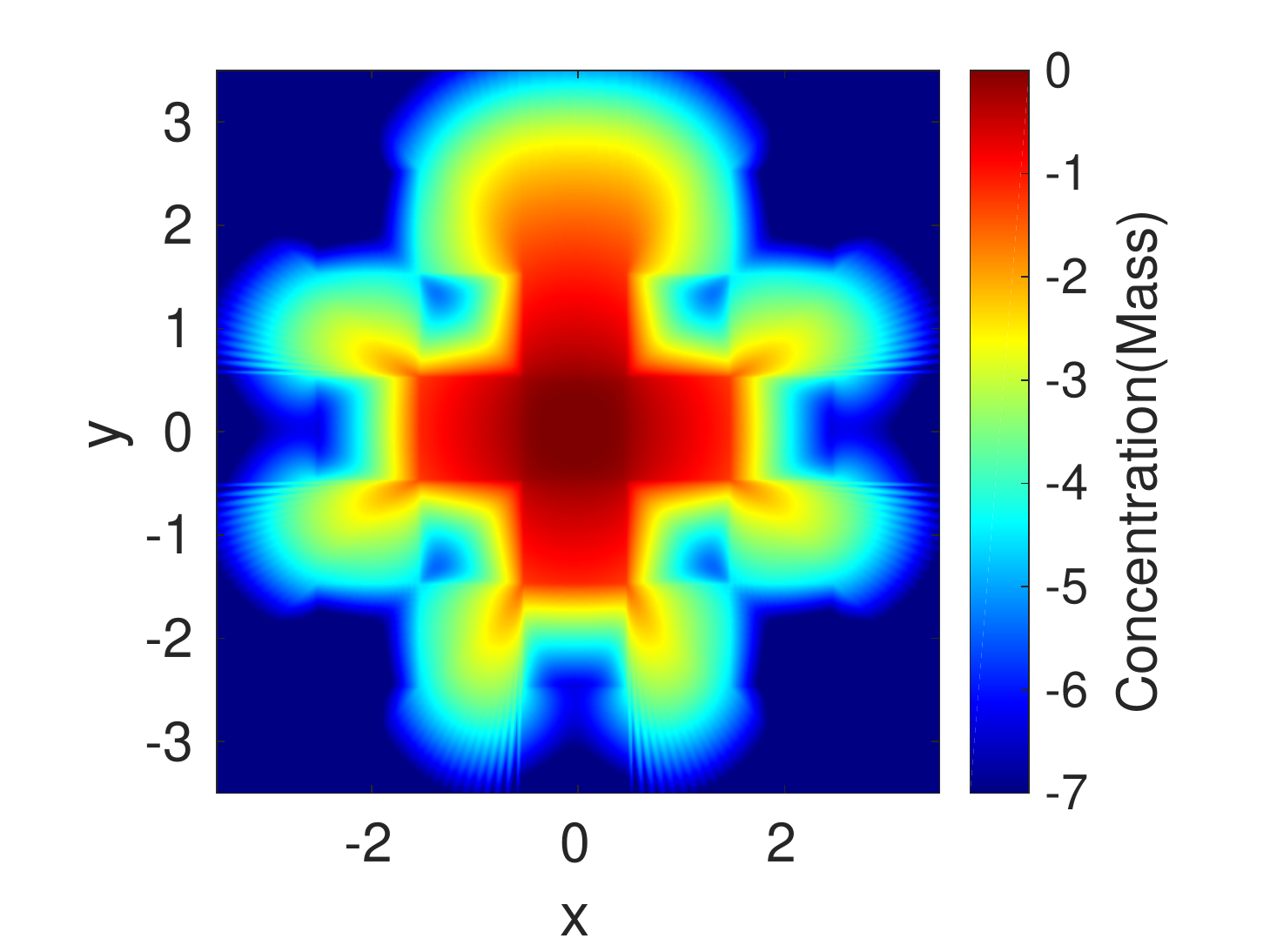}
\caption{DG $\nsol{S}_{64}$ numerical reference}
\end{subfigure}%

}\\
\makebox[\linewidth][c]{%
\captionsetup[subfigure]{justification=centering}
\centering
\begin{subfigure}[b]{0.33\textwidth}
\centering
\includegraphics[width=1\textwidth]{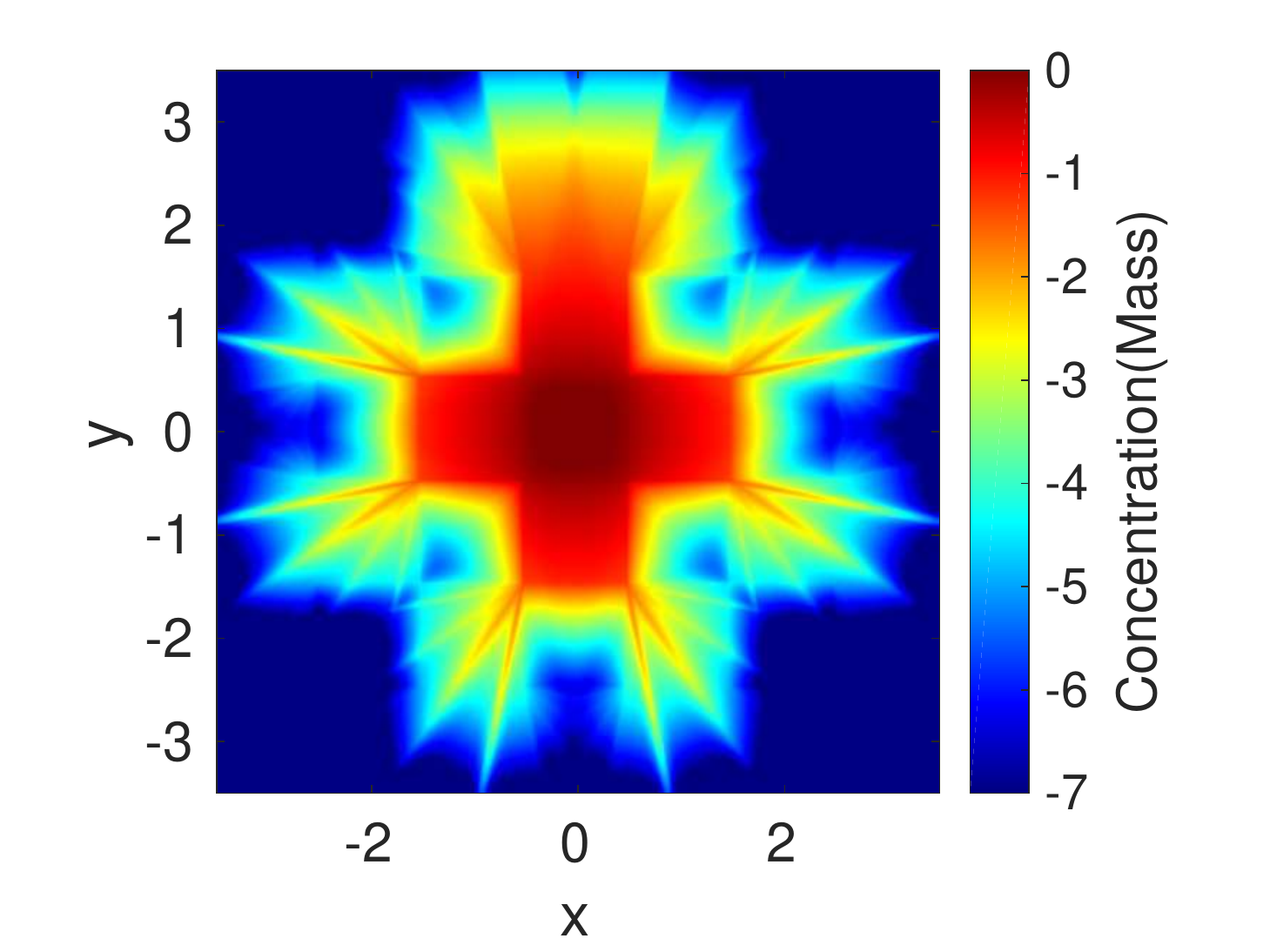}
\caption{DG $\nsol{S}_{8}$}
\end{subfigure}%
\begin{subfigure}[b]{0.33\textwidth}
\centering
\includegraphics[width=1\textwidth]{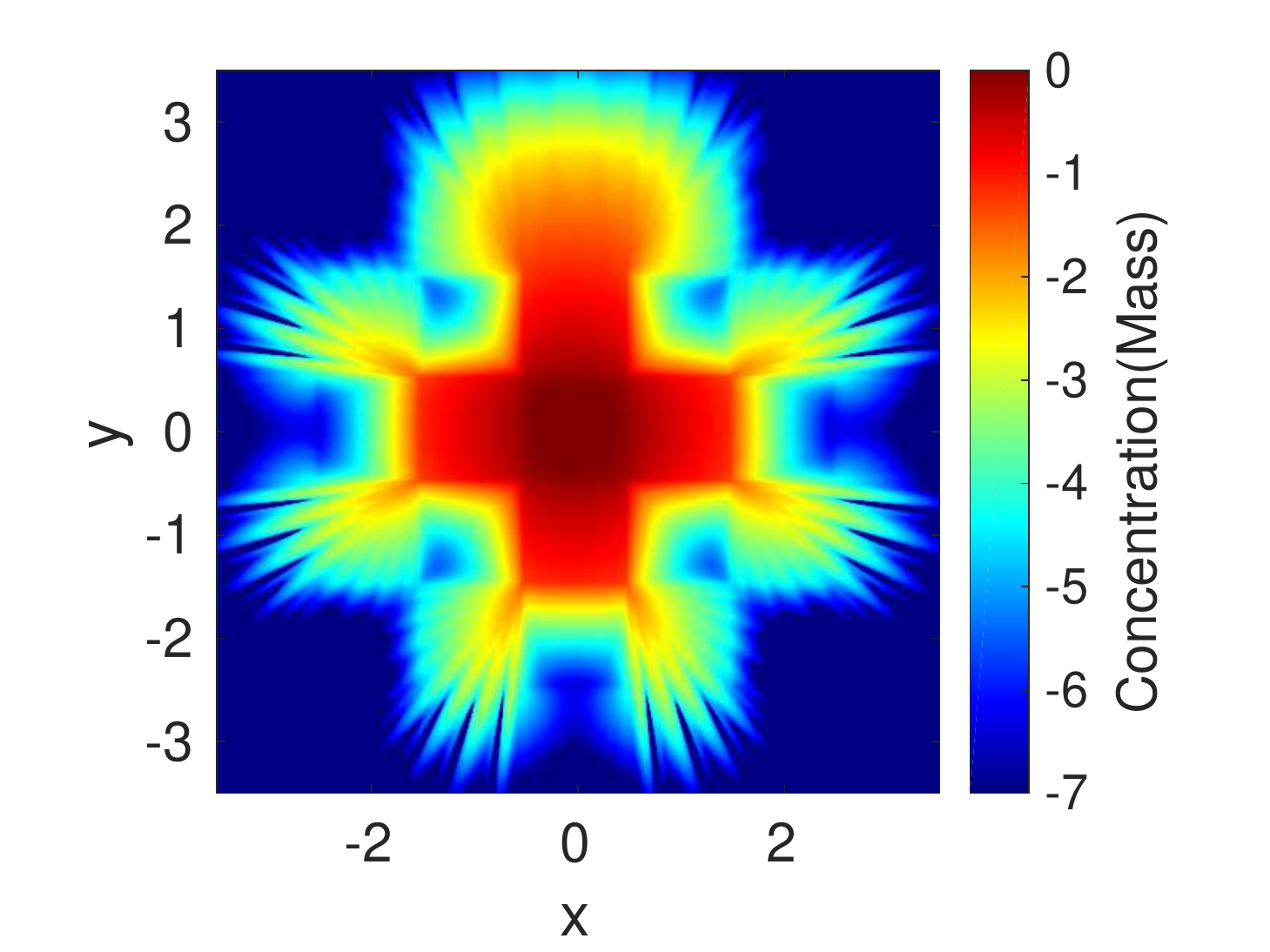}
\caption{FV-DG $\nsol{S}_{16}\nsol{S}_4$}
\end{subfigure}%
\begin{subfigure}[b]{0.33\textwidth}
\centering
\includegraphics[width=1\textwidth]{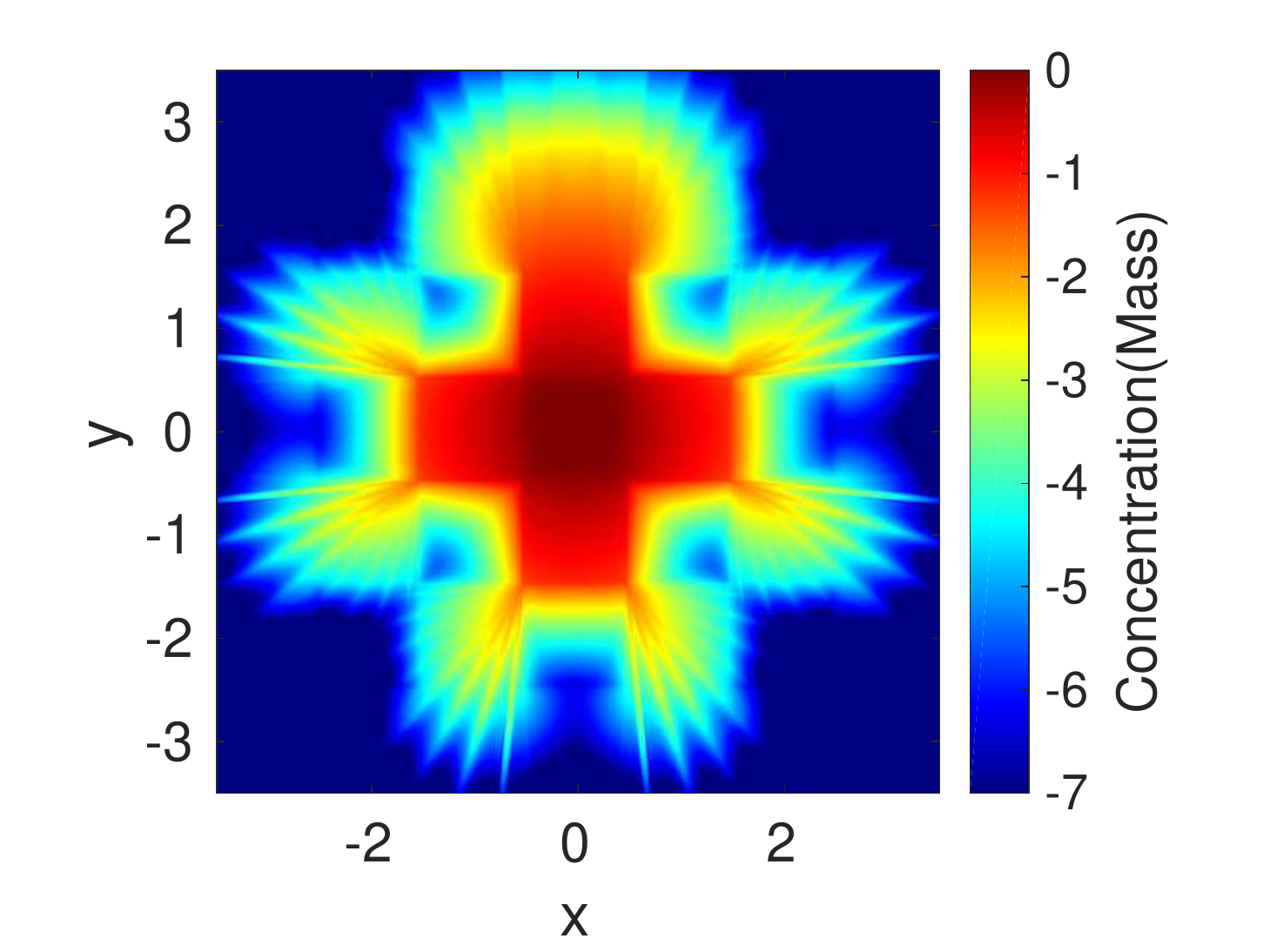}
\caption{DG-DG $\nsol{S}_{16}\nsol{S}_4$}
\end{subfigure}%
}\\
\caption{Lattice plots of particle concentration in a logarithmic scale. Simulations are run on a $504\times504$ grid to $t=2.8$ with $\Delta t = 10\Delta x$. Top row: material coefficient (a), with corresponding values (b), and DG numerical reference (c). Bottom row: numerical solutions.
}  
\label{Latfigure1}
\end{figure}



\begin{table}[bp!]
\centering
\footnotesize
\begin{tabular}{|c|c|c|c|c|}
\hline
Method & Reference & DG $\nsol{S}_{8}$ & DG-DG $\nsol{S}_{16}\nsol{S}_4$ & FV-DG $\nsol{S}_{16}\nsol{S}_4$ \\ \hline
Run time (mins) & 569.1 & 7.5 & 14.5 & 4.2 \\ \hline
$L^2$ Error & - & 0.0094 & 0.0029 & 0.0032 \\ \hline
$L^{\infty}$ Error & - & 0.015 & 0.0062 & 0.011 \\ \hline
\end{tabular}
\caption{Table of run times and errors (see \eqref{eq:err}) for numerical solutions in Figure \ref{Latfigure1}.}\label{table6}
\end{table}


%

\section{Conclusion}\label{sec:Conclude}
We have presented a hybrid spatial discretization of the radiation transport equation (RTE) based on the formulation introduced in \cite{Coryhybrid}.  The method relies on the separation of the RTE into two components, one of which is discretized in space with a finite volume (FV) method, which uses less memory, and the other with a discontinuous Galerkin (DG) method, which performs better in scattering dominated regimes.   The spatial discretization is then combined with a DIRK method for time integration and a discrete ordinate discretization in angle.

Following the approach in \cite{Guermond}, we show that, like standard DG, the hybrid spatial discretization converges, in the limit of infinite scattering, to a  consistent discretization of the diffusion limit \eqref{eq:difflimeq}.   We also demonstrate  the hybrid approach is more efficient, in terms of memory usage and computational time than a uniform DG discretization.  The formulation in \cite{Coryhybrid} allows for hybridization in both space and angle, and we show how a combination of the two can improve the efficiency of simulations for two benchmark problems.

\begin{appendices}
%
%
\section{Computational scaling}\label{App1}

In this appendix, we explain the details behind the numbers in Table \ref{Table 0} and the charts in Figure \ref{Bar2}.  All computations rely on four main subroutines:  \texttt{source}, which implements \eqref{eq:setsource}; \texttt{integrate}, which implements \eqref{eq:integrate}; \texttt{copy}, which implements  \eqref{eq:copy}; and \texttt{sweep}, which implements the inversion of the operator in \eqref{eq:sweep} for a given source. 
In the \texttt{source} subroutine, a known source function is used to compute a coefficient for every unknown in a mesh cell and every angle. In the \texttt{integrate} subroutine, the angular unknowns associated to each spatial unknown are mapped to a single value.
 In the \texttt{copy} routine, a single value of each spatial unknown is copied across all angles. The \texttt{sweep} routine solves a linear system (which in part requires a matrix inversion of a size equal to the number of spatial unknowns for every angle and mesh cell). When the cross-sections are constant, which we assume for the experiment in Section \ref{sec:Eff}, the matrix used in the inversion can be pre-factored. The result is that the usual $O(n^3)$ operation count for an $n \times n$ matrix is reduced to $O(n^2)$, where $n$ is the number of unknowns.

 The cost of each of the subroutines above depends on the number of angles, number of mesh cells, and number of unknowns per mesh cell.  In standard DG or FV codes, we use $N^*$ angles and $M$ cells. In hybrid DG-DG and FV-DG we use $N_{\rm u}^*$ and $N_{\rm c}^*$  points for the uncollided and collided equations, respectively, on $M$ cells. FV methods will use one unknown per angle per cell for both quadrilateral  and triangular cells, and DG methods will use $2^d$ unknowns for quadrilateral cells and $(d+1)$ unknowns for triangular cells where $d$ is the dimension of the spatial domain.  With these values the  number of flops for each subroutine is given in Tables \ref{App Table 0} and \ref{App Table 1}.  The results in Table \ref{Table 0} are obtained by summing across each row in Tables \ref{App Table 0} or \ref{App Table 1}.

\begin{table}[tbhp]
\centering
\footnotesize
\begin{tabular}{|c|c|c|c|c|}
\hline
{}& \texttt{source} & \texttt{integrate} & \texttt{copy} & \texttt{sweep}\\
\hline
FV & $N^*$ & $N^*$ & $N^*$ & $N^*$\\
\hline
DG & $2^d N^*$ & $2^d N^*$ & $2^d N^*$ & $2^{2d} N^*$\\
\hline
DG-DG & $2^dN_{\rm u}^*$ & $2^d(N_{\rm u}^* + N_{\rm c}^*)$ & $2^d(N_{\rm u}^* + N_{\rm c}^*)$ & $2^{2d}(N_{\rm u}^* + N_{\rm c}^*)$\\
\hline
FV-DG & $N_{\rm u}^*$ & $N_{\rm u}^* + 2^dN_{\rm c}^*$ & $N_{\rm u}^* + 2^dN_{\rm c}^*$ & $N_{\rm u}^* + 2^{2d}N_{\rm c}^*$\\
\hline
\end{tabular} 
\caption{Computational scaling leading orders per rectangular element.}\label{App Table 0}
\end{table}

\begin{table}[tbhp]
\centering
\footnotesize
\begin{tabular}{|c|c|c|c|c|}
\hline
{}& \texttt{source} & \texttt{integrate} & \texttt{copy} & \texttt{sweep}\\
\hline
FV & $N^*$ & $N^*$ & $N^*$ & $N^*$\\
\hline
DG & $(d+1)N^*$ & $(d+1)N^*$ & $(d+1)N^*$ & $(d+1)^2 N^*$\\
\hline
DG-DG & $(d+1)N_{\rm u}^*$ & $(d+1)(N_{\rm u}^* + N_{\rm c}^*)$ & $(d+1)(N_{\rm u}^* + N_{\rm c}^*)$ & $(d+1)^2(N_{\rm u}^* + N_{\rm c}^*)$\\
\hline
FV-DG & $N_{\rm u}^*$ & $N_{\rm u}^* + (d+1)N_{\rm c}^*$ & $N_{\rm u}^* + (d+1)N_{\rm c}^*$ & $N_{\rm u}^* + (d+1)^2N_{\rm c}^*$\\
\hline
\end{tabular} 
\caption{Computational scaling leading orders per triangular element.}\label{App Table 1}
\end{table}

To generate the predictions in Figure \ref{Bar2}, we use the leading orders in Table \ref{App Table 0} along with knowledge of how many times each subroutine is called within a program.
Let $T_{\rm ref}$ be the minutes it takes to compute the standard DG reference, and let $\texttt n_{\rm so}$, $\texttt n_{\rm int}$, $\texttt n_{\rm cp}$ and ${\texttt n_{\rm sw}}$ be the total occurrences of the \texttt{source}, \texttt{integrate}, \texttt{copy}, and \texttt{sweep} subroutines respectively in the reference simulation. These are acquired by knowing either how many times these subroutines are performed in the code per iteration of the iterative solver (GMRES in our case) or per time step. We assume the total number of time steps and iterations of the iterative solver are known (the later based on calibration with the DG $\nsol{S}_4$ simulation). 
Then $T_{\rm ref} = (2^d{\texttt n_{\rm so}} + 2^d{\texttt n_{\rm int}} + 2^d{\texttt n_{\rm cp}} +2^{2d} {\texttt n_{\rm sw}})kN^*M$, where $k$ is an unknown conversion constant that is assumed to be independent of the type of method used.  This implies that
\begin{equation}
kN^*M = \frac{T_{\rm ref}}{(2^d{\texttt n_{\rm so}}+2^d{\texttt n_{\rm int}} + 2^d{\texttt n_{\rm cp}} +2^{2d} {\texttt n_{\rm sw}})},
\end{equation}
 Let $\texttt n_{\rm int,m}$, $\texttt n_{\rm cp,m}$ and ${\texttt n_{\rm sw,m}}$ be the total occurrences of the \texttt{integrate}, \texttt{copy}, and \texttt{sweep} subroutines respectively with a loop structure involving $N_{\rm m}^*$ angles for $\rm m\in\{u,c\}$. With the constant $k$ determined, we assume the total number of time steps and iterations of the iterative solver in the other simulations are the same as in the reference simulation. The predicted times for the other three methods in Figure \ref{Bar2} are calculated as follows:
\begin{align}
T_{\text FV} & = ({\texttt n_{\rm so}} + {\texttt n_{\rm int}} + {\texttt n_{\rm cp}} + {\texttt n_{\rm sw}})kN^*M\\
T_{\text {DG-DG}} & = \left(({\texttt n_{\rm so}} + {\texttt n_{\rm int,u}} + {\texttt n_{\rm cp,u}} +2^{d} {\texttt n_{\rm sw,u}})N_{\rm u}^* + ({\texttt n_{\rm int,c}} + {\texttt n_{\rm cp,c}} +2^d{\texttt n_{\rm sw,c}})N_{\rm c}^*\right)2^dkM\\
T_{\text {FV-DG}} & = \left(({\texttt n_{\rm so}} + {\texttt n_{\rm int,u}} + {\texttt n_{\rm cp,u}} + {\texttt n_{\rm sw,u}})N_{\rm u}^* + 2^d({\texttt n_{\rm int,c}} + {\texttt n_{\rm cp,c}} +2^{d} {\texttt n_{\rm sw,c}})N_{\rm c}^*\right)kM,
\end{align}
where $N_{\rm u}^*$ and $N_{\rm c}^*$ are known. 
For the simulations in Figure \ref{Bar2}, the values for the number of times each subroutine is performed is as follows.
\begin{table}[tbhp]
\footnotesize
\centering
\begin{tabular}{|c|c|c|c|c|c|c|c|c|c|}
\hline
${\texttt n_{\rm so}}$ & ${\texttt n_{\rm int}}$& ${\texttt n_{\rm cp}}$ & ${\texttt n_{\rm sw}}$ & ${\texttt n_{\rm int,u}}$ & ${\texttt n_{\rm cp,u}}$ & ${\texttt n_{\rm sw,u}}$ & ${\texttt n_{\rm int,c}}$ & ${\texttt n_{\rm cp,c}}$ & ${\texttt n_{\rm sw,c}}$\\
\hline
42 & 124 & 124 & 166 & 42 & 42 & 84 & 124 & 124 & 124\\
\hline
\end{tabular} 
\caption{Number of occurrences of subroutines used to compute solutions of the simulations in Figure \ref{Bar2}.}\label{App Table 2}
\end{table}

The memory predictions are easier to compute. Using the DG $\nsol{S}_4$ simulation as a reference, we measure the maximum memory expenditure during the run of the simulation. We assume that the majority of the memory expenditure is taken up by the largest arrays and Krylov vectors created by the GMRES solver in the code,  
and we know ahead of time how many arrays or vectors are needed to run the simulation. 
These arrays are either used to hold various portions of the solution at each time step or are temporary arrays used in the GMRES solver. These arrays either scale with the number of angles and mesh cells or scale with just the mesh cells. The Krylov vectors are formed as part of the Krylov subspace used by the GMRES solver to solve \eqref{eq:hyFVDGmatvecBKrylov} and only scale with the number of mesh cells.
During the run of the code, we require 4 arrays of size $2^dN^*M$ and 4 arrays of size $2^dM$ to hold various forms of the solution and source at every time step. During the GMRES solver step a number of temporary arrays are created, one of size $2^dN^*M$ and 2 of size $2^dM$. Additionally as part of building the Krylov subspace, the $k$-th iteration of the GMRES solver 
requires $k+1$ vectors of size $2^dM$.
The maximum iterations the solver took was 2 throughout all of the runs whose memory usage is shown in Figure \ref{Bar2}
, so the code required an additional 3 vectors. The codes used eight bytes of memory ($7.63\times10^{-6}$ MB) for every entry in an array or Krylov vector and the product quadrature has $N^*=N^2$ ordinates. The computational domain uses $301\times301=M$ mesh cells and $d=2$. This leads to the following:
\begin{equation}
267.8 \text{ MB}= \left(5*2^dN^*M + 9*2^dM\right)*7.63\times10^{-6}\text{ MB} + x \implies x= 21.7\text{ MB}.
\end{equation}
This $x$ value is attributed to the overhead of the code and various other values that are held in memory that does not scale with $M$ or $N^*$. We assume that this value of $x$ is constant across simulations shown in Figure \ref{Bar2}.
To predict the other values in Figure \ref{Bar2} we simply count all the total entries from all relevant arrays or vectors, multiply by 8 bytes, ($7.63*10^{-6}$ MB), and then add $x$. By relevant, we refer to arrays or vectors that scale in angles and mesh cells or just the mesh cells. The number of relevant arrays and vectors for each method is shown in Table \ref{App Table 3}, and for simplicity are all referred to as vectors.
\begin{table}[H]
\begin{tabular}{|c|c|}
\hline
Method & Relevant vectors\\
\hline
DG & 5 vectors of size $2^dN^*M$, 9 vectors of size $2^dM$\\
\hline
FV & 5 vectors of size $N^*M$, 9 vectors of size $M$\\
\hline
DG-DG & 4 vectors of size $2^dN_{\rm u}^*M$, 1 vector of size $2^dN_{\rm c}^*M$, 12 vectors of size $2^dM$\\
\hline
FV-DG & 4 vectors of size $N_{\rm u}^*M$, 1 vector of size $2^dN_{\rm c}^*M$, 12 vectors of size $2^dM$\\
\hline
\end{tabular}
\caption{Number and type of relevant vectors in each method.}
\label{App Table 3}
\end{table}
\section{Second-order finite volume reconstruction}\label{App2}
In this section, we specify the form of the reconstruction operator $\SRfull$ that is used for the calculations of Section \ref{sec:HyImp}; see \eqref{eq:FVbilin}.  As in Section \ref{sec:HyImp}, we assume a two-dimensional geometry using quadrilateral elements.    Let $\cT_h$ be a partition of $X\subset \bbR^2$ into $J^*\times K^*$ mesh cells. Let $C_{j,k}\in \cT_h$ be a quadrilateral with cell center $(x_j,y_k)$  and cell size $\Delta x\Delta y$ for all $1\le j\le J^*$, $1\le k\le K^*$. Let $N^*\in\bbN$ and let $\{\Omega_i\}_{i=1}^{N^*}\subset \sph$ where $\Omega_i:=(\Omega_{i,x},\Omega_{i,y},\Omega_{i,z})$. Let $f=[f_1,f_2,\ldots,f_{N^*}]^T\in (\cX_{h,0})^{N^*}$, where $f_i\in \cX_{h,0}$ for all $1\le i\le N^*$.  Denote the value of $f_i$ on cell $C_{j,k}$ as $f_{i,j,k}$. Then
\begin{equation}
\left.\left(\SRfull f\right)_i(x,y)\right|_{C_{j,k}}=f_{i,j,k} + s^x_{i,j,k}(x-x_j) + s^y_{i,j,k}(y-y_k),\ \forall\  (x,y)\in C_{j,k},
\end{equation}
where
\begin{equation}
s^x_{i,j,k}=\begin{cases}
\frac{f_{i,j,k}-f_{i,j-1,k}}{\Delta x}, & \Omega_{i,x}\ge 0,\ j > 1,\\
\frac{2(f_{i,j,k}-f_{i,j-1/2,k})}{\Delta x}, & \Omega_{i,x}\ge0,\ j = 1,\\
\frac{f_{i,j,k}-f_{i,j+1,k}}{\Delta x},& \Omega_{i,x}\le 0,\ j < J^*,\\
\frac{2(f_{i,j,k}-f_{i,j+1/2,k})}{\Delta x}, & \Omega_{i,x}\le 0,\ j = J^*,
\end{cases}\quad
\end{equation}
$s^y_{i,j,k}$ is defined similarly, the boundary terms are
\begin{subequations}
\begin{align}
f_{i,1/2,k}&=f_{{\rm b},i}(x_{1/2},y_k),\quad f_{i,J^*+1/2,k} = f_{{\rm b},i}(x_{J^*+1/2},y_k),\\ f_{i,j,1/2}&= f_{{\rm b},i}(x_{j},y_{1/2}),\quad f_{i,j,K^*+1/2} = f_{{\rm b},i}(x_{j},y_{K^*+1/2}),
\end{align}
\end{subequations} 
and we assume that the point-wise values of  $f_{{\rm b},i}$,  $1\le i\le N^*$, above are well-defined.

\end{appendices}

\bibliographystyle{siamplain}
\bibliography{./main}

\begin{thebibliography}{10}

\bibitem{adams1991even}
{\sc M.~L. Adams}, {\em Even-parity finite-element transport methods in the
  diffusion limit}, Progress in Nuclear Energy, 25 (1991), pp.~159--198.

\bibitem{adams}
{\sc M.~L. Adams}, {\em Discontinuous finite element transport solutions in
  thick diffusive problems}, Nuclear Science and Engineering, 137 (2001),
  pp.~298--333.

\bibitem{adams2002fast}
{\sc M.~L. Adams and E.~W. Larsen}, {\em Fast iterative methods for
  discrete-ordinates particle transport calculations}, Progress in nuclear
  energy, 40 (2002), pp.~3--159.

\bibitem{Alcouffe}
{\sc R.~E. Alcouffe}, {\em A first collision source method for coupling {M}onte
  {C}arlo and discrete ordinates for localized source problems}, in Monte-Carlo
  Methods and Applications in Neutronics, Photonics and Statistical Physics,
  Springer, 1985, pp.~352--366.

\bibitem{DIRK}
{\sc R.~Alexander}, {\em Diagonally implicit {R}unge-{K}utta method for stiff
  o.d.e.'s}, SIAM Journal on Numerical Analysis, 14 (1977), pp.~1006--1021.

\bibitem{Atk}
{\sc K.~Atkinson}, {\em Numerical intergration on the sphere}, H. Austral.
  Math. Soc., 23 (1982), pp.~332--347.

\bibitem{Brunner}
{\sc T.~Brunner and J.~Holloway}, {\em Two-dimensional time-dependent {R}iemann
  solvers for neutron transport}, Journal of Computational Physics, 210 (2005),
  pp.~386--399.

\bibitem{Buras}
{\sc R.~Buras, M.~Rampp, H.-T. Janka, and K.~Kifonidis}, {\em Two-dimensional
  hydrodynamic core-collapse supernova simulations with spectral neutrino
  transport}, Astronomy and Astrophysics, 447 (2005), pp.~1049--1092.

\bibitem{Case-Zweifel-1967}
{\sc K.~Case and P.~Zweifel}, {\em Linear Transport Theory}, Addison-Wesley,
  Reading, MA, 1967.

\bibitem{Chandra}
{\sc S.~Chandrasekhar}, {\em Radiative Transfer}, Dover Publications, Inc.,
  1960.

\bibitem{CrockattDis}
{\sc M.~Crockatt}, {\em Hybrid methods for radiation transport using integral
  deferred correction}, 2018.

\bibitem{Crockatt}
{\sc M.~Crockatt, A.~Christlieb, C.~K. Garrett, and C.~Hauck}, {\em An
  arbitrary-order, fully implicit, hybrid kinetic solver for linear radiative
  transport using integral deferred correction}, Journal of Computational
  Physics, 346 (2017), pp.~212--241.

\bibitem{Lions}
{\sc R.~Dautray and J.-L. Lions}, {\em Mathematical Analysis and Numerical
  Methods for Science and Technology}, vol.~6, Springer-Verlag, Berlin, 1984.

\bibitem{egger2012mixed}
{\sc H.~Egger and M.~Schlottbom}, {\em A mixed variational framework for the
  radiative transfer equation}, Mathematical Models and Methods in Applied
  Sciences, 22 (2012), p.~1150014.

\bibitem{Ganapol}
{\sc B.~Ganapol, R.~Baker, J.~Dahl, and R.~Alcouffe}, {\em Homogeneous infinite
  media time-dependent analytical benchmarks}, tech. report, Los Alamos
  National Laboratory, 2001.

\bibitem{Garrett}
{\sc C.~K. Garrett and C.~Hauck}, {\em A comparison of moment closures for
  linear kinetic transport equations: the line source benchmark}, Transport
  Theory and Statistical Physics, 42 (2014), pp.~203--235.

\bibitem{Gosse2004}
{\sc L.~Gosse and G.~Toscani}, {\em Asymptotic-preserving {\&} well-balanced
  schemes for radiative transfer and the rosseland approximation}, Numerische
  Mathematik, 98 (2004), pp.~223--250,
  \url{https://doi.org/10.1007/s00211-004-0533-x},
  \url{https://doi.org/10.1007/s00211-004-0533-x}.

\bibitem{Guermond}
{\sc J.-L. Guermond and G.~Kanschat}, {\em Asymptotic analysis of upwind
  discontinuous {G}alerkin approximation of the radiative transport equation in
  the diffusion limit}, SIAM Journal of Numerical Analysis, 48 (2010),
  pp.~53--78.

\bibitem{Habetler}
{\sc G.~J. Habetler and B.~J. Matkowsky}, {\em Uniform asymptotic expansions in
  transport theory with small mean free paths, and the diffusion
  approximation}, Journal of Mathematical Physics, 16 (1975), p.~846.

\bibitem{Han}
{\sc W.~Han, J.~Huang, and J.~A. Eichholz}, {\em Discrete-ordinate
  discontinuous {G}alerkin methods for solving the radiative transfer
  equation}, SIAM Journal of Scientific Computing, 32 (2010),
  \url{http://epubs.siam.org/doi/abs/10.1137/090767340}.

\bibitem{hauck2009temporal}
{\sc C.~D. Hauck and R.~B. Lowrie}, {\em Temporal regularization of the ${P_N}$
  equations}, Multiscale Modeling \& Simulation, 7 (2009), pp.~1497--1524.

\bibitem{jin}
{\sc S.~Jin}, {\em Asymptotic preserving ({AP}) schemes for multiscale kinetic
  and hyperbolic equations: a review}, Lecture Notes for Summer School on
  Methods and Models of Kinetic Theory(M\&MKT), Porto Ercole (Grosseto, Italy),
   (2010), pp.~177--216.

\bibitem{jin1996}
{\sc S.~Jin and C.~D. Levermore}, {\em Numerical schemes for hyperbolic
  conservation laws with stiff relaxation terms}, Journal of computational
  physics, 126 (1996), pp.~449--467.

\bibitem{jin2000uniformly}
{\sc S.~Jin, L.~Pareschi, and G.~Toscani}, {\em Uniformly accurate diffusive
  relaxation schemes for multiscale transport equations}, SIAM Journal on
  Numerical Analysis, 38 (2000), pp.~913--936.

\bibitem{LarsenKeller}
{\sc E.~W. Larsen and J.~B. Keller}, {\em Asymptotic solution of neutron
  transport problems for small mean free paths}, Journal of Mathematical
  Physics, 15 (1974), p.~75.

\bibitem{MorelLarsen}
{\sc E.~W. Larsen and J.~E. Morel}, {\em Asymptotic solutions of numerical
  transport problems in opticaly thick, diffusive regimes {II}}, Journal of
  Computational Physics, 83 (1989), pp.~212--236.

\bibitem{Larsen}
{\sc E.~W. Larsen and J.~E. Morel}, {\em Nuclear Computional Science},
  Springer, 2010, ch.~Advances in discrete-ordinates methodology.

\bibitem{MorelLarsen2}
{\sc E.~W. Larsen, J.~E. Morel, and W.~F. Miller~Jr}, {\em Asymptotic solutions
  of numerical transport problems in opticaly thick, diffusive regimes},
  Journal of Computational Physics, 69 (1987), pp.~283--324.

\bibitem{lewis2010second}
{\sc E.~E. Lewis}, {\em Second-order neutron transport methods}, in Nuclear
  Computational Science, Springer, 2010, pp.~85--115.

\bibitem{Lewis}
{\sc E.~E. Lewis and J.~W.F.~Miller}, {\em Computational Methods of Neutron
  Transport}, American Nuclear Society, La Grange Park, IL, 1993.

\bibitem{mcclarren2008}
{\sc R.~G. McClarren, T.~M. Evans, R.~B. Lowrie, and J.~D. Densmore}, {\em
  Semi-implicit time integration for {$P_N$} thermal radiative transfer},
  Journal of Computational Physics, 227 (2008), pp.~7561--7586.

\bibitem{Coryhybrid}
{\sc R.~G. McClarren and C.~D. Hauck}, {\em A collision-based hybrid method for
  time-dependent, linear, kinetic transport equations}, Multiscale Modeling and
  Simulation, 11 (2013), p.~1197–1227.

\bibitem{Mezzacappa}
{\sc A.~Mezzacappa, A.~Calder, S.~Bruenn, J.~Blondin, M.~Guidry, M.~Strayer,
  and A.~Umar}, {\em An investigation of neutrino-driven convection and the
  core collapse supernova mechanism using multigroup neutrino transport}, The
  Astrophysical Journal, 495 (1998), pp.~911--926.

\bibitem{Mihalis-Mihalis-1999}
{\sc D.~Mihalis and B.~Weibel-Mihalis}, {\em Foundations of Radiation
  Hydrodynamics}, Dover, Mineola, New York, 1999.

\bibitem{Pomraning-1973}
{\sc G.~C. Pomraning}, {\em Radiation Hydrodynamics}, Pergamon Press, New York,
  1973.

\bibitem{ReedHill}
{\sc W.~H. Reed and T.~Hill}, {\em Triangular mesh methods for the neutron
  transport equation}, tech. report, Los Alamos Scientific Lab., N. Mex.(USA),
  1973.

\bibitem{memusg}
{\sc J.~Shin}, {\em memusg}, \url{https://gist.github.com/netj/526585}.

\bibitem{ZHENGMING1993673}
{\sc L.~Zheng-Ming and A.~Brahme}, {\em An overview of the transport theory of
  charged particles}, Radiation Physics and Chemistry, 41 (1993), pp.~673 --
  703.

\end{thebibliography}
\end{document}